\documentclass[reqno,11pt]{amsart}
\usepackage[dvipsnames,usenames]{color} 

\usepackage{enumitem}     

\usepackage{pifont}
\usepackage{upgreek}
\usepackage{calligra}
\usepackage{hyperref,caption}
\usepackage{ifpdf} 
\usepackage[mathscr]{euscript} 
\usepackage{cite} 
\usepackage{graphicx}  
\usepackage[all]{xy} \xyoption{arc} \xyoption{color}

\usepackage{amsmath} 
\usepackage{amsthm} 
\usepackage{amsfonts}  
\usepackage{amssymb} 

\usepackage{verbatim}

\numberwithin{equation}{section} 

\newtheorem{corollary}{\sc Corollary}[section]
\newtheorem{theorem}{\sc Theorem}[section]
\newtheorem{lemma}{\sc Lemma}[section]
\newtheorem{proposition}{\sc Proposition}[section]
\newtheorem{remark}{\sc Remark}
\newtheorem{definition}{\sc Definition}[section]

\DeclareMathAlphabet{\mathcalligra}{T1}{calligra}{m}{n}
\DeclareFontShape{T1}{calligra}{m}{n}{<->s*[2.2]callig15}{}

\newcommand{\norm}{\mathscr{S}}

\newcommand{\D}{\nabla}
\newcommand{\ITIMESMD}{\mathbf{D}}
\newcommand{\MD}{D}
\newcommand{\AngMD}{\bar{\partial}}

\newcommand{\Enth}{S}

\newcommand{\Cof}{a}

\newcommand{\Lagdiff}{{^{(\upeta)}\nabla}}

\newcommand{\Lagdiv}{{^{(\upeta)} \mkern-2mu \mbox{\upshape div}}}

\newcommand{\Lagvort}{{^{(\upeta)} \mkern-2mu \mbox{\upshape vort}}}

\newcommand{\Flatdiv}{{^{(3)} \mkern-2mu \mbox{\upshape div}}}
\newcommand{\Flatvort}{{^{(3)} \mkern-2mu \mbox{\upshape vort}}}

\newcommand{\LM}{L^2(\mathfrak M_\tau)}
\newcommand{\LI}{L^\infty(\mathfrak M_\tau)}

\def\be{\begin{equation}}
\def\ee{\end{equation}}
\def\g{\partial}
\def\t{\bar{\partial}}
\def\A{\mathscr{A}}

\newcommand{\R}{\mathbb R}

\begin{document}
\title[Relativistic Euler with moving vacuum boundary]{A priori estimates for solutions to the relativistic Euler equations with a moving vacuum boundary}
\author{Mahir Had\v zi\'c, Steve Shkoller, Jared Speck}

\address{Department of Mathematics, King's College London, London, UK}
\address{Department of Mathematics, University of California at Davis, Davis, USA}
\address{Department of Mathematics, Massachusetts Institute of Technology, Cambridge, USA}


\begin{abstract}
We study the relativistic Euler equations on the Minkowski spacetime background.
We make assumptions on the equation of state and the initial data that are
relativistic analogs of the well-known physical vacuum boundary condition,
which has played an important role in prior work on
the non-relativistic compressible Euler equations.
Our main result is the derivation, 
relative to Lagrangian (also known as co-moving) coordinates,
of local-in-time a priori estimates for the 
solution. The solution features a fluid-vacuum boundary,
transported by the fluid four-velocity,
along which the hyperbolicity of the equations degenerates.
In this context, the relativistic Euler equations
are equivalent to a degenerate quasilinear hyperbolic wave-map-like system
that cannot be treated using standard energy methods.
\end{abstract}
  \setcounter{tocdepth}{2}

\maketitle


\section{Introduction}

The relativistic Euler equations
in Minkowski spacetime $(\mathbb{R}^{1+3},g)$
(where $g$ denotes the Minkowski metric)
are the fundamental equations of motion 
for (special) relativistic fluids.
In this article, we derive local-in-time a priori estimates 
in weighted Sobolev spaces for solutions featuring a dynamic fluid-vacuum boundary
such that the fluid density vanishes at a specific rate, described below.
Solutions featuring our specific rate of vanishing are said to have
a ``physical-vacuum'' boundary.
The vanishing rate that we consider permits the fluid to accelerate along the fluid-vacuum boundary.
Relative to suitable Lagrangian coordinates, described below, 
the equations of motion take the form of a
\emph{degenerate quasilinear hyperbolic wave-map-like system\footnote{In particular, the map is 
constrained as a consequence of the four-velocity normalization condition \eqref{E:normalizationeulerian}.}}
that cannot be treated using standard energy methods such as those afforded by
the symmetric hyperbolic framework. Thus, 
to prove our result, we develop a relativistic extensions of the method used in
~\cite{CoLiSh2010,CoutandShkoller2011,CoutandShkoller2012} in the
non-relativistic problem. In proving our result, 
we exploit the full nonlinear structure of the Euler equations
relative to a Lagrangian coordinate system.
With respect to the non-relativistic counterpart~\cite{CoLiSh2010,CoutandShkoller2012} 
of our result, there are several new essential difficulties that we must overcome. 
First, our Lagrangian coordinate time slices are hypersurfaces of constant proper time,
which are not generally slices of constant inertial Minkowski time.
Second, unlike in the non-relativistic case,
the vorticity of the fluid acceleration is equal to \emph{non-vanishing} source terms;
in order control these source terms, we must exploit the full nonlinear structure of the equations.   Third,
we must extract three-dimensional spatial regularity from the four-dimensional spacetime operators 
that naturally arise in energy estimates and elliptic-type estimates.

We will first state the relativistic Euler equations in Eulerian coordinates, 
but our analysis will be founded upon
the Lagrangian formulation.

A different approach to deriving 
a priori estimates for the relativistic Euler equations has recently been established in \cite{JangLeflochMasmoudi2015},  
using a generalization of the approach in \cite{JangMasmoudi2009, JangMasmoudi2015} for the non-relativistic problem.

\subsection{The Eulerian formulation}
\label{SS:EULERIANFORM} Throughout the paper, we set the speed of light equal to $1$.
In Eulerian variables and relative to arbitrary coordinates on $(\mathbb{R}^{1+3},g)$,  
the isentropic relativistic Euler equations are 
\begin{subequations} \label{E:EULER}
\begin{align}
	\D_{\alpha} (n u^{\alpha}) 
	& = 0,
		\label{E:EULERCONTINUITY} \\
	(\rho + p)u^{\alpha} \D_{\alpha} u^{\mu}
		+ (u^{\mu} u^{\alpha} + g^{\mu \alpha}) \D_{\alpha} p 
		& = 0\,,
		\label{E:REVELOCITY}
\end{align}
\end{subequations}
for $\mu = 0,1,2,3$, and where $\D$ denotes the (flat) Levi-Civita connection associated with $g$.
In standard rectangular coordinates $\lbrace x^{\mu} \rbrace_{\mu = 0,1,2,3}$ on $\mathbb{R}^{1+3}$,
in which $g = \mbox{\upshape diag}(-1,1,1,1)$,
the operator $\D_{\alpha}$ coincides with the rectangular partial derivative operator
$\frac{\partial}{\partial x^{\alpha}}$;
moreover, $n$ denotes the \emph{fluid proper number density},
$\rho$ denotes the \emph{proper energy density}, and $p$ denotes the \emph{pressure}.
Physically relevant solutions require that
$n \geq 0$, $\rho \geq 0$, and $p \geq 0$.
The \emph{four-velocity} $u$ is a timelike future-directed\footnote{Relative to the rectangular coordinates,
the future-directed property of $u$ is equivalent to $u^0 > 0$.} 
vectorfield normalized by 
\be\label{E:normalizationeulerian}
g_{\alpha \beta} u^{\beta} u^{\beta} = - 1.
\ee
Condition \eqref{E:normalizationeulerian} 
can be viewed as a constraint on the initial data that
is preserved by the nonlinear flow.
We refer the reader to  \cite{Sy1937} and \cite{dC2007} for a detailed discussion of the relativistic Euler equations.  
We shall employ the Einstein summation convention,
in which pairs of raised and lowered indices are summed; we give a more detailed description of this notation below.

The equations (\ref{E:EULER}) are relativistic analogs 
of the conservation of mass and momentum from the non-relativistic setting.
We note that there are more unknown quantities
than equations present. Thus, to close the system, 
we assume a \emph{barotropic equation of state} 
\begin{align}\label{E:EQUATIONOFSTATE}
p & = p(\rho).
\end{align}
The precise function $p(\rho)$ will be described below 
in Sect.~\ref{SS:DATAANDEOS}.

In our analysis, we use the following constitutive relation 
 afforded by the laws of thermodynamics:
\begin{align}\label{E:THERMO}
s
& = \frac{d \rho}{d n}\,,
\end{align} 
where
\begin{align} 
s & := \frac{\rho + p}{n}
\label{E:ENTHALPY}
\end{align}
is the \emph{enthalpy per particle}.

It is a standard fact that equation
\eqref{E:REVELOCITY} can be replaced with the equivalent equation
 \begin{align}\label{E:EULERIANENTHALPYEVOLUTION}
u^{\alpha} \D_{\alpha}(s u_{\mu}) + \D_{\mu} s = 0\,,
\end{align}
for $\mu =0,1,2,3$; see, for example, \cite{Sy1937,dC2007}.
The equivalence of \eqref{E:REVELOCITY} and  \eqref{E:EULERIANENTHALPYEVOLUTION}  follows from the identity
$u_{\alpha} \D_{\beta} u^{\alpha} = 0 $, which is a simple consequence of \eqref{E:normalizationeulerian}.

\begin{remark}
The structure of equation \eqref{E:EULERIANENTHALPYEVOLUTION}
is of fundamental importance for our analysis of the regularity of the 
the fluid vorticity; see Remark~\ref{R:ALLTHREEFORMS} below.
\end{remark}

\subsection{Initial data, physical vacuum, and the equation of state}
\label{SS:DATAANDEOS}

\subsubsection{The vacuum boundary}
We consider the following 
initial data for the relativistic Euler equations:
\begin{align} \label{E:INITIALDATAEULERIAN}
	(\mathring{n},\mathring{v}^1,\mathring{v}^2,\mathring{v}^3)
	:=
	(n,u^1,u^2,u^3)|_{t=0}\,,
\end{align}
 where throughout we use $t$ to denote the Minkowski time coordinate $x^0$.
Given a barotropic equation of state, 
the data \eqref{E:INITIALDATAEULERIAN} uniquely determines all fluid variables $p$, $\rho$, etc.
at the initial time $t=0$. We study solutions of the Euler equations that propagate the dynamic vacuum
boundary $\g \mathscr{P}$, 
where $\mathscr{P}$ is the closure of the subset of $\mathbb{R}^{1+3}$
where $n$ is positive:
\begin{align} \label{E:POSITIVITYSET}
\mathscr{P} := \mbox{cl} \lbrace q \in \mathbb{R}^{1+3} | \ n(q) > 0 \rbrace.
\end{align}
Points o the vacuum boundary $\partial\mathscr P$ move with the fluid velocity $u$.  
That is, points $q$ in the $3$-surface $\partial\mathscr P$
have the velocity $u(q) \in T_q \partial\mathscr P$. In particular,
\begin{align} \label{E:FREEBOUNDARYCONDITION}
\text{ the vectorfield }  \ u^{\alpha} \frac{\partial}{\partial x^{\alpha}} \ \text{ is tangent to } \ \partial\mathscr P.
\end{align}
For our analysis, it will be convenient to
consider a particular foliation of $\mathscr{P} \subset \mathbb{R}^{1+3}$
by spacelike hypersurfaces.
Specifically, our analysis relies on the following foliation:
\begin{align} \label{E:POSITIVITYFOLIATED}
	\mathscr{P} 
	& = \cup_{\tau \in I} \widetilde{\mathfrak{M}}_{\tau},
\end{align}
where the manifolds 
$\widetilde{\mathfrak{M}}_{\tau}$ are defined in Remark~\ref{rem_Mslice}.

\subsubsection{The  material manifold $\mathfrak{M}$}
\label{SS:materialmanifold}

The $3-$d \emph{material manifold}
$\mathfrak{M}$ parameterizes the collection of fluid particles at 
the initial time $t=0$. More precisely, we define $\mathfrak{M}$ to be
\begin{align} \label{E:MATMAN}
	\mathfrak{M} \simeq \mbox{cl} \lbrace (0,x^1,x^2,x^3) \in \mathbb{R}^{1+3} \ | \ \mathring{n} (x^1,x^2,x^3)> 0 \rbrace \,.
\end{align}

\subsubsection{The physical vacuum condition}

The speed of sound $c_s$ is defined as
\begin{align} \label{E:SOUNDSPEED}
	c_s := \sqrt{\frac{dp}{d \rho}} \,.
\end{align}
Our admissible equations of state, by definition, must have the following properties:
\begin{subequations}
\begin{align}
	p & = 0 \iff n = 0, 
		\label{E:PVANISHESIFFNDOES} \\
	\rho & = 0 \iff n = 0,
		\label{E:RHOVANISHESIFFNDOES} \\
	0 & \leq c_s \leq 1,
		\label{E:SPEEDOFSOUNDCONSTRAINTS} 
			\\
	c_s & = 0 \iff n = 0.
	\label{E:SPEEDOFSOUNDVANISHESINVACUUM}
\end{align}
\end{subequations}
The conditions in \eqref{E:SPEEDOFSOUNDCONSTRAINTS} can be stated in physical terms as follows: 
the speed of sound is non-negative and does not exceed the speed of light;
\eqref{E:SPEEDOFSOUNDCONSTRAINTS} is connected to the character of the equations,
which are hyperbolic when $c_s > 0$ and degenerate hyperbolic when $c_s = 0$;
the condition \eqref{E:SPEEDOFSOUNDVANISHESINVACUUM} implies, in particular,
that $c_s$ vanishes along the vacuum boundary $\partial \mathscr{P}$, from which it follows that
sound waves are not able to penetrate $\partial \mathscr{P}$
and propagate into vacuum spacetime.

In order to permit the fluid to accelerate along the vacuum boundary, 
we must impose an additional condition on the {\it rate of degeneracy} of the
sound speed $c_s$. We require that
\begin{align} \label{E:PHYSICALVACUUM}
	\D_{\hat{N}} (c_s^2)|_{\partial \mathfrak{M} } < 0 \,,
\end{align}
where
$\D_{\hat{N}}$ denotes the 
derivative of $c_s^2$ in the direction of the 
\emph{outer} unit normal $\hat{N}$ to the two-dimensional boundary surface
$\partial \mathfrak{M}$. Since \eqref{E:SPEEDOFSOUNDVANISHESINVACUUM}
implies that $c_s^2$ vanishes along $\partial \mathfrak{M}$,
we see from \eqref{E:PHYSICALVACUUM} that 
$c_s^2 \rightarrow 0$ like the distance function to $\partial \mathfrak{M}$.
In particular, there are constants $0 < C_1 < C_2$ such that
for $q \in \mathfrak{M} $, we have
\begin{align} \label{E:SPEEDOFSOUNDSQUAREDGROWSLIKEDISTANCETOBOUNDARY}
	C_1 d(q,\partial \mathfrak{M}) 
	& \leq c_s^2(q) 
	\leq C_2 d(q,\partial \mathfrak{M}),
\end{align}
where $d(q,\partial \mathfrak{M})$
denotes the Euclidean distance\footnote{Here we are viewing $\mathfrak{M}$ as a subset of Euclidean space $\mathbb{R}^3$.} 
in $\mathfrak{M}$ from $q$ to $\partial \mathfrak{M}$.

The condition \eqref{E:PHYSICALVACUUM} has been termed the {\it physical vacuum condition} in the non-relativistic setting; see 
\cite{CoutandShkoller2012} and \cite{JangMasmoudi2015}, and the references therein.  In the case that $c_s^2 \to 0$ like 
$d(q,\partial \mathfrak{M} )^p$ for integers $p\ge 2$, the analysis becomes significantly easier, but the fluid acceleration vanishes along the 
vacuum boundary.  If $c_s^2 \to 0$ like 
$d(q,\partial \mathfrak{M} )^s$ for some real number $s < 1$, then the problem is made highly singular and currently no theory is available for this
case.

\subsubsection{The equation of state}
A typical example of an equation of state used in astrophysics is that of a neutron star, where  in the interior region of the star we have $p\sim \frac{1}3\rho$ and as $\rho\to0$ the equation of state takes the form $p\sim \rho^{5/3}$ (see, for example, \cite{HsLiMa2004}).
For an extensive discussion of the equations-of-state occurring in the description of stars we refer the reader to~\cite{ZeNo}.

We consider the following equation-of-state:
\begin{align} \label{E:QUADRATIC}
p(\rho)  = \rho^\gamma \,, \ \ \gamma > 1 \,,
\end{align} which due to~\eqref{E:THERMO}--\eqref{E:ENTHALPY},
enforces the functional dependence between $\rho$ and $n$ of the form 
$\rho(n)=n (k-n^{\gamma-1})^{\frac1{1-\gamma}}$ for some $k>0$. Without loss of generality, we assume that $k=1$ and therefore that
\be\label{E:RHOOFN}
\rho = \frac{n}{(1-n^{\gamma-1})^{\frac 1{\gamma-1}}}.
\ee
It follows that $c_s^2 = \frac{\gamma n^{\gamma-1}}{1-n^{\gamma-1}}$ and since $0\le c_s^2<1$ we have
\begin{align} 
\label{E:PROPERNUMBERDENSITYSIZERESTRICTION0}
0\le n^{\gamma-1} < \frac{1}{\gamma+1}  \ \ \text{ on } \ \mathbb{R}^{1+3},
\end{align}
while~\eqref{E:SPEEDOFSOUNDSQUAREDGROWSLIKEDISTANCETOBOUNDARY}
implies that there exist constants $c_1,c_2>0$ such that 
\begin{align} \label{E:VACUUMBOUNDARY}
	c_1 d(q,\partial \mathfrak{M} ) 
	& \leq \mathring n^{\gamma-1}(q) 
	\leq c_2 d(q,\partial\mathfrak{M} ), \ \ q\in \mathfrak{M} \,.
\end{align}
From~\eqref{E:ENTHALPY}, we derive the following expression for the enthalpy per particle:
\begin{align}
\label{E:ENTHALPYSPECIAL0}
s & = \frac{1}{(1-n^{\gamma-1})^{\frac \gamma{\gamma-1}}}, \ \ \gamma>1.
\end{align}

\begin{remark}
When $\gamma=1$,  the equation of state $p = \mathscr{A} \rho, \ \ \mathscr{A}>0$ is
commonly studied in cosmology~\cite{Weinberg2009}, but it is not well-suited for the study of isolated fluid systems. 
\end{remark}

\subsubsection{Basic assumptions}\label{SSS:SIMPLE}
For clarity of exposition, we shall make a few simplifying assumptions on the initial data and the equation of state, which 
 capture all of the essential features of the analysis. 
 
\vspace{.1 in}
\noindent
{\bf Assumption 1. Quadratic equation of state.}
We consider the following equation of state
\begin{align} \label{E:QUADRATIC}
p(\rho)  = \rho^2,
\end{align} which due to~\eqref{E:RHOOFN}
gives the relationship
$$
\rho = \frac{n}{1-n}.
$$
It follows that $c_s^2 = \frac{2n}{1-n}$ and by~\eqref{E:PROPERNUMBERDENSITYSIZERESTRICTION0}
we have
\begin{align} 
\label{E:PROPERNUMBERDENSITYSIZERESTRICTION}
0\le n < \frac{1}3  \ \ \text{ on } \ \mathbb{R}^{1+3},
\end{align}
while~\eqref{E:SPEEDOFSOUNDSQUAREDGROWSLIKEDISTANCETOBOUNDARY}
implies that there exist constants $c_1,c_2>0$ such that 
\begin{align} \label{E:VACUUMBOUNDARY}
	c_1 d(q,\partial \mathfrak{M} ) 
	& \leq \mathring n(q) 
	\leq c_2 d(q,\partial\mathfrak{M} ), \ \ q\in \mathfrak{M} \,.
\end{align}
From~\eqref{E:ENTHALPYSPECIAL0}, we obtain the following expression for the enthalpy per particle:
\begin{align}
\label{E:ENTHALPYSPECIAL}
s & = \frac{1}{(1-n)^2}.
\end{align}

\vspace{.1 in}
\noindent
{\bf Assumption 2. Horizontally-periodic  reference domain $\mathfrak{M} $ }.
We assume that our initial $3$-surface $\mathfrak{M} $ is of the form
\[
\mathfrak{M}  = \mathbb T^2\times[0,1],
\] 
where $\mathbb T^2$ denotes the $2$-torus. This assumption allows us to use a single coordinate chart to describe the material manifold.
The surfaces $\{(t,x^3) = (0,0)\}$ and $\{(t,x^3) = (0,1)\}$ correspond to the initial location of the vacuum boundary.

\vspace{.1 in}
\noindent
{\bf Assumption 3. Small initial number density $\mathring n$.}
Without loss of generality, 
we assume that the initial number density $\mathring n$ is uniformly small on $\mathfrak{M}$:
\be\label{E:SMALLDENSITY}
\|\mathring n\|_{L^\infty(\mathfrak{M} )} \le \frac{\varepsilon}{\|\mathring{v}^0\|_{L^ \infty ( \mathfrak{M} )} } \ll 1.
\ee
By the vacuum boundary condition, the assumption~\eqref{E:SMALLDENSITY} is always verified in a small neighborhood of the boundary $\partial \mathfrak{M}$.
Away from the boundary, the Euler equations~\eqref{E:EULER}-~\eqref{E:normalizationeulerian} are non-degenerate and a priori estimates
in Sobolev spaces follow from standard energy methods for hyperbolic equations.
The removal of the assumption~\eqref{E:SMALLDENSITY} requires a standard partition-of-unity argument, separating the analysis into regions
that are near and away from the vacuum boundary. The assumption \eqref{E:SMALLDENSITY} allows us to avoid this partitioning argument 
and focus on the essential difficulties produced by the degeneracy of the vacuum boundary.

\subsection{Lagrangian coordinates}
\label{SS:LAGRANGIANCOODS}
We let $(y^1,y^2,y^3)$ be the rectangular coordinates induced on $\mathfrak{M}$ by the rectangular spacetime coordinates.
Our goal is to
formulate the relativistic Euler equations as 
a quasilinear wave-map-type system of equations
for the components of a map 
$\upeta: I \times \mathfrak{M} \rightarrow \mathbb{R}^{1+3}$,
where $I$ is an interval of proper time (as explained below) containing the initial time
$0$ in its interior.
By definition, Lagrangian coordinates are such that $\upeta$ is the flow map of the four-velocity vectorfield $u$:
\begin{align}
	\partial_{\tau} \upeta^{\nu}(\tau,y^1,y^2,y^3)
	& = u^{\nu} \circ \upeta (\tau,y^1,y^2,y^3), 
		\label{E:FLOWMAPEEQUATION} \\
	\upeta^{\nu}(0,y^1,y^2,y^3)
	& = (0,y^1,y^2,y^3).
		\label{E:FLOWMAPIC}
\end{align}
The components of $\upeta$ may be identified with the rectangular spacetime coordinates:
$x^{\nu} = \upeta^{\nu}(\tau,y^1,y^2,y^3)$. 
We often use the alternate notation $y^0 = \tau$ and $x^0 = t$. 
Lagrangian coordinates correspond to the {\em co-moving} coordinates, 
while the original rectangular coordinates on $\R^{1+3}$
correspond to the {\em inertial} coordinates.

For the remainder of the article, we equip $I\times\mathfrak{M}$ with the Lagrangian coordinates $(\tau,y^1,y^2,y^3)$.
We denote the full spacetime Lagrangian coordinate gradient by
\[
\ITIMESMD := (\partial_{\tau},\partial_1,\partial_2,\partial_3),
\]
the Lagrangian spatial $3$-gradient by 
\[
\MD := (\partial_1,\partial_2,\partial_3),
\] 
and the Lagrangian spatial horizontal gradient by 
\[
\AngMD := (\partial_1,\partial_2).
\]

\begin{definition}[\textbf{Important submanifolds of the spacetime cylinder $I \times\mathfrak{M}$ }]\label{D:CONSTANTTIMEHYPERSURFACES}
We introduce the constant $\tau$-surfaces:
	\begin{align}
		\mathfrak{M}_{\tau}
		& := \lbrace (\hat\tau,y^1,y^2,y^3) \ | \ \hat\tau = \tau, (y^1,y^2,y^3) \in \mathfrak{M} \rbrace \subset I \times \mathfrak{M}.
\end{align}
 For fixed $\tau$ and $y^3$,  we introduce the codimension two surface $\partial \mathfrak{M}_{ \tau, y^3}$ as
\begin{align}
\partial \mathfrak{M}_{ \tau, y^3}
& := \lbrace (\hat \tau,y^1,y^2,\hat y^3) \ | \ \hat \tau = \tau, \hat y^3 = y^3 \rbrace \subset \mathfrak{M}_{\tau}.
\end{align}
The sets $\partial \mathfrak{M}_{ \tau,y^3=1}$ and $\partial \mathfrak{M}_{ \tau,y^3=0}$ denote the top and bottom, respectively, the  constant
$\tau$-slices $\mathfrak{M}_ \tau$, and we denote the union $\partial \mathfrak{M}_{ \tau,y^3=1} \cup \partial \mathfrak{M}_{ \tau,y^3=0}$
by $\partial\mathfrak{M}_\tau$.
\end{definition}

%

\begin{remark} \label{R:PARTIALTAUISU}
Note that $\partial_{\tau}$ is precisely the vectorfield $u^{\alpha} \frac{\partial}{\partial x^{\alpha}}$ expressed
relative to Lagrangian coordinates; see \eqref{E:PARTIALTAUISTHEvectorfieldU}.
\end{remark}

\begin{remark}
We note that the image of a slice of constant Lagrangian coordinate time under the Lagrangian flow 
$\upeta$ is no longer (necessarily) a flat time slice.   
That is, $\upeta(\tau,\mathfrak{M})$ is generally not a constant time-slice with respect to inertial coordinates.    
\end{remark}

\subsection{Summation convention}\label{sec::summation_convention}
Lowercase Greek indices $\alpha$, $\beta$, etc. range from $0$ to $3$ and correspond to 
the rectangular coordinates $\lbrace x^{\alpha} \rbrace_{\alpha = 0,1,2,3}$ on $\mathbb{R}^{1+3}$.
Similarly, lowercase Latin indices $a$, $b$, etc. range from $1$ to $3$ and correspond to 
the rectangular \emph{spatial} coordinates $\lbrace x^a \rbrace_{a=1,2,3}$.
We write $\frac{\partial}{\partial x^{\alpha}}$ to denote the corresponding rectangular coordinate partial derivative operator.
We often use the alternate notation $x^0 = t$ and $\frac{\partial}{\partial x^0} = \frac{\partial}{\partial t}$.

Capital Latin indices $A,$ $K$, $L$, etc. range from $0$ to $3$ and correspond to the Lagrangian spacetime coordinates 
$\lbrace y^K \rbrace_{K = 0,1,2,3}$ on $I \times \mathfrak{M}$.

We write $\partial_K$ to denote the corresponding Lagrangian coordinate partial derivative operator.
We often use the alternate notation $y^0 = \tau$ and $\partial_0 = \partial_{\tau}$.

We sum repeated indices over their respective ranges.
We lower and raise  Greek and Latin indices with the Minkowski metric $g$ and its inverse $g^{-1}$, respectively.

\subsection{Lagrangian formulation}
\label{SS:LAGRANGIAN}

\subsubsection{Definitions and preliminary geometric constructions}

\begin{definition}[\textbf{The fluid variables as a function of the Lagrangian coordinates}]
\label{D:LAGRANGIANUNKNOWNS}
In terms of the flow map $\upeta$ of Sect.~\ref{SS:LAGRANGIANCOODS}, we define
(for $\mu = 0,1,2,3$)
\begin{align} \label{E:LAGRANGIANUNKNOWNS}
f = n \circ \upeta,  \ \ 
v^\mu = u^{\mu}\circ \upeta, \ \ 
S = s\circ\upeta.
\end{align}
\end{definition}

In deriving the Lagrangian formulation of the equations
and in our analysis, we will often 
change back and forth between the Lagrangian
and the rectangular coordinates.

\begin{definition}[\textbf{The change-of-variables matrices}]
\label{D:CHOVMATRICES}
We define the $4 \times 4$ 
``change-of-variables"
matrices $\mathscr{M}$ and $\mathscr{A}$,
the Jacobian determinant $\mathscr{J}$,
and the $4 \times 4$ cofactor transpose matrix $\Cof$
by
\begin{subequations}
\begin{align}
	\mathscr{M}_K^{\nu}
	& := \partial_K \upeta^{\nu},
		\label{E:MDEF} \\
	\mathscr{A}_{\nu}^K
	& := (\mathscr{M}^{-1})_{\nu}^K,
		\label{E:MINVERSEDEF} \\
	\mathscr{J}
	& := \mbox{\upshape{det}} \mathscr{M},
		\\
	\Cof_{\alpha}^K 
	& :=  \mathscr{J} \mathscr{A}_{\alpha}^K,
	\label{E:COFMATDEF}
\end{align}
\end{subequations}
\end{definition}
\noindent
with $K,L,\alpha,\beta,\mu,\nu = 0,1,2,3$.
Note that
\begin{align} \label{E:MATRIXTIMESMATRIXINVERSE}
	\mathscr{M}_K^{\alpha} \mathscr{A}_{\alpha}^L
	 = \delta_K^L, \  \ 	\mathscr{M}_A^{\mu} \mathscr{A}_{\nu}^A
	= \delta_{\nu}^{\mu},
\end{align}
where $\delta_K^L$ and $\delta_{\nu}^{\mu}$  both denote the  Kronecker delta, and where we sum over repeated lower and upper indices.
Moreover, from \eqref{E:MDEF}, 
\eqref{E:FLOWMAPEEQUATION},
and \eqref{E:MATRIXTIMESMATRIXINVERSE},
we have
\begin{align} \label{E:FOURVELOCITYAPPEARSINCHOVMATRIX}
	\mathscr{M}_0^{\mu}
	 = u^{\mu},
	\ \  v^{\alpha} \mathscr{A}_{\alpha}^K 
	= \delta_0^K.
\end{align}

By the chain rule, we have
\begin{align} \label{E:CHAINRULES}
	\partial_K 
	 = \mathscr{M}_K^{\alpha} \frac{\partial}{\partial x^{\alpha}},
	\ \  
	\frac{\partial}{\partial x^{\alpha}}
	= \mathscr{A}_{\alpha}^K \partial_K.
\end{align}
Thus, from \eqref{E:FOURVELOCITYAPPEARSINCHOVMATRIX} and \eqref{E:CHAINRULES}, 
we see that
\begin{align} \label{E:PARTIALTAUISTHEvectorfieldU}
\partial_\tau:= \partial_0 = u^{\alpha} \frac{\partial}{\partial x^{\alpha}},
\end{align}
as we mentioned in Remark~\ref{R:PARTIALTAUISU}.

For any vectorfield differential operator $\partial$, 
we note the following standard matrix and determinant differentiation identities:
\begin{subequations}
\begin{align} \label{E:DETAINVERSEDIFFERENTIATED}
	\partial \mathscr{A}_{\alpha}^K
	& = - \mathscr{A}_{\beta}^K 		
			\partial \partial_L \upeta^{\beta}
			\mathscr{A}_{\alpha}^L,
				\\
	\partial \mathscr{J}
	& = \mathscr{J} \mathscr{A}_{\alpha}^K \partial \mathscr{M}_K^{\alpha}
		= \mathscr{J} \mathscr{A}_{\alpha}^K \partial_K \partial \upeta^{\alpha},
			\label{E:JACOBIANDETDIFFERNTIATED}
			\\
	\partial (\mathscr{J} \mathscr{A}_{\alpha}^K)
	& = \mathscr{J} \partial_L \partial \upeta^{\beta}
			\left\lbrace
					\mathscr{A}_{\alpha}^K \mathscr{A}_{\beta}^L
				- \mathscr{A}_{\alpha}^L \mathscr{A}_{\beta}^K 
			\right\rbrace.
				\label{E:DETERMINANTTIMESDETAINVERSEDIFFERENTIATED} 
\end{align}
\end{subequations}

Moreover, as a simple consequence of \eqref{E:DETERMINANTTIMESDETAINVERSEDIFFERENTIATED}
and the symmetry property $\partial_L \partial_K \upeta^{\beta} = \partial_K \partial_L \upeta^{\beta}$,
we have the well-known \emph{Piola identity}:
\begin{align} \label{E:PIOLA}
	\partial_K \Cof_{\alpha}^K & = 0.
\end{align}

In addition, we deduce from 
\eqref{E:CHAINRULES}
and
\eqref{E:JACOBIANDETDIFFERNTIATED}
that
\begin{align} \label{E:PARTIALTAUJISDETERMINEDINTERMSOFDIVU}
	\partial_{\tau} \mathscr{J}
	& = \mathscr{J} \mathscr{A}_{\alpha}^K \partial_K v^{\alpha}
		= \mathscr{J} \partial_{\alpha} u^{\alpha}.
\end{align}

\begin{definition}(\textbf{Differential operators})
	\label{D:DIFFERENTIALOPERATORS}
	We define the following differential operators on spacetime vectorfields $X^{\mu}$:
	\begin{subequations}
	\begin{align}\label{E:LAGRANGIANGRADIENT}
		\Lagdiff_\mu X^{\nu}
		& := 
			\mathscr{A}_{\mu}^K \partial_K X^{\nu}, 
			\\
		\Lagdiv X
		& := \Lagdiff_{\alpha} X^{\alpha}
			=
			\mathscr{A}_{\alpha}^K \partial_K X^{\alpha},
			\label{E:LAGRANGIANDIVERGENCE} \\
		(\Lagvort X)_{\mu \nu}  
		 \label{E:LAGRANGIANVORTICITY} 
		& := \Lagdiff_{\mu} X_{\nu}
					-
				 	\Lagdiff_{\nu} X_{\mu}
			=
			\mathscr{A}_{\mu}^K \partial_K X_{\nu}
			-
			\mathscr{A}_{\nu}^K \partial_K X_{\mu}, 
	\end{align}
	\end{subequations}
	where the second equalities in the above formulas
	follow from \eqref{E:CHAINRULES}.
	
	Furthermore, we define, relative to the Lagrangian 
	coordinate vectorfields $(\partial_1,\partial_2,\partial_3)$, the following flat divergence and vorticity operators:
	\begin{subequations}
	\begin{align}
		\Flatdiv Y
		& := \sum_{A,a=1}^3\delta_a^{A} \partial_{A} Y^a,
			\label{E:FLATDIV} \\
		(\Flatvort Y)_{K j} 
		 & := \partial_{K} Y_j
		 	- \sum_{A,a=1}^3\delta_{K}^a \delta_j^{A} \partial_{A} Y_a, \ \ (K,j=1,2,3).
		 	\label{E:FLATVORT}
		\end{align}
	\end{subequations}
	In the above formulas, 
	$Y 
	=
	\sum_{a=1}^3 Y^a \frac{\partial}{\partial x^a}
	=
	\sum_{a=1}^3Y^a(\tau,y^1,y^2,y^3) \frac{\partial}{\partial x^a}$
	denotes a vectorfield in $\mathbb{R}^3$
	defined along $\mathfrak{M}_{\tau}$. Thus, 
	$Y^a(\tau,y^1,y^2,y^3)$ denotes a rectangular 
	spatial component function
	viewed as a function of the Lagrangian coordinates.
\end{definition}

\begin{remark}
	Note that \eqref{E:LAGRANGIANGRADIENT},
	\eqref{E:LAGRANGIANDIVERGENCE},
	and
	\eqref{E:LAGRANGIANVORTICITY} 
	are just the usual gradient, divergence, and vorticity operators
	associated to the Minkowski connection 
	$\D$. The purpose of the superscript
	``$\upeta$'' is to remind the reader that when the differential operators
	are expressed relative to Lagrangian coordinates,
	they depend on the derivatives of the flow map $\upeta$.
\end{remark}

\begin{definition}(\textbf{Contractions and inner products})
\label{D:LAGRANGIANPRODUCTS}
We define the following operators on vectorfields $X^\mu$ and $Y^\nu$:
	\begin{subequations}
	\begin{align}
		\Lagdiff X \cdot \Lagdiff Y
		& := \Lagdiff_\alpha X^{\beta} \ 
				 \Lagdiff_{\beta} Y^{\alpha}, \label{E:GRADIENTINNERPRODUCT}
			\\
		\langle X, Y \rangle_g
		& := X_{\alpha} Y^{\alpha},
			\label{E:INNERPRODUCT} \\
		\langle \Lagvort X, \Lagvort Y \rangle_g
		& := \Lagvort X_{\alpha\beta} \  \Lagvort Y^{\alpha \beta}. \label{E:VORTICITYINNERPRODUCT}
	\end{align}
	\end{subequations}
\end{definition}

\subsubsection{Initial values of the change of variables matrices and the Jacobian determinant}
\label{SSS:IVMATRICES}
The $4\times 4$ matrices 
$\mathring{\mathscr{M}} := \mathscr{M}|_{t=\tau=0}$
and
$\mathring{\mathscr{A}} := \mathscr{A} \big|_{t=\tau=0}$ 
are given by
\begin{subequations}
\begin{align} \label{E:INITIALCHOVMATRIX}
\left( \begin{array}{cc}
\mathring{\mathscr{M}}_0^0 & \mathring{\mathscr{M}}_0^i  \\
\mathring{\mathscr{M}}_j^0 & \mathring{\mathscr{M}}_j^i 
\end{array} \right)
& =
 \left( \begin{array}{cc}
 \mathring{v}^0 &  (\mathring{v}^i)_{1 \times 3} \\
 (\mathbf{0})_{3 \times 1} & (\mathscr \delta^i_j)_{3 \times 3}
\end{array} \right),
\end{align}
\begin{align} \label{E:INITIALINVERSECHOVMATRIX}
\left( \begin{array}{cc}
\mathring{\mathscr{A}}_0^0 &  \mathring{\mathscr{A}}_j^0 \\
\mathring{\mathscr{A}}_0^i & \mathring{\mathscr{A}}_j^i 
\end{array} \right)
& =
 \left( \begin{array}{cc}
 \frac{1}{\mathring{v}^0} &  (\mathbf{0})_{3 \times 1} \\
	(-\frac{\mathring{v}^i}{\mathring{v}^0})_{1 \times 3}& (\mathscr \delta^i_j)_{3 \times 3}
\end{array} \right).
\end{align}
\end{subequations}
In the corresponding non-relativistic problem (see~\cite{CoutandShkoller2012, JangMasmoudi2015}),
the Jacobian determinant $\mathscr{J}$ is initially identically $1$. In contrast,
in the relativistic setting, we deduce from \eqref{E:INITIALCHOVMATRIX} that
\begin{align} \label{E:INITIALJACOBIANDETERMINANT}
 \mathring{\mathscr{J}}
 :=\mathscr{J}\big|_{\tau=0} 
 = \mathring{v}^0.
\end{align}

\subsubsection{The Lagrangian formulation of the equations}
We now derive our Lagrangian formulation of the relativistic Euler equations
in the special case of the equation of state \eqref{E:QUADRATIC}.
We recall that by \eqref{E:ENTHALPYSPECIAL} and~\eqref{E:LAGRANGIANUNKNOWNS}, we have the following
expression for the enthalpy per particle in terms of $f$:
\begin{align} \label{E:ENTHAGAIN}
	\Enth
	& = \frac{1}{(1-f)^2}.
\end{align}

\begin{proposition}[\textbf{Lagrangian formulation of the relativistic Euler equations}] 
\label{P:LAGFORMULATION}
Relative to the Lagrangian coordinates constructed in Sect.~\ref{SS:LAGRANGIANCOODS}
and the Lagrangian unknowns of Definition~\ref{D:LAGRANGIANUNKNOWNS},
the relativistic Euler equations
\eqref{E:EULERCONTINUITY}-\eqref{E:normalizationeulerian}
with the equation of state \eqref{E:QUADRATIC}
can be expressed as the following system 
in the components $\lbrace \upeta^{\alpha} \rbrace_{\alpha = 0,1,2,3}$:
\begin{align} \label{E:INITIALFUNCTTION}
	f \mathscr{J}
	& = \mathring{F},
		\\
	\mathring{F} \Enth \partial_{\tau} v_{\mu}
	+
	\partial_K
	\left\lbrace
		\mathring{F}^2 \mathscr{A}^K_{\mu} \mathscr{J}^{-1} \Enth
	\right\rbrace
	- 2 \mathring{F}^2 v_{\mu} \mathscr{J}^{-2} \Enth^{3/2}\partial_{\tau} \mathscr{J} 
	& = 0,
	 \label{E:LAGRNAGIANMAINEVOLUTIONEQUATION} \\
g_{\alpha\beta}v^{\alpha}v^{\beta} 
& = -1,
\label{E:NORMALIZATIONLAGRANGIAN}
\end{align}
where 
\begin{align}
	\mathring{F}
	& :=  \mathring{n}  \mathring{v}^0
		\label{E:MATHRINGFDEFINITION}
\end{align}
and, in view of \eqref{E:NORMALIZATIONLAGRANGIAN}, we have $v^0 = \sqrt{1 + \sum_{a=1}^3 (v^a)^2}$.
In the above equations,
$v^{\mu} = \partial_{\tau} \upeta^{\mu}$,
$\mathscr{A}^K_{\mu}$ and $\mathscr{J}$ 
depend on the first Lagrangian derivatives of $\upeta$ as specified in
Definition~\ref{D:CHOVMATRICES},
and $\Enth$ is determined in terms of $f$ via \eqref{E:ENTHAGAIN}.

Moreover, equation 
 \eqref{E:LAGRNAGIANMAINEVOLUTIONEQUATION} can be written in the following two equivalent forms:
\begin{subequations}
\begin{align}
	\partial_{\tau} (\Enth v_{\mu}) 
	+ 
	\mathscr{A}^K_{\mu} \partial_K \Enth 
	& = 0,
	\label{E:LAGRANGIANENTHALPYEVOLUTION}
		\\
	\mathring{F} \Cof^3_{\mu} \partial_3 (\mathscr{J}^{-2}) 
	+
	2 \Cof^3_{\mu} (\partial_3 \mathring{F}) \mathscr{J}^{-2} 
	& = 
	- \Enth^{-1/2} \partial_{\tau}^2 \upeta_{\mu}
	+ 2 \mathring{F} \mathscr{J}^{-2} \g_{\tau} \mathscr{J} \partial_{\tau} \upeta_{\mu}
		\label{E:LAGRANGIANENTHALPYEVOLUTIONNEEDEDFORELLIPTIC} 
			\\
	& \ \
		- \mathring{F} \Cof_{\mu}^0 \partial_{\tau} (\mathscr{J}^{-2})
		- \mathring{F}\sum_{A=1}^2 \Cof_{\mu}^{A} \partial_{A} (\mathscr{J}^{-2})
		- 2 \sum_{A=1}^2(\Cof_{\mu}^{A} \partial_{A} \mathring{F}) \mathscr{J}^{-2}.
			\notag
\end{align}
\end{subequations}
\end{proposition}

\begin{remark} From \eqref{E:SMALLDENSITY}, it follows that
$
\|\mathring F\|_{L^\infty(\mathfrak{M} )} \le \varepsilon \ll 1.
$
\end{remark}

\begin{remark}
Equations \eqref{E:INITIALFUNCTTION} and \eqref{E:NORMALIZATIONLAGRANGIAN}
are constraints, while \eqref{E:LAGRNAGIANMAINEVOLUTIONEQUATION}
is a second-order wave-map-type system in the components 
$\lbrace \upeta^{\alpha} \rbrace_{\alpha = 0,1,2,3}$.
\end{remark}

\begin{remark}\label{rem_Mslice}  Because the vacuum boundary moves with the fluid velocity $u$ according to \eqref{E:FREEBOUNDARYCONDITION}, the
set  $\mathscr{P} $ defined in \eqref{E:POSITIVITYSET} is in fact given as
	$\mathscr{P} 
	 = \cup_{\tau \in I}  \upeta(\tau,\mathfrak{M})\,$.
We see, then, that the hypersurfaces on the right-hand side of \eqref{E:POSITIVITYFOLIATED} 
are
$$
\widetilde{ {\mathfrak{M}}_{\tau}} := \upeta(\tau,\mathfrak{M}) \,.
$$
Thus, once we solve for the Lagrangian flow map $\upeta$, 
we obtain precise information about the location of the vacuum boundary
and the tangential motion of particles along it.
 \end{remark}

\begin{remark}[\textbf{All three forms of the momentum equation are important}]
	\label{R:ALLTHREEFORMS}
	Though equations
	\eqref{E:LAGRNAGIANMAINEVOLUTIONEQUATION}, 
	\eqref{E:LAGRANGIANENTHALPYEVOLUTION},
	and \eqref{E:LAGRANGIANENTHALPYEVOLUTIONNEEDEDFORELLIPTIC}
	are equivalent (given \eqref{E:INITIALFUNCTTION} and \eqref{E:NORMALIZATIONLAGRANGIAN}),
	we exploit the precise structure of each of these equations in a distinct way.
	We use \eqref{E:LAGRNAGIANMAINEVOLUTIONEQUATION} for energy estimates to control horizontal and 
	time derivatives and to establish the regularity of the vacuum boundary, while
	\eqref{E:LAGRANGIANENTHALPYEVOLUTION} is used for vorticity estimates,
	and \eqref{E:LAGRANGIANENTHALPYEVOLUTIONNEEDEDFORELLIPTIC} is used for estimating vertical derivatives by degenerate
	elliptic-type estimates.
\end{remark}
\begin{remark}[\textbf{Importance of~\eqref{E:LAGRANGIANENTHALPYEVOLUTION}}]
Since 
$\frac{\partial}{\partial x^{\mu}} \Enth = \mathscr{A}^K_{\mu}\g_K \Enth$
is annihilated by the vorticity operator \eqref{E:LAGRANGIANVORTICITY}, 
equation \eqref{E:LAGRANGIANENTHALPYEVOLUTION} allows us to derive 
an evolution equation for the vorticity of $\Enth v_{\mu}$ 
that leads to a gain in regularity compared to the regularity
suggested by a naive derivative count. 
We use a high-order differentiated version of this statement in Sect.~\ref{S:VORTICITY}.
\end{remark}

\begin{proof}[Proof of Prop.~\ref{P:LAGFORMULATION}]
Equation \eqref{E:NORMALIZATIONLAGRANGIAN} is a trivial consequence of \eqref{E:normalizationeulerian}.

To obtain \eqref{E:INITIALFUNCTTION}, 
we first note that in rectangular coordinates,
equation~\eqref{E:EULERCONTINUITY} is equivalent to
$
	\frac{\partial}{\partial x^{\alpha}} (n u^{\alpha}) = 0
$.
Expressing this equation in Lagrangian coordinates
with the help of \eqref{E:CHAINRULES} and \eqref{E:PARTIALTAUISTHEvectorfieldU},
we obtain 
\begin{align}
\partial_{\tau} f 
& =
- f \mathscr{A}_{\alpha}^K \partial_{\tau} \partial_K \upeta^{\alpha}.
\label{E:FEVOLUTION}
\end{align}
Next, we use \eqref{E:JACOBIANDETDIFFERNTIATED} to deduce
that $\mathscr{J}$ verifies a similar equation (but without a minus sign on the right-hand side of):
\begin{align} \label{E:JACOBIANEVOLUTION}
\partial_{\tau} \mathscr{J}
= \mathscr{J} \mathscr{A}_{\alpha}^K \partial_K \partial_{\tau} \upeta^{\alpha}.
\end{align} 
Combining \eqref{E:FEVOLUTION} and \eqref{E:JACOBIANEVOLUTION}, we deduce that 
$
\partial_{\tau} \ln (f \mathscr{J}) 
= 0
$.
Since 
$\mathscr{J}|_{\tau = 0} = \mathring{v}^0 := v^0|_{\tau=0}$,
we therefore obtain \eqref{E:INITIALFUNCTTION} as desired.

To obtain \eqref{E:LAGRANGIANENTHALPYEVOLUTION},
we simply use \eqref{E:CHAINRULES} and \eqref{E:PARTIALTAUISTHEvectorfieldU}
to express equation~\eqref{E:EULERIANENTHALPYEVOLUTION} in Lagrangian coordinates.

To derive \eqref{E:LAGRNAGIANMAINEVOLUTIONEQUATION},
we first use \eqref{E:ENTHAGAIN},
\eqref{E:FEVOLUTION}-\eqref{E:JACOBIANEVOLUTION},
and \eqref{E:INITIALFUNCTTION}
to deduce that
\begin{align} \label{E:ENTHDERIVATIVEID}
	\g_{\tau} \Enth
	& = 2 \Enth^{3/2} \g_{\tau} f
		= - 2 \Enth^{3/2} f \mathscr{J}^{-1} \g_{\tau} \mathscr{J}
		= - 2 \mathring{F} \Enth^{3/2} \mathscr{J}^{-2} \g_{\tau} \mathscr{J}.
\end{align}
Thus, we can express 
the product of $f \mathscr{J}$ and
the first product on LHS \eqref{E:LAGRANGIANENTHALPYEVOLUTION} as follows:
$
f \mathscr{J} \g_{\tau} (\Enth v_{\mu}) 
= f \mathscr{J} \Enth \g_{\tau} v_{\mu}
+ f \mathscr{J} v_{\mu} \g_{\tau} \Enth
= f \mathscr{J} \Enth \g_{\tau} v_{\mu}
- 2 \mathring{F}^2  v_{\mu} \Enth^{3/2} \mathscr{J}^{-2} \g_{\tau} \mathscr{J}
$.
Moreover, using 
\eqref{E:ENTHAGAIN},
definition \eqref{E:COFMATDEF}, 
and the Piola identity \eqref{E:PIOLA}, 
we can express 
the product of $f \mathscr{J}$ and
the second product on LHS \eqref{E:LAGRANGIANENTHALPYEVOLUTION} as follows:
$
f \mathscr{J} \mathscr{A}^K_{\mu} \g_K \Enth
=
	\partial_K(\Cof^K_{\mu} f^2 \Enth)
= \partial_K(\mathscr{A}^K_{\mu} \mathscr{J} f^2 \Enth)
$.
Inserting these identities into 
$f \mathscr{J} \times$ \eqref{E:LAGRANGIANENTHALPYEVOLUTION}
and using \eqref{E:INITIALFUNCTTION} to express
$f \mathscr{J} = \mathring{F}$,
we arrive at the desired equation \eqref{E:LAGRNAGIANMAINEVOLUTIONEQUATION}.

Equation \eqref{E:LAGRANGIANENTHALPYEVOLUTIONNEEDEDFORELLIPTIC} follows from
equation \eqref{E:LAGRANGIANENTHALPYEVOLUTION},
the fact that $\partial_K \Enth = 2 \Enth^{3/2} \partial_K f$
(much like in the first equality of \eqref{E:ENTHDERIVATIVEID}),
from using \eqref{E:INITIALFUNCTTION} to substitute 
$f = \mathring{F} \mathscr{J}^{-1}$ in the previous equation, 
from using the simple identity
$2 \mathscr{A}^K_{\mu} \partial_K (\mathring{F} \mathscr{J}^{-1})
= 2 \mathscr{A}^K_{\mu} (\partial_K \mathring{F}) \mathscr{J}^{-1}
	+
	2 \mathring{F} \Cof^K_{\mu} \mathscr{J}^{-1} \partial_K (\mathscr{J}^{-1})
= 
2 \mathscr{A}^K_{\mu} (\partial_K \mathring{F}) \mathscr{J}^{-1}
+
\Cof^K_{\mu} \partial_K (\mathscr{J}^{-2})
$
(see \eqref{E:PIOLA}),
and from straightforward calculations.
\end{proof}

\subsection{Notation}
\label{SS:NOTATION}

Given a $4$-vector $X^\mu,$ we denote by 
\[
\underline X = (X^1,X^2,X^3)
\]
the projection onto the spatial components of $X$.  
In our analysis, especially in Sect.~\ref{S:ELLIPTIC}, we shall often use the spatial components projection  $\underline{\upeta} = (\upeta^1,\upeta^2,\upeta^3)$ 
of the flow map $\upeta$.

$C$ and $c_i$ denote universal positive constants that may change from line to line. 
If $x,y \in \mathbb R$, then $x \lesssim y$ means 
that there exists a universal constant $C > 0$ such that $x \leq C y$. 
If $Z$ is non-negative, then we write $x \lesssim_Z y$ to mean that the constant $C$ from the previous inequality is
allowed to depend in a continuous increasing fashion on $Z$.
$\mathring{M}$ denotes a real number whose size 
is bounded by the initial data and it is allowed to vary from line to line.
For convenience, we sometimes soak factors of $\tau$ and $\delta^{-1}$ into $\mathring{M}$,
where $\delta > 0$ is a small constant introduced later in the paper.

We use the notation $P(\norm)$ to denotes a generic positive, increasing function of $\norm$. 
It is allowed to vary from line to line. 
We adopt the convention that \emph{$P$ is allowed to depend on constants, $\mathring{M}$, and $\tau$.}
To ease the notation when dealing with various error terms, by
$O_{L^p(\mathfrak{M}_{\tau})}(Q)$, 
we denote a term 
whose $\| \cdot \|_{L^p(\mathfrak{M}_{\tau})}$ norm is $\leq C Q$.

\subsubsection{Lebesgue and Sobolev norms relative to the Lagrangian coordinates}
\label{SS:LEBESGUEANDSOBOLEV}

We define spatial integrals over the hypersurfaces $\mathfrak{M}_{\tau}$ with 
respect to the volume form of the flat Euclidean metric on $\mathfrak{M}_{\tau}$, 
or equivalently, with respect to the measure $dy:=dy^1 dy^2 dy^3$.
That is,
\begin{align} \label{E:INTEGRALSDEFINITION}
	\int_{\mathfrak{M}_{\tau}}
		f
	\, dy
	& := 
	\int_{(y^1,y^2,y^3) \in \mathfrak{M}}
		f(\tau,y^1,y^2,y^3)
	\, dy^1 dy^2 dy^3.
\end{align}

For $1 \leq p < \infty$, we define the corresponding Lebesgue norms by
\begin{align} 
	\| f \|_{L^p(\mathfrak{M}_{\tau})}
	& := 
		\left(
			\int_{\mathfrak{M}_{\tau}}
				f^p
			\, dy
		\right)^{1/p}.
\end{align}

For integers $k \geq 0$, we define the corresponding 
$L^2-$based Sobolev norm by
\begin{align}
	\| f \|_{H^k(\mathfrak{M}_{\tau})}^2
	& := 
		\sum_{i_1 + i_2 + i_3 \leq k}
		\|\partial_1^{i_1} \partial_2^{i_2} \partial_3^{i_3} f \|_{L^2(\mathfrak{M}_{\tau})}^2,
\end{align}
where $\lbrace \partial_{A} \rbrace_{A=1,2,3}$ 
denote the Lagrangian \emph{spatial} coordinate partial derivative vectorfields.

\subsection{Main result}
Let $T>0$ and let $\upeta: [0,T] \times \mathfrak{M} \to \R^{1+3}$ be a smooth flow map.
For $\tau [0,T]$, we define the following square norm:
\begin{align} 
\norm(\tau)  
&:= 
\sup_{\tau' \in [0,\tau]}
	\sum_{p=0}^4\| \g_{\tau}^{2p} \upeta \|_{H^{4-p}(\mathfrak{M}_{\tau'})}^2 
		 + 
\sup_{\tau' \in[0,\tau]}
	\sum_{p=0}^3 \|\mathring{F}\g_{\tau}^{ 2p}(\mathscr{J}^{-2}) \|_{H^{4-p}(\mathfrak{M}_{\tau'})}^2 
	\notag  \\
& \ \
	+ 
	\sup_{\tau' \in[0,\tau]}
	\sum_{p=0}^4
	\left[
		\|\mathring{F} \g_{\tau}^{2p} \AngMD^{4-p} \MD \upeta \|_{L^2(\mathfrak{M}_{\tau'})}^2
		+ 
		\left\|
			\sqrt{\mathring{F}} \g_{\tau}^{2p+1} \AngMD^{4-p} \upeta 
		\right\|_{L^2(\mathfrak{M}_{\tau'})}^2
	\right] \notag \\
& \ \	
	+\sup_{\tau'\in[0,\tau]} \|\Lagvort v \|_{H^3(\mathfrak M_{\tau'})}^2 
	+ \sup_{\tau'\in[0,\tau]}\|\mathring{F}\AngMD^4\Lagvort v \|_{L^2(\mathfrak M_{\tau'})}^2\,,
\label{E:NORM}
\end{align}
where we recall the definition~\eqref{E:MATHRINGFDEFINITION} of $\mathring{F}$.
Our main theorem provides a short-time a priori bound for $\norm(\tau)$.

\begin{theorem}[A priori estimates in Sobolev spaces]\label{T:MAIN}
Let the initial particle number density $\mathring n\in H^4(\mathfrak{M})$ satisfy the physical vacuum boundary condition~\eqref{E:VACUUMBOUNDARY}
and the condition~\eqref{E:SMALLDENSITY}. 
If $\upeta^\mu$ is a  smooth solution to~\eqref{E:LAGRNAGIANMAINEVOLUTIONEQUATION}-~\eqref{E:NORMALIZATIONLAGRANGIAN}, then there exists a time $T=T(\norm(0))$ and a constant $C$ such that 
\begin{align}
\norm(\tau) \le C \norm(0), \ \ \tau \in I=[0,T].
\end{align}
\end{theorem}

\begin{remark}
As we explained in Sect.~\ref{SSS:SIMPLE}, the assumption~\eqref{E:SMALLDENSITY} is in not essential, 
but it allows us to simplify the induction scheme that we use to close our estimates.
All of the essential difficulties of the problem remain under the smallness assumption. 
Note that for $\varepsilon<\frac{1}{3},$ 
assumption~\eqref{E:SMALLDENSITY} in particular implies the validity of the condition 
$0\le \mathring f<\frac{1}{3}$, 
which corresponds to the physical requirement that 
the speed of sound is less than the speed of light; 
see~\eqref{E:PROPERNUMBERDENSITYSIZERESTRICTION}.
\end{remark}

\begin{remark}
Although we only study special relativistic fluids in this paper,
our theorem and methods can be easily adapted to 
prove a priori bounds for solutions of~\eqref{E:EULER} 
on a general Lorentzian spacetime $(\tilde{M},\tilde{g})$ with suitable smoothness assumptions on the Lorentzian metric 
$\tilde{g}$. A non-flat metric $\tilde{g}$ 
leads to the presence of additional lower-order source terms arising from 
the covariant derivatives in~\eqref{E:EULER}. 
They can be regarded as lower-order error terms from the point of view of our analysis.\end{remark}

\subsubsection{The energy}
In order to obtain energy estimates for tangential derivatives (horizontal and temporal), it is convenient to 
introduce the following Riemannian metric (for $\mu,\nu=0,1,2,3$):
\begin{align} \label{E:HMETRIC}
h_{\mu\nu} = g_{\mu\nu} + 2 v_\mu v_\nu.
\end{align}
Note that
\begin{align} \label{E:HINVERSEMETRIC}
	(h^{-1})^{\mu \nu} = g^{\mu \nu} + 2 v^{\mu} v^{\nu},
\end{align}
as can easily be checked with the help of \eqref{E:NORMALIZATIONLAGRANGIAN}.
We establish the positivity of $h$ in Lemma~\ref{L:POSITIVITYOFH}.

\begin{definition}[\textbf{High-order energy function}]
\label{D:ENERGY}
We define 
\begin{align} 
\mathcal{E}
= \mathcal{E}(\tau)
:= \sum_{p=0}^4 \mathcal{E}_p(\tau),
\label{E:ENERGY}
\end{align}
where for $p=0,1,2,3,4$, we set
\begin{align}\label{E:energya}
\mathcal{E}_p(\tau) 
& := 
 	\frac{1}{2}
	\int_{\mathfrak{M}_{\tau}}
		\frac{\mathring{F}}{(1-f)^2} \langle \partial_{\tau}^{2p+1}\AngMD^{4-p} \upeta, \, \partial_{\tau}^{2p+1}\AngMD^{4-p} \upeta \rangle_h\notag \\
 &\qquad\qquad
		+ 
		\frac{\mathring{F}^2\mathscr{J}^{-1}}{(1-f)^2}\langle \Lagdiff \partial_{\tau}^{2p}\AngMD^{4-p} \upeta , \Lagdiff \partial_{\tau}^{2p}\AngMD^{4-p} \upeta \rangle_h  
	\, dy \\ 
 & 	\qquad \qquad
 			+
 			\frac{1}{2}
 			\int_{\mathfrak{M}_{\tau}}
 				\mathring{F}^2 \frac{1+f}{(1-f)^3} \mathscr{J}^{-3}\left( \partial_{\tau}^{2p}\AngMD^{4-p} \mathscr{J}\right)^2 
 			\, dy. 
 			\notag
\end{align}
\end{definition}

\begin{remark}
Note that the integrands in $\mathcal{E}$ are positive definite 
due to the positivity of the metric $h$ and the inequality $f<1$, which holds for short time due to the 
assumption~\eqref{E:SMALLDENSITY}. 
We also note that by \eqref{E:CONNECTIONCOMPARISON},
the operator $\Lagdiff$ appearing in \eqref{E:energya} is effectively 
equivalent to the operator $\ITIMESMD$.
We provide the main a priori estimate for $\mathcal{E}$
in Prop.~\ref{P:ENERGYLEMMA}.   
Furthermore, in conjunction with vorticity estimates and elliptic-type estimates for vertical derivatives,
bounds for $\mathcal{E}(\tau)$ will provide bounds for the norm $\norm(\tau)$.
\end{remark}

\subsection{Methodology and outline of the paper}
		 In our proof, we obtain bounds for $\norm(\tau)$ by a combination of energy estimates, 
		 vorticity estimates, and degenerate elliptic-type estimates, generalizing the approach for the non-relativistic Euler
		 equations presented in \cite{CoutandShkoller2012}.
		
In order to prove Theorem~\ref{T:MAIN}, we establish a bound of the form
\begin{align}\label{E:POLYNOMIAL}
\norm(\tau) \leq \mathring{M} + \tau P(\norm(\tau)),
\end{align}
where $P$ is a positive, increasing function of $\norm$ 
and $\mathring{M}$ is a quantity depending only on $\norm(0)$.
This type of polynomial inequality yields the desired a priori bound by a standard continuity argument; see, for instance, Section 4.5 of~\cite{CoutandShkoller2012}.

The building block of our approach is the use of the differential operators $\g_\tau$, $\AngMD$  
that are tangent to the characteristics\footnote{Here, by characteristics, we mean $\partial \mathscr{P}$,
where $\mathscr{P}$ is defined in \eqref{E:POSITIVITYSET}.}
(see Sect.~\ref{SS:NOTATION}). 
A curious aspect of our problem, already present at the non-relativistic level, is the effective scaling between the space and time derivatives: in our
norm $ \norm(\tau)$, two time derivatives scale like one space derivative.
This is caused by the degeneracy of the physical vacuum condition~\eqref{E:PHYSICALVACUUM} and follows from
weighted embedding~\eqref{E:WEIGHTEDEMBEDDING}.
For this reason, our energy estimates are based on the use of the operators
\[
\partial_{\tau}^{2p} \AngMD^{4-p}, \ \ p=0,1,2,3,4
\]
which have the two-to-one scaling.

As Step 1 in Sect.~\ref{S:VORTICITY}, we begin with the relativistic {\em vorticity} estimates.
In contrast to the non-relativistic case studied in \cite{CoLiSh2010,CoutandShkoller2012, JangMasmoudi2015},   
the vorticity of $\g_\tau v$ 
{\em does not} vanish. Instead, due to the formulation~\eqref{E:LAGRANGIANENTHALPYEVOLUTION},
we infer that 
\[
\Lagvort \left(\g_\tau (\Enth v)\right) = 0.
\]
Therefore, the vorticity of $\g_\tau v$ equals an effective source term, which contains top-order terms as shown in 
Lemma~\ref{L:VORTICITYLEMMA}.
Fortunately, these top-order terms contain pure $\tau$-derivatives. Their $\mathring{F}$-weighted norms are bounded by $\varepsilon \norm$,
(with $\varepsilon$ as in \eqref{E:SMALLDENSITY})
 since the horizontal derivatives of $v=\g_\tau\upeta$
appear in the norm $\norm$ with a different homogeneity in $\mathring{F}$ resulting in an estimate of the type:
\[
\|\mathring{F}\g_\tau^{2p}\t^{4-p}v\|_{\LM}^2 \le \|\mathring{F}\|_{\LI} \|\sqrt{\mathring{F}}v\|_{\LM}^2 \lesssim \varepsilon \norm(\tau),
\]
where we use the smallness~\eqref{E:SMALLDENSITY} of $\mathring{F}$.
The details can be found in~Prop.~\ref{P:VORTICITYESTIMATES}.
We also establish unweighted $L^2$ estimates for the spatial components of the vorticity of $\upeta$ up to top order as stated in part (2) of Prop.~\ref{P:VORTICITYESTIMATES}.
We again exploit the structure provided by~\eqref{E:LAGRANGIANENTHALPYEVOLUTION}, which upon applying $\Lagvort$, can be 
thought of as an effective transport equation for the quantity $\Lagvort\g_\tau\upeta$ with controllable source terms.
	
In Sect.~\ref{S:ENERGY}, we establish Step 2 of our strategy. 
We derive energy estimates for $\partial_{\tau} \upeta$, $D \upeta$, and the time-derivatives $\partial_{\tau}$ 
and horizontal-derivatives $\AngMD$ up to top order, as explained above.
A fundamental aspect of our energy identity, see~\eqref{E:ENERGYIDENTITY1}, is the occurrence of a large vorticity source term that 
\emph{can only be controlled because we have already obtained independent estimates for the vorticity} in the previous step.
Secondly, since the natural ``norm'' in the problem is dictated by the Lorentzian metric $\langle\cdot,\cdot\rangle_g$ we conduct our high-order energy estimates using the 
formulation~\eqref{E:LAGRNAGIANMAINEVOLUTIONEQUATION} and contract with respect to $g$. However, quantities of the form $\langle X,X\rangle_g$ are not necessarily positive definite and 
for that reason, we use the Riemannian metric $h$ introduce in \eqref{E:HMETRIC},  
and show that the difference
\[
\langle X,X\rangle_g - \langle X,X\rangle_h
\]
is of lower-order and therefore controllable by the fundamental theorem
of calculus.

Finally, in Sect.~\ref{S:ELLIPTIC}, we carry out Step 3.
That is, we derive control of the higher ``vertical'' derivatives $\partial_3$ of $\upeta$.
This is the most complicated part of the proof, and it relies on the vorticity bounds, the energy estimates, and the Hodge-type elliptic estimate (see Lemma~\ref{L:HODGEEST}),
which allows us to control the full Sobolev norm of a vectorfield 
from knowledge of the Sobolev norms of its vorticity, divergence, and its
tangential trace on the boundary.   
One of the main tasks
is to obtain control over derivatives of $\mbox{\upshape div} \upeta$ involving at least one vertical derivative,  
that is, over terms such as 
$\partial_3 \mbox{\upshape div} \upeta$
$\partial_3^2 \mbox{\upshape div} \upeta$, and their derivatives up to top order with respect to $\partial_{\tau}$. 
Once we have obtained such estimates, 
we can combine them with the vorticity estimates from Step 1  to recover Sobolev estimates for $\upeta$. This step is based on the identity 
$\partial_{\tau} \ln \mathscr{J} = \mbox{\upshape div} \upeta$, together with the use of the Euler equations to algebraically express $\partial_3 \partial_{\tau} \mathscr{J}$
in terms of quantities that already have a bound, and also on an important integration by parts argument already used in~\cite{CoutandShkoller2012, JangMasmoudi2015}, (see~Prop.~\ref{P:ELLIPTIC})  that crucially 
relies on the physical vacuum condition~\eqref{E:PHYSICALVACUUM}
\[
\partial_3 \mathring{F} \sim \text{ const.} < 0
\]
near the vacuum boundary $\lbrace y^3 = 0 \rbrace$.
				
Finally the proof of Theorem~\ref{T:MAIN} follows from establishing inequality~\eqref{E:POLYNOMIAL}, 
which follows from Steps 1--3.
		
It is worth noting that there is a strong analogy between the use of Lagrangian coordinates in the study of the Euler equations
and the use of an \emph{eikonal function}\footnote{Also known as an \emph{optical function}.} in the study
of quasilinear wave equations. Eikonal functions are fundamental objects in nonlinear geometric optics.
They be used to construct a sharp coordinate system, adapted to the true dynamic characteristics,
much like the Lagrangian coordinates employed in the present article.
Like Lagrangian coordinates, eikonal functions can allow one to derive sharp information about the solution 
that is not readily accessible via standard rectangular coordinates.
The first use of eikonal functions in the context of globally solving a nonlinear wave equation
is found in  \cite{dCsK1993}
on the global stability of Minkowski spacetime as a solution to Einstein's equations.
Eikonal functions have also been used as fundamental ingredients in the proofs
of local well-posedness at low regularity levels (see, for example, \cite{sKiR2003}
and the recent proof of the Bounded $L^2$ Curvature Conjecture \cite{sKiRjS2012})
and in proofs of the formation of shock waves in solutions to quasilinear wave equations 
in more than one space dimension (see, for example,  \cite{dC2007} as well as the additional works \cite{dCsM2012,jS2014,sMpY2014}).

\subsection{History of prior results}

Early developments of the theory of vacuum states can be traced back to~\cite{Lin1987,LiuSmoller1980}. 
One of the few rigorous results pertaining to the non-relativistic compactly supported compressible fluids can be found
in~\cite{Makino1986}, wherein the author constructs solutions in all of $ \mathbb{R}^3  $, but for which it is not possible to track
 the behavior of the vacuum boundary. 
 An extension of that result to relativistic fluids can be found in~\cite{Rendall1992}. 
 For the rich history of vacuum states 
for viscous Euler equations we refers the reader to the introduction of~\cite{CoutandShkoller2012}. 

Major advances in the development of the theory of compressible Euler equations with a free vacuum boundary occurred in the series of works~\cite{CoLiSh2010, CoutandShkoller2011, CoutandShkoller2012}  and independently in~\cite{JangMasmoudi2009, JangMasmoudi2015}, wherein local-in-time well-posedness was established in the Newtonian setting. See also \cite{Oliynyk2012} for the one-dimensional setting.

For existence theorems for the multi-dimensional compressible Euler equations modeling  {\em liquids}, for which the fluid density $\mathring \rho$ is uniformly bounded from below by a strictly positive constant, we refer the reader to \cite{Lindblad2005} and \cite{Trakhinin2009} which employ
a Nash-Moser iteration,  and to  \cite{CoutandHoleShkoller2013} for uniform estimates without derivative loss, as well as  the degenerate limit of  vanishing surface tension.  For a priori estimates for the relativistic liquid, see the preprint \cite{Oliynyk2015}.

The compressible liquid, unlike the compressible gas, with vacuum boundary is 
a uniformly hyperbolic system of conservation laws, and hence does not suffer from the
fundamental degeneracy caused by the physical vacuum boundary.
We refer the reader to the introduction of \cite{CoutandShkoller2012} for a more complete set of references to prior work on the physical vacuum boundary.

Finally, we refer the reader to the recent preprint  \cite{JangLeflochMasmoudi2015}  which employs a different strategy to obtain
a priori estimates for the relativistic Euler equations.

\section{Technical Lemmas}
\label{S:BASICINEQUALITIES}

In this section, we provide some technical lemmas that we use
throughout the remainder of the paper.

\subsection{Hardy-type inequality and Sobolev embeddings}
We shall make use of the following higher-order Hardy-type inequality established in \cite{CoutandShkoller2011,CoutandShkoller2012}.
\begin{lemma}[\textbf{Higher-order Hardy-type inequality}]
\label{L:HARDY}
Let $s\ge1$, $s\in\mathbb N$ and assume that $f\in H^s(\mathfrak{M})\cap \dot H^1_0(\mathfrak{M})$.
Assume that the Euclidean distance function $d(\cdot) =d(\cdot,\g\mathfrak{M}) \in H^r(\mathfrak{M})$, $r=\max(s-1,3)$.
Then $ \frac{f(\cdot)}{d(\cdot,\g\mathfrak{M})}\in H^{s-1}(\mathfrak{M})$ and 
\be\label{E:HARDY}
\left\|
	\frac{f}{d}
\right\|_{H^{s-1}(\mathfrak{M})} \lesssim \|f\|_{H^s(\mathfrak{M})}.
\ee
\end{lemma}	
\begin{definition}
\label{D:WEIGHTEDSOBOLEV} For $k=1$ or $2$, 
the weighted Sobolev spaces $H^1_{d^k}(\mathfrak{M})$ is equipped with the norm
\[
\|f\|_{H^1_{d^k}(\mathfrak{M})}^2:= \int_{\mathfrak{M}} d(y)^k\left(|f(y)|^2 + |\MD f(y)|^2\right)\,dy.
\]
\end{definition}
The following well-known inequality is established in \cite{Kufner}.
\begin{lemma}
\label{L:WEIGHTEDEMBEDDING}
For $k=1$ or $2$, the weighted Sobolev space $H^1_{d^k}(\mathfrak{M})$ 
embeds continuously into the unweighted Sobolev space $H^{1-\frac{k}{2}}(\mathfrak{M})$
and  
\begin{align} \label{E:WEIGHTEDEMBEDDING}
\|f\|_{H^{1-k/2}(\mathfrak{M})} 
\lesssim_{\mathring{M}} \|f\|_{H^1_{d^k}(\mathfrak{M})}.
\end{align}
\end{lemma}

\subsection{Elliptic estimates}

\begin{lemma}[\textbf{Hodge-type elliptic estimate}]
\label{L:HODGEEST}
Let $s \geq 1$ be an integer and let
$Y = \sum_{a=1}^3Y^a \frac{\partial}{\partial x^a}$
be an $\mathbb{R}^3-$valued vectorfield defined along
$\mathfrak{M}$.
Then the following Hodge-type estimate holds:
\begin{align} \label{E:HODGEEST}
\| Y \|_{H^s(\mathfrak{M})} 
\lesssim 
\| Y \|_{L^2(\mathfrak{M})} 
+ \|\Flatdiv Y \|_{H^{s-1}(\mathfrak{M})} 
+ \| \Flatvort Y\|_{H^{s-1}(\mathfrak{M})} 
+ \sum_{a=1}^2 \|\AngMD Y^a \|_{H^{s-3/2}(\partial \mathfrak{M})}.
\end{align}
\end{lemma}

\begin{proof}
	See Proposition 6.2 in~\cite{CoutandShkoller2012}.
\end{proof}

\subsection{Trace Estimates}

\begin{lemma}[\textbf{Tangential trace inequality}]
	\label{L:TANGENTIALTRACE}
	Let $Y$ be as in Lemma~\ref{L:HODGEEST}. Then the following trace estimate holds: 
	\begin{align} \label{E:TANGENTIALTRACE}
		\sum_{a=1}^2 \|\AngMD Y^a \|_{H^{-0.5}(\partial \mathfrak{M})}
		& \lesssim
			\|\Flatvort Y \|_{L^2(\mathfrak{M})} 
			+ 
			\|\AngMD Y \|_{L^2(\mathfrak{M})}.
	\end{align}
\end{lemma}

\begin{proof}
See Lemma 4 in~\cite{CoutandShkoller2012}.
\end{proof}

\section{Bootstrap Assumptions and Basic Estimates}
\label{SS:BOOTSTRAPASSUMPTIONSANDBASICESTIMATES}

\subsection{Bootstrap assumptions}
\label{SS:BOOTSTRAPASSUMPTIONS}
In the rest of the paper, we assume that 
$\upeta:[0,T]\times\mathfrak{M}\to\mathbb R^{1+3}$ is a smooth solution to~\eqref{E:INITIALFUNCTTION}-~\eqref{E:NORMALIZATIONLAGRANGIAN}.
Moreover we assume that the following bootstrap assumptions hold on $[0,T]\times\mathfrak{M}$:
\begin{subequations}
\begin{align} 
		f & \leq \frac{1}{8},
			\label{E:BOOTSTRAPNUMBERDENSITY} \\
		\|\upeta\|_{H^{3.5}(\mathfrak{M}_\tau)} 
		&  \le 2\|\mathring\upeta\|_{H^{3.5}(\mathfrak{M})} + 1,
		\label{E:BOOTSTRAPETA}
			\\
			\|\g_\tau^a v\|_{H^{3-\frac a2}(\mathfrak{M}_\tau)} & \le \|\g_\tau^a v|_{\tau=0}\|_{H^{3-\frac a2}(\mathfrak{M})} + 1, \label{E:VELOCITYBOOTSTRAP}\\
		\mathring{v}^0 - \frac{1}{2}
		& \leq \mathscr{J} 
			\leq 
			\mathring{v}^0 + \frac{1}{2},
			\label{E:BOOTSTRAPJACOBIANPOSITIVITY} \\
		1/4 + (\mathring{v}^1)^2 + (\mathring{v}^2)^2 
		& \leq g^{\mu \nu} a^3_{\mu} a^3_{\nu} 
			\leq 2 + (\mathring{v}^1)^2 + (\mathring{v}^2)^2.
			\label{E:BOOTSTRAPCOF3ONEFORMPOSITIVITY}
\end{align}
\end{subequations}

We use the bootstrap assumptions 
\eqref{E:BOOTSTRAPNUMBERDENSITY}-\eqref{E:BOOTSTRAPCOF3ONEFORMPOSITIVITY}
throughout the paper, 
often incorporating them into the generic data-dependent parameter
{\color{DarkOrchid} $\mathring{M}$}. 
Once the energy estimates are established, it is easy to
derive strict improvements of the bootstrap assumptions; 
see Lemma~\ref{L:IMPROVMENTOFBOOTSTRAP}.

\subsection{Positivity of the metric $h$}
\label{SS:HISPOSITIVEDEFINITE}

We now establish quantitative positivity estimates for the metric $h$ defined in \eqref{E:HMETRIC}.

\begin{lemma}[\textbf{Positivity of the Riemannian metric}]
	\label{L:POSITIVITYOFH}
	Recall that $\D$ denotes the flat connection of the Minkowksi metric
	and that $\ITIMESMD$ denotes the Lagrangian coordinate spacetime gradient.
	Let $h$ be the Riemannian metric defined in \eqref{E:HMETRIC},
	and let $X$ and $Y$ be vectorfields. There exists a constant $\mathring{M} > 0$,
	depending on the data, such that under the bootstrap assumptions of Sect.~\ref{SS:BOOTSTRAPASSUMPTIONS}, 
	we have
	\begin{subequations}
	\begin{align}  \label{E:HMETRICPOSITIVITY}
		\frac{1}{\mathring{M}}
		\sum_{\alpha = 0}^3 |X^{\alpha}|^2 
		& \leq
		\langle X \rangle_h
		:= h_{\alpha \beta} X^{\alpha} X^{\beta} 
		\leq \mathring{M}
		\sum_{\alpha = 0}^3 |X^{\alpha}|^2,
			\\
		\frac{1}{\mathring{M}}
		|\ITIMESMD Y| 
		& \leq
		 |\Lagdiff Y|
		\leq
		\mathring{M}
		|\ITIMESMD Y|.
		\label{E:CONNECTIONCOMPARISON}
	\end{align}
	\end{subequations}
\end{lemma}

\begin{proof}
Let $\underline X = (X_1,X_2,X_3)$ and $\underline v = (v_1,v_2,v_3)$.
Note that by definition
\begin{align*}
\langle X \rangle_h & = \langle X\rangle_g + 2 \langle v,X\rangle_g^2
= - X_0^2 + |\underline X|^2 + 2(-X_0v_0 + \underline X\cdot \underline v)^2 \\
& = X_0^2 (1+2|\underline v|^2) + |\underline X|^2 - 2X_0\sqrt{1+|\underline v|^2}\underline X\cdot\underline v +2 (\underline X\cdot \underline v)^2.
\end{align*}
Setting $\varphi:= \frac{\underline X}{|\underline X|}\in \mathbb \mathbb{S}^2$ and $t = \frac{X_0}{|\underline X|}$ we can rewrite $\langle X\rangle_h$ in the following way:
\begin{align}
\langle X \rangle_h = |\underline X|^2 \left( t^2(1+|v|^2) + 1 +2(\varphi\cdot \underline v)^2 - 2 t\sqrt{1+|\underline v|^2}\varphi\cdot \underline v)\right)
\end{align}
It is an easy exercise to check that the discriminant of the quadratic function in the parenthesis above is uniformly-in-$\varphi$ bounded from below by $0$.
In particular, the function has a uniform lower bound and we conclude that $\langle X \rangle_h \ge C |\underline X|^2$ for some $C$.
To obtain the lower bound $\langle X \rangle_h \ge C X_0^2$, 
we use similar analysis, but we divide by $X_0$
instead of $|\underline X|$.
The upper bound in~\eqref{E:HMETRICPOSITIVITY} is straightforward 
to derive and relies on the bootstrap assumption~\eqref{E:VELOCITYBOOTSTRAP}.

To prove the first inequality in \eqref{E:CONNECTIONCOMPARISON},
we first use \eqref{E:CHAINRULES} to deduce that
$$|\ITIMESMD Y|
\lesssim \sum_{A=0}^3 |\partial_A Y|
\lesssim \sum_{A=0}^3 \sum_{\alpha = 0}^3 |\partial_A \upeta^{\alpha}||\Lagdiff Y|\,.
$$
We then obtain the desired bound
$
|\partial_K \upeta^{\alpha}|
\leq \mathring{M}
$
from the bootstrap assumptions.
Similarly, to obtain the second inequality
in \eqref{E:CONNECTIONCOMPARISON},
we use \eqref{E:CHAINRULES} to find that
$|\Lagdiff Y| \lesssim \sum_{\alpha = 0}^3 |\frac{\partial}{\partial x^{\alpha}} Y|
\lesssim \sum_{A=0}^3 \sum_{\alpha = 0}^3 |\mathscr{A}_{\alpha}^A \partial_A Y|
\lesssim \sum_{A=0}^3 \sum_{\alpha = 0}^3 |\mathscr{A}_{\alpha}^A| |\ITIMESMD Y| \,,
$
from which it follows that
$|\mathscr{A}_{\alpha}^A| \leq \mathring{M}$
by using the bootstrap assumptions,
since $\mathscr{A}$ is the inverse of the $4 \times 4$ matrix
$\ITIMESMD \upeta$, which has bounded entries and
a strictly positive Jacobian determinant.
\end{proof}

\subsection{Improvement of the bootstrap assumptions}
	
	The next lemma becomes relevant at the end of the paper,
	after we have derived our main a priori estimate for
	$\norm(\tau)$ under the bootstrap assumptions of Sect.~\ref{SS:BOOTSTRAPASSUMPTIONS}.
	The lemma supplies improvements of the bootstrap assumptions, thus closing the proof.
	
	\begin{lemma}[\textbf{Improvement of the bootstrap assumptions}]
	\label{L:IMPROVMENTOFBOOTSTRAP}
	Assume that the bootstrap assumptions of Sect.~\ref{SS:BOOTSTRAPASSUMPTIONS} and the
	estimate $\norm(\tau) \leq \mathring{M}$
	all hold for $\tau$ sufficiently small. 
	Then, after possibly shrinking the allowable
	smallness of $\tau$, we have the following
	improvements of the bootstrap assumptions
	\begin{subequations}
	\begin{align} 
		f & \leq \frac{1}{6},
			\label{E:NUMBERDENSITYNOTTOOBIG} \\
			\|\upeta\|_{H^{3.5}(\mathfrak{M}_\tau)} 
		&  \le 2\|\mathring\upeta\|_{H^{3.5}(\mathfrak{M})} + \frac{1}{2},
		\label{E:BOOTSTRAPETABETTER}
			\\
			\|\g_\tau^a v\|_{H^{3-\frac a2}(\mathfrak{M}_\tau)} & \le \|\g_\tau^a v|_{\tau=0}\|_{H^{3-\frac a2}(\mathfrak{M})} + \frac{1}{2}, \label{E:VELOCITYBOOTSTRAPBETTER}\\
\label{E:JACOBIANPOSITIVITY}
		\mathring{v^0} - \frac{1}{4} 
		& \leq \mathscr{J} 
			\leq 
			\mathring{v^0} + \frac{1}{4},
			\\
		\frac{1}{2} + (\mathring{v}^1)^2 + (\mathring{v}^2)^2 
		& \leq g^{\mu \nu} a^3_{\mu} a^3_{\nu} 
			\leq \frac3{2} + (\mathring{v}^1)^2 + (\mathring{v}^2)^2.
			\label{E:COF3ONEFORMPOSITIVITY}
	\end{align}
	\end{subequations}
\end{lemma}

\begin{proof}
To prove \eqref{E:NUMBERDENSITYNOTTOOBIG}, we first recall
our smallness assumption~\eqref{E:SMALLDENSITY}:
$\mathring f \leq \varepsilon << \frac{1}{3}$. Since $f= \mathring f \mathring v^0 \mathscr{J}^{-1}$,
we have $|f-\mathring f| = \mathring f (1-\mathring v^0\mathscr{J}^{-1})$,
and we can apply the fundamental theorem of calculus to 
deduce $|1-\mathring v^0\mathscr{J}^{-1}|\leq \tau \mathring{M}$. 
The estimates~\eqref{E:BOOTSTRAPETABETTER}-~\eqref{E:JACOBIANPOSITIVITY} follow in an analogous way, where we recall~\eqref{E:INITIALJACOBIANDETERMINANT}.
To prove \eqref{E:COF3ONEFORMPOSITIVITY},
we note that $\Cof_{\mu}^3$ is a sum of cubic terms of the schematic form 
$(\partial_{\tau} \upeta) \cdot \AngMD \upeta \cdot \AngMD \upeta$.
Hence, using Sobolev embedding $H^2(\mathfrak{M}) \hookrightarrow L^{\infty}(\mathfrak{M})$,
we see that 
$\| \partial_{\tau}(g^{\mu \nu} a^3_{\mu} a^3_{\nu}) \|_{L^{\infty}(\mathfrak{M}_{\tau})}
\leq P(\mathring{M})
$.
Thus, by the fundamental theorem of calculus, we have
$
\left|
	g^{\mu \nu} a^3_{\mu} a^3_{\nu}
	- 
	g^{\mu \nu} a^3_{\mu} a^3_{\nu}|_{\tau=0}
\right|
\leq \tau P(\mathring{M})
$.
Moreover, using~\eqref{E:INITIALINVERSECHOVMATRIX} and~\eqref{E:INITIALJACOBIANDETERMINANT},
we compute that
$g^{\mu \nu} a^3_{\mu} a^3_{\nu}|_{\tau=0} 
= 1 + (\mathring{v}^1)^2 + (\mathring{v}^2)^2$.
The desired estimate \eqref{E:COF3ONEFORMPOSITIVITY} now easily follows.
\end{proof}

\subsection{Simple estimates relying on the fundamental theorem of calculus}
\label{SS:FTCESTIMATES}

Throughout the paper, we often bound simple error terms without giving full
details, mentioning only that we are using ``the fundamental theorem of calculus.''
In this section, we briefly describe what we mean by this.
When we say that we are bounding a quantity $Q$ in some norm $\| \cdot \|$
using the fundamental theorem of calculus,
we mean that the standard Sobolev calculus 
(which in the present article relies on Holder's inequality,
interpolation,
and the non-degenerate Sobolev embeddings
$H^2(\mathfrak{M}) \hookrightarrow L^{\infty}(\mathfrak{M})$,
$H^1(\mathfrak{M}) \hookrightarrow L^6(\mathfrak{M})$,
$H^1(\mathfrak{M}) \hookrightarrow L^4(\mathfrak{M})$,
$H^{.5}(\mathfrak{M}) \hookrightarrow L^3(\mathfrak{M})$)
implies that 
$\| \partial_{\tau} Q(\tau) \| \leq P(\norm(\tau))$,
where $\norm$ is the norm defined in \eqref{E:NORM}.
Thus, integrating in time, we obtain 
$\| \partial_{\tau} Q(\tau) \| \leq \mathring{M} + \tau P(\norm(\tau))$.
A typical term that can bounded in this way is one that
is below-top-order in the sense of the norm $\norm$.

As examples of quantities that can be bounded in this way, we cite:
$\| \g_{\tau}^{2p} \upeta \|_{H^{3-p}(\mathfrak{M}_{\tau})}$ for $0 \leq p \leq 3$,
$\| \g_{\tau} \upeta \|_{H^3(\mathfrak{M}_{\tau})}$,
$\| \sqrt{\mathring{F}} \g_{\tau}^{2p} \AngMD^{3-p} \upeta \|_{L^2(\mathfrak{M}_{\tau})}$
for $0 \leq p \leq 3$,
$\|\mathring{F}\g_{\tau}^{2p}(\mathscr{J}^{-2}) \|_{H^{3-p}(\mathfrak{M}_{\tau})}$
and $\|\mathring{F}\g_{\tau}^{2p} \mathscr{J} \|_{H^{3-p}(\mathfrak{M}_{\tau})}$
for $0 \leq p \leq 3$,
and 
$\| \g_{\tau}^{2p} \MD^{2-p} \upeta \|_{L^{\infty}(\mathfrak{M}_{\tau})}$ for 
$0 \leq p \leq 2$.

\subsection{Estimates involving contractions against the four-velocity}
\label{SS:FOURVELOCITYESTIMATES}

Recall that 	
$\underline{\upeta} = (\upeta^1,\upeta^2,\upeta^3)$
denotes the projection of the flow map $\upeta$ onto $\mathbb{R}^3$. 
We use the next lemma to show that the top-order derivatives of $\upeta^0$
are controlled by corresponding top-order derivatives of
$\underline{\upeta}$.
The main idea of the proof is that the constraint \eqref{E:NORMALIZATIONLAGRANGIAN}
yields improved estimates for the high derivatives of $\upeta^{\alpha}$
whenever it is \emph{contracted} against the (undifferentiated) one-form $v_{\alpha}$.

\begin{lemma}
	\label{L:CONTRACTIONAGAINSTUFIRSTSOBOLEVESTIMATE}
	Under the bootstrap assumptions of Sect.~\ref{SS:BOOTSTRAPASSUMPTIONS},
	we have the following estimates 
	for $\tau \in [0,T]$:	
	\begin{subequations}
	\begin{align} 
		\sum_{p=0}^4 
			\| v_{\alpha} \partial_{\tau}^{2p} \ITIMESMD \upeta^{\alpha} \|_{H^{4-p}(\mathfrak{M}_{\tau})}^2
		& \leq
			P(\norm(\tau)),
			\label{E:NOTTIMEINTEGRATEDCONTRACTIONAGAINSTUFIRSTSOBOLEVESTIMATE} \\
			\sum_{p=0}^4 
			\| \sqrt{\mathring{F}} v_{\alpha} \partial_{\tau}^{2p+1} \ITIMESMD \upeta^{\alpha} \|_{H^{4-p}(\mathfrak{M}_{\tau})}^2
			& \leq
			P(\norm(\tau)),
			 \label{E:NOTTIMEINTEGRATEDCONTRACTIONAGAINSTUSECONDSOBOLEVESTIMATE} \\
		\sum_{p=0}^4 
			\| v_{\alpha} \partial_{\tau}^{2p} \upeta^{\alpha} \|_{H^{4-p}(\mathfrak{M}_{\tau})}^2
		& \leq
		\mathring{M}
		+
		C \tau \norm(\tau),
			\label{E:CONTRACTIONAGAINSTUFIRSTSOBOLEVESTIMATE} \\
		\sum_{p=0}^4 \| \sqrt{\mathring{F}} v_{\alpha} \g_{\tau}^{2p+1} \AngMD^{4-p} \upeta^{\alpha} \|_{L^2(\mathfrak{M}_{\tau})}^2
		& \leq
		\mathring{M}
		+
		C \tau \norm(\tau),
			\label{E:WEIGHTEDCONTRACTIONAGAINSTUFIRSTSOBOLEVESTIMATE} \\
		\sum_{p=0}^4 \|\mathring{F} v_{\alpha} \g_{\tau}^{2p} \AngMD^{4-p} \MD \upeta^{\alpha} \|_{L^2(\mathfrak{M}_{\tau})}^2
		& \leq
		\mathring{M}
		+
		C \tau \norm(\tau),
			\label{E:TOPORDERWEIGHTEDCONTRACTIONAGAINSTUFIRSTSOBOLEVESTIMATE} \\
	\| 
		\partial_{\tau}^{2p} \upeta^0
	\|_{H^{4-p}(\mathfrak{M}_{\tau})}
	& \leq
		\mathring{M} 
		+ \tau P(\norm(\tau))
		+
		\sum_{a=1}^3
		\| 
			\g_{\tau}^{2p} \upeta_a
		\|_{H^{4-p}(\mathfrak{M}_{\tau})}.
	\label{E:0COMPONENTINTEMRSOFSPATIALCOMPONENT}
	\end{align}
	\end{subequations}
\end{lemma}

\begin{proof}
	To prove \eqref{E:CONTRACTIONAGAINSTUFIRSTSOBOLEVESTIMATE}, we
	first differentiate \eqref{E:NORMALIZATIONLAGRANGIAN} with
	$\partial_{\tau}^{2p} \MD^{4-p}$
	and obtain
	\begin{align} 
		\left|
			\partial_{\tau}(v_{\alpha} \partial_{\tau}^{2p} \MD^{4-p} \upeta^{\alpha})
		\right|
		& \lesssim 
			\left|
				\partial_{\tau}^2 \upeta_{\alpha} 
			\right|
			\left|
				\partial_{\tau}^{2p} \MD^{ 4-p} \upeta^{\alpha}
			\right|
			+
			\mathop{\sum_{p_1 + p_2 \leq 2p-1}}_{q_1 + q_2 \leq 3-p}
			\left|
				\partial_{\tau}^{p_1 + 1} \MD^{q_1} \upeta
			\right|
			\left|
				\partial_{\tau}^{p_2 + 1} \MD^{q_2} \upeta
			\right|.
			\label{E:PARTIALTUCONTRACTEDPOINTWISE}
	\end{align}
	From 
	\eqref{E:PARTIALTUCONTRACTEDPOINTWISE},
	definition \eqref{E:NORM} and  the Sobolev embedding
	$H^2(\mathfrak{M}) \hookrightarrow L^{\infty}(\mathfrak{M})$,
	we deduce that 
	$
		\|  \partial_{\tau}(v_{\alpha} \partial_{\tau}^{ 2p} \MD^{ 4-p} \upeta^{\alpha}) \|_{L^2(\mathfrak{M}_{\tau})}
		\leq 
		\mathring{M}
		+ C \norm(\tau)
	$.
	Integrating in time and using the previous estimate, 
	we deduce
	$
		\|
			 v_{\alpha} \partial_{\tau}^{ 2p} \MD^{4-p} \upeta^{\alpha} 
		\|_{L^2(\mathfrak{M}_{\tau})}^2
	\leq \mathring{M} \tau + C \tau \norm(\tau)
	$,
	from which the desired estimate \eqref{E:CONTRACTIONAGAINSTUFIRSTSOBOLEVESTIMATE} readily follows.
	The proofs of
	\eqref{E:WEIGHTEDCONTRACTIONAGAINSTUFIRSTSOBOLEVESTIMATE} 
	and 
	\eqref{E:TOPORDERWEIGHTEDCONTRACTIONAGAINSTUFIRSTSOBOLEVESTIMATE} are similar,
	and we omit the details.
	The proof of \eqref{E:NOTTIMEINTEGRATEDCONTRACTIONAGAINSTUFIRSTSOBOLEVESTIMATE} also 
	is similar but does not involve a time integration; we omit the details.
	
To prove \eqref{E:0COMPONENTINTEMRSOFSPATIALCOMPONENT}, we first decompose 
$\partial_{\tau}^{2p} \MD^{4-p} \upeta^0
= - \frac{1}{v^0} v_{\alpha} \partial_{\tau}^{2p} \MD^{4-p} \upeta^{\alpha}
+ \frac{v_a}{v^0} \partial_{\tau}^{2p} \MD^{4-p} \upeta^a
$.
From \eqref{E:NORMALIZATIONLAGRANGIAN}, we deduce
$\| 
	\frac{1}{v^0} 
\|_{L^{\infty}(\mathfrak{M}_{\tau})}
\leq 1
$
and
$
\left\|
	\frac{v^a}{v^0}
\right\|_{L^{\infty}(\mathfrak{M}_{\tau})}
< 1
$.
The desired bound \eqref{E:0COMPONENTINTEMRSOFSPATIALCOMPONENT}
follows easily from these estimates and
\eqref{E:CONTRACTIONAGAINSTUFIRSTSOBOLEVESTIMATE}.
\end{proof}

\section{Vorticity estimates}\label{S:VORTICITY}

In this section, we use the special structure of
equation \eqref{E:LAGRANGIANENTHALPYEVOLUTION} to
derive estimates for the vorticity up to top order.
It is important for our strategy that the vorticity estimates
can be obtained independently of our estimates for other quantities.

We derive our main vorticity estimates in Prop.~\ref{P:VORTICITYESTIMATES}.
We start with a preliminary lemma in which we obtain estimates
for the vorticity of $\partial_{\tau} v$.

\begin{lemma}\label{L:VORTICITYLEMMA}
Let $\delta > 0$ be a constant and let $\varepsilon$ be as in \eqref{E:SMALLDENSITY}.
Under the bootstrap assumptions of Sect.~\ref{SS:BOOTSTRAPASSUMPTIONS},
we have the following estimates for $\tau \in [0,T]$ 
(where $\mathring{M}$ is allowed to depend on $\delta^{-1}$
and $\mu,\nu = 0,1,2,3$):
\begin{subequations}
\begin{align} 
\sum_{p=0}^3 
\left\|
	\mathring{F} \g_\tau^{2p}\AngMD^{3-p} (\Lagvort \g_{\tau}v)_{\mu\nu}
\right\|_{L^2(\mathfrak{M}_{\tau})}^2  
& \leq \mathring{M} +(\delta + C \varepsilon)\norm(\tau) + \tau P(\norm(\tau)) \,,  \label{hss1}\\
\sum_{p=0}^2
\left\|
	\g_{\tau}^{2p}(\Lagvort \g_{\tau}v)_{\mu\nu} 
\right\|_{H^{2-p}(\mathfrak{M}_{\tau})}^2
& \leq \mathring{M} + \tau P(\norm(\tau)) \,.  \label{hss2}
\end{align} 
\end{subequations}
\end{lemma}

\begin{proof}
We start by computing the vorticity of LHS~\eqref{E:LAGRANGIANENTHALPYEVOLUTION}.

Using the definition of $\Lagvort_{\mu\nu}$ given in \eqref{E:LAGRANGIANVORTICITY}, we find that
$\Lagvort_{\mu\nu}$
annihilates $\Lagdiff \Enth$ since $\Enth$ is a perfect gradient.
That is, 
\begin{align*}
\left( \Lagvort \Lagdiff \Enth \right)_{\mu\nu}
& = \left(\Lagvort\left( \mathscr{A}^K_{\bullet}\g_K \Enth\right)\right)_{\mu\nu}
= 0.
\end{align*}
Therefore, from~\eqref{E:LAGRANGIANENTHALPYEVOLUTION}, we obtain
\begin{align}
\label{E:VORTICITY1}
\left(\Lagvort \left(\g_{\tau}\left(\Enth v\right)\right)\right)_{\mu\nu} = 0.
\end{align}
We now want to rewrite 
$\left(\Lagvort \g_{\tau}\left(\Enth v\right)\right)_{\mu\nu}$ as 
$\Enth \left(\Lagvort \g_{\tau}v\right)_{\mu\nu}$ plus a remainder, which should be thought of as a lower-order error term.
A simple calculation 
based on using equation \eqref{E:LAGRANGIANENTHALPYEVOLUTION}
to substitute for $\mathscr{A}^K_{\mu} \partial_K \Enth$
yields the identity
\begin{align}\label{E:VORTICITY2}
\left(\Lagvort \g_{\tau}\left(\Enth v\right)\right)_{\mu\nu} 
& = \Enth (\Lagvort \g_{\tau} v)_{\mu\nu} + \g_{\tau}\Enth (\Lagvort v)_{\mu\nu} 
 + \mathscr{A}_{\mu}^K \partial_K \Enth \g_{\tau}v^{\nu}
				-  \mathscr{A}_{\nu}^K \partial_K \Enth \g_{\tau}v_{\mu} \\
				& +2\g_{\tau} \Enth\left(\g_{\tau}v_{\nu}v_{\mu}-\g_{\tau}v_{\mu}v_{\nu}\right) + \Enth \left(\g_{\tau}^2v_{\nu}v_{\mu}-\g_{\tau}^2v_{\mu}v_{\nu}\right). \notag
				\end{align}
Combining~\eqref{E:VORTICITY1} and~\eqref{E:VORTICITY2}
and again using equation \eqref{E:LAGRANGIANENTHALPYEVOLUTION}
to substitute for the third and fourth products on right-hand side of~\eqref{E:VORTICITY2},
we obtained the desired expression
\begin{align}
(\Lagvort \g_{\tau}v)_{\mu\nu}  = &  - \g_{\tau}\Enth (\Lagvort v)_{\mu\nu}  
 +\g_\tau(\Enth v_\mu) \g_{\tau}v_{\nu}\Enth^{-1}- \g_\tau(\Enth v_\nu) \g_{\tau}v_{\mu}\Enth^{-1} \notag\\
 & -2\Enth^{-1}\g_{\tau} \Enth\left(\g_{\tau}v_{\nu}v_{\mu}-\g_{\tau}v_{\mu}v_{\nu}\right)  
+ \g_{\tau}^2v_{\nu}v_{\mu}-\g_{\tau}^2v_{\mu}v_{\nu}. \label{E:VORTICITYBASIC}
\end{align}

{\em Proof of \eqref{hss1}.}
Let $p \in \lbrace 0,1,2,3 \rbrace$.
We apply $\partial_{\tau}^{2p}\AngMD^{3-p}$ to equation \eqref{E:VORTICITYBASIC},
multiply by $\mathring{F}$,
take the norm $\| \cdot \|_{L^2(\mathfrak{M}_{\tau})}$ of both sides, and square.
We start by bounding the term generated by the first product on right-hand side of~\eqref{E:VORTICITYBASIC}.
Note that $\g_\tau \Enth = - 2 \Enth^{3/2}\mathscr{J}^{-2}\mathring{F} \g_\tau \mathscr{J}$.
Thus, when $\partial_{\tau}^{2p}\AngMD^{3-p}$ falls onto $\g_\tau \Enth$, the highest-order term thus arising is $\g_\tau^{2p+1}\AngMD^{3-p}\mathscr{J}$.
The resulting product
$\Enth^{3/2}\mathscr{J}^{-2}\mathring{F} \g_\tau^{2p+1}\AngMD^{3-p}\mathscr{J}(\Lagvort v)_{\mu\nu} $
is below top-order and thus, using the fundamental theorem of calculus as described in
Sect.~\ref{SS:FTCESTIMATES}, we can bound it
in the norm 
$\| \cdot \|_{L^2(\mathfrak{M}_{\tau})}$
by 
$\leq \mathring{M} + \tau P(\norm(\tau))$
as desired.

Using similar reasoning, we can bound the product that arises when all derivatives
$\g_\tau^{2p}\AngMD^{3-p}$ fall on $(\Lagvort v)_{\mu\nu} $ since the top-order terms 
scale like $\mathring{F} \mathscr{A} \g_\tau^{2p}\AngMD^{3-p} \ITIMESMD v$ 
or $\mathring{F} \g_\tau^{2p}\AngMD^{3-p} \ITIMESMD \upeta v$.
The below top-order terms can be bounded 
in the norm $\| \cdot \|_{L^2(\mathfrak{M}_{\tau})}$
by $\leq \mathring{M} + \tau P(\norm(\tau))$
via the fundamental theorem of calculus.
We can use essentially the same reasoning to bound 
the $\g_\tau^{2p}\AngMD^{3-p}$ derivative of
all remaining products on right-hand side of~\eqref{E:VORTICITYBASIC}, except for the last two.
The $\g_\tau^{2p}\AngMD^{3-p}$ derivative of each of the last two terms
can be bounded in the same way, so we focus only on the last one. 
To proceed, we note that 
for $p\in\{0,1,2,3\}$,
the top-order term generated by the Leibniz expansion of 
$\g_\tau^{2p}\AngMD^{3-p}(\g_{\tau}^2v_{\nu}v_{\mu})$ 
is  
$\g_\tau^{2p+2}\AngMD^{3-p}v_\nu v_\mu$.  
We now use the smallness assumption~\eqref{E:SMALLDENSITY},
Sobolev embedding, 
and the definition~\eqref{E:NORM} of the norm $\norm$
to deduce the desired bound as follows:
\begin{align*}
&\left\|	
	\mathring{F}\g_\tau^{2p+2}\AngMD^{3-p}v_\nu v_\mu
\right\|_{L^2(\mathfrak{M}_{\tau})}^2 \\
& \qquad\qquad\qquad
\leq 
\left\| 
	\mathring{F}
\right\|_{L^{\infty}(\mathfrak{M})}
\left\|
	\sqrt{\mathring{F}}\g_\tau^{2p+2}\AngMD^{3-p}v_\nu
\right\|_{L^2(\mathfrak{M}_{\tau})}^2
\left\|
	v_\nu
\right\|_{L^{\infty}(\mathfrak{M}_{\tau})}^2
\lesssim_{\mathring{M} + \tau P(\norm(\tau))} \varepsilon \norm(\tau).
\end{align*}
All of the remaining terms arising from the Leibniz 
expansion of $\g_\tau^{2p+2}\AngMD^{3-p} (v_\nu v_\mu)$
are lower-order terms that, by virtue of the fundamental theorem of calculus,
are bounded in the norm $\| \cdot \|_{L^2(\mathfrak{M}_{\tau})}$
by $\leq \mathring{M} + \tau P(\norm(\tau))$. 
This concludes our proof of \eqref{hss1}.
\\

\noindent
{\em Proof of \eqref{hss2}.}
We apply $\partial_{\tau}^{2p} \MD^{2-p}$ to equation \eqref{E:VORTICITYBASIC}.
All terms on the right-hand side of the resulting equation are below top-order.
Thus, using the fundamental theorem of calculus as described in
Sect.~\ref{SS:FTCESTIMATES},
we can bound them in the norm 
$\| \cdot \|_{L^2(\mathfrak{M}_{\tau})}$
by 
$\leq \mathring{M} + \tau P(\norm(\tau))$
as desired.
\end{proof}

We now derive our main estimates for the vorticity.
\begin{proposition}\label{P:VORTICITYESTIMATES}
Let $\delta > 0$ be a constant and let $\varepsilon$ be as in \eqref{E:SMALLDENSITY}.
Under the bootstrap assumptions of Sect.~\ref{SS:BOOTSTRAPASSUMPTIONS},
we have the following estimates for $\tau \in [0,T]$
(where we recall that $\underline{\upeta} = (\upeta_1,\upeta_2,\upeta_3)$
and $\mathring{M}$ is allowed to depend on $\delta^{-1}$):
\begin{subequations}
\begin{align}
\left|
	\int_{\mathfrak{M}_\tau}
		\mathring{F}^2 
		\langle \Lagvort \, \partial_{\tau}^{2p}\AngMD^{4-p}\upeta,\Lagvort \, \partial_{\tau}^{2p}\AngMD^{4-p} \upeta\rangle_g 
	\, dy
\right| 
& \leq \mathring{M} + \left(\delta + C \varepsilon\right)\norm(\tau)+ \tau P(\norm(\tau)),
	\label{E:MAINVORTICITYESTIMATES} 
	\\
\|\Flatvort \g_{\tau}^{2p} \underline{\upeta}\|_{H^{3-p}(\mathfrak{M}_{\tau})}^2 	
& \leq \mathring{M} + \tau P(\norm(\tau)).
\label{E:FLATMAINVORTICITYESTIMATES}
\end{align}
\end{subequations}
\end{proposition}

\begin{proof}
{\bf Proof of \eqref{E:MAINVORTICITYESTIMATES}.}
We first prove the estimate in the case $p=0$. 
We start by noting that
$
\mathring{F} \Lagvort \AngMD^4 \upeta
=
\mathring{F} \AngMD^3 \Lagvort \AngMD \upeta
+ O_{L^2(\mathfrak{M}_{\tau})}(\mathring{M} + \tau P(\norm(\tau)))
$
since the difference
$\mathring{F} \Lagvort \AngMD^4 \upeta
- \mathring{F} \AngMD^3 \Lagvort \AngMD \upeta
$
is easy to bound in the norm
$\| \cdot \|_{L^2(\mathfrak{M}_{\tau})}$
via the fundamental theorem of calculus, as we described in
Sect.~\ref{SS:FTCESTIMATES}.
Thus, to prove \eqref{E:MAINVORTICITYESTIMATES} in the case $p=0$,
it suffices to bound
$
\int_{\mathfrak{M}_\tau}\mathring{F}^2 \langle \AngMD^4 \Lagvort \upeta,\AngMD^3 \Lagvort \AngMD \upeta \rangle_g \, dy 
$
by $\leq \mbox{right-hand side of~\eqref{E:MAINVORTICITYESTIMATES}}$.
To proceed, we will use the following identity, 
obtained from repeated use of commutation identities of the form
$([\partial_{\tau}, \Lagvort] \upeta)_{\mu \nu}
= \partial_{\tau} \mathscr{A}_{\mu}^K \partial_K \upeta_{\nu} 
- \partial_{\tau} \mathscr{A}_{\nu}^K \partial_K \upeta_{\mu} 
$
and the fundamental theorem of calculus
and valid for $A=1,2,3$:
\begin{align}
(\Lagvort \partial_A \upeta)_{\mu\nu}
= & (\Lagvort \partial_A\upeta)_{\mu\nu} \big|_{\tau=0} + \tau \partial_A(\Lagvort v)_{\mu\nu} \big|_{\tau=0}  \notag \\
& + \int_0^{\tau}\left(\g_{\tau}\mathscr{A}^K_{\mu}\g_K\partial_A\upeta_{\nu}-\g_{\tau}\mathscr{A}^K_{\nu}\g_{K}\partial_A\upeta_\mu
- \partial_A \mathscr{A}^K_{\mu}\g_Kv_{\nu}+\partial_A \mathscr{A}^K_{\nu}\g_Kv_{\mu}\right) \,d\tau'\notag \\
& + \int_0^{\tau}\int_0^{\tau'}\partial_A \left(\mathcal{B}_{\mu\nu}(\mathscr{A},\ITIMESMD v)
+ (\Lagvort \g_\tau v)_{\mu \nu} \right)\,d\tau''d\tau' \,,
\label{E:ETAVORTICITY}
\end{align}
where 
\begin{align} \label{E:BVORTICITYERRORTERM}
 B_{\mu\nu}(\mathscr{A},\ITIMESMD v) := -\g_{\tau}\mathscr{A}^L_{\mu}\g_Lv^{\nu} + \g_{\tau}\mathscr{A}^L_{\nu}\g_Lv_{\mu}.
\end{align}

We now apply $\AngMD^3$ to \eqref{E:ETAVORTICITY} and consider the relevant cases
$A=1,2$, our goal being to show that all terms $Q$ on $\AngMD^3 (\mbox{right-hand side of \eqref{E:ETAVORTICITY}})$
verify the bound
$
\left|
\int_{\mathfrak{M}_{\tau}} 
	\mathring{F}^2 \langle Q,Q \rangle_g
\, dy
\right|
\leq \mathring{M} + \tau P(\norm(\tau))$.
Clearly the terms generated by the two products on the first line of right-hand side of~\eqref{E:ETAVORTICITY} satisfy the desired bound.
The terms generated by the single time integral on right-hand side of~\eqref{E:ETAVORTICITY}
can all be treated with the standard Sobolev calculus,
which yields that the corresponding integral 
$
\left|
\int_{\mathfrak{M}_{\tau}} \langle \cdots \rangle_g
\right|
$ 
is
$\leq \mathring{M} + \tau P(\norm(\tau))$.
To handle the terms generated by the double time integral on right-hand side of~\eqref{E:ETAVORTICITY},
we must integrate by parts in time once, as we now explain.

We first consider the double time integral of 
$\AngMD^4 \mathcal{B}_{\mu\nu}$. We need only to consider the top-order terms
since the below-top-order ones are easy to treat using the standard Sobolev calculus.
Specifically, the top-order terms are
\[
\int_0^{\tau}\int_0^{\tau'} \mathring{F}\AngMD^4\g_\tau \mathscr{A} \MD v\,d\tau'' d\tau' \ \ \text{ and } \ 
\int_0^{\tau}\int_0^{\tau'} \mathring{F}\g_\tau \mathscr{A} \AngMD^4D \g_\tau \upeta \,d\tau'' d\tau'.
\]
In the first case, we integrate the $\partial_{\tau}$ derivative away from 
$\AngMD^4 \partial_{\tau} \mathscr{A}$ and in the second case 
we integrate it away from $\AngMD^4 \MD \partial_{\tau} \upeta$. 
We then use the standard Sobolev calculus
to obtain that
$
\left\|
	\int_0^{\tau}\int_0^{\tau'} \mathring{F}\AngMD^4\g_\tau \mathscr{A} \MD v\,d\tau'' d\tau'
\right\|_{\mathfrak{M}_{\tau}}
\leq \tau P(\norm(\tau))
$,
and similarly for the other double time integral,
from which the desired bounds for the 
integral
$
\left|
\int_{\mathfrak{M}_{\tau}} \langle \cdots \rangle_g
\right|
$
of interest easily follow. 
We emphasize that the presence of the double integration in $\tau$ is crucial 
for obtaining the factor of $\tau$ in the previous inequality.
The integral
$
\int_{\mathfrak{M}_{\tau}} \langle \cdots \rangle_g
$
generated by the term $\int_0^{\tau}\int_0^{\tau'} \mathring{F} \AngMD^4 (\Lagvort \partial_{\tau} v)_{\mu\nu} \,d\tau'' d\tau'$ 
(which corresponds to the last term on the right-hand side of \eqref{E:ETAVORTICITY})
is a bit more difficult to handle. We first split it into two pieces:
\begin{align}
\int_0^{\tau}\int_0^{\tau'} \mathring{F} \AngMD^4 (\Lagvort \partial_{\tau} v)_{\mu\nu} \,d\tau'' d\tau' =& 
\int_0^{\tau}\int_0^{\tau'} \mathring{F} \AngMD^4 \left( (\Lagvort \partial_{\tau} v)_{\mu\nu} 
- (\g_{\tau}^2v_{\nu}v_{\mu}-\g_{\tau}^2v_{\mu}v_{\nu})\right)\,d\tau'' d\tau'  \notag \\
 &+ \int_0^{\tau}\int_0^{\tau'} \mathring{F} \AngMD^4\left( \g_{\tau}^2v_{\nu}v_{\mu}-\g_{\tau}^2v_{\mu}v_{\nu}\right)\,d\tau'' d\tau' 
 \notag \\
 & =: I_1+I_2.
\end{align}
To handle the integrals corresponding to $I_1$, we first rewrite the integrand 
by using equation \eqref{E:VORTICITYBASIC} for substitution.
We can then bound
$
\left|
\int_{\mathfrak{M}_{\tau}} 
	\langle I_1,I_1 \rangle_g
\, dy
\right|
$
by using the standard Sobolev calculus and,
to handle some of the integrals,
by integrating by parts in time as before.
To bound 
$
\left|
\int_{\mathfrak{M}_{\tau}} 
	\langle I_2,I_2 \rangle_g
\, dy
\right|
$,
we again note that it suffices to 
bound the integral 
$
\left|
\int_{\mathfrak{M}_{\tau}} \langle \cdots \rangle_g
\right|
$
generated by the highest-order part
\begin{align} \label{E:TRICKY0}
	\int_0^{\tau}\int_0^{\tau'}
	\mathring{F} \left(\AngMD^4\g_\tau^2 v^\nu v_\mu - \AngMD^4\g_\tau^2 v_\mu v^\nu\right)  d\tau'' d\tau'.
\end{align}
We now integrate by parts in time twice in \eqref{E:TRICKY0} to obtain the identity
\begin{align}
&  \int_0^{\tau}\int_0^{\tau'} \mathring{F}  
\left(\AngMD^4\g_\tau^2 v_\nu v_\mu - \AngMD^4\g_\tau^2 v_\mu v_\nu\right)  d\tau'' d\tau' \notag \\
&  = \int_0^\tau \mathring{F}  \left(\AngMD^4\g_\tau v_\nu v_\mu - \AngMD^4\g_\tau v_\mu v_\nu\right) \,d\tau'
 - \int_0^{\tau}\int_0^{\tau'} \mathring{F} 
 \left(\AngMD^4\g_\tau v_\nu \g_\tau v_\mu + \AngMD^4\g_\tau v_\mu \g_\tau v_\nu\right)  \,d\tau'' \,d\tau' \notag  \\
 & = \mathring{F} \AngMD^4v_\nu v_\mu - \mathring{F} \AngMD^4 v_\mu v_\nu  
 	\notag \\
 & \ \
	 	-  \int_0^\tau \mathring{F} \left(\AngMD^4 v_\nu \g_\tau v_\mu - \AngMD^4v_\mu \g_\tau v_\nu\right) \,d\tau'
 	- \int_0^{\tau}\int_0^{\tau'} \mathring{F} \left(\AngMD^4\g_\tau v_\nu \g_\tau v_\mu + \AngMD^4\g_\tau v_\mu \g_\tau v_\nu\right)  \,d\tau'' \,d\tau'. \label{E:TRICKY}
\end{align}
We can bound the integrals 
$
\left|
\int_{\mathfrak{M}_{\tau}} \langle \cdots \rangle_g
\right|
$
corresponding to the two integrals on the last line of right-hand side of~\eqref{E:TRICKY} 
by using the standard Sobolev calculus.
To bound the integral 
$
\left|
\int_{\mathfrak{M}_{\tau}} \langle \cdots \rangle_g
\right|
$
generated by the remaining 
(non-time integrated)
terms on right-hand side of~\eqref{E:TRICKY},
we first use \eqref{E:NORMALIZATIONLAGRANGIAN} to derive the identity
\begin{align*}
\left(\AngMD^4v_\nu v_\mu - \AngMD^4 v_\mu v_\nu\right)\left(\AngMD^4v^\nu v^\mu - \AngMD^4 v^\mu v^\nu \right)
 = - 2 \langle \AngMD^4 v,\AngMD^4 v\rangle_g - 2 (\AngMD^4 v_\mu v^\mu)^2.
\end{align*}
Thus, using \eqref{E:WEIGHTEDCONTRACTIONAGAINSTUFIRSTSOBOLEVESTIMATE}, we deduce
\be\label{E:AUXBOUND2}
\int_{\mathfrak{M}_{\tau}}\mathring{F}^2 (\AngMD^4 v_\mu v^\mu)^2\,dy \le \mathring{M} + \tau P(\norm(\tau)).
\ee
Finally, using the smallness assumption~\eqref{E:SMALLDENSITY},  
we deduce the desired bound as follows:
\begin{align} \label{E:AUXBOUND2.1}
\left|
\int_{\mathfrak{M}_{\tau}}
	\mathring{F}^2 
	\langle\AngMD^4 v,\AngMD^4 v\rangle_g
\,dy
\right|
& \lesssim \|\mathring{F} \|_{L^{\infty}(\mathfrak{M}_{\tau})} 
\left\|
	\sqrt{\mathring{F}}\AngMD^4v
\right\|_{L^2(\mathfrak{M}_{\tau})}^2 
\lesssim \varepsilon \norm(\tau.)
\end{align}
We have thus shown that
$
\left|
\int_{\mathfrak{M}_{\tau}} 
	\langle I_2,I_2 \rangle_g
\, dy
\right|
\leq \mathring{M} + C\left(\delta+\varepsilon\right)\norm(\tau)+ \tau P(\norm(\tau))
$,
which completes the proof of~\eqref{E:MAINVORTICITYESTIMATES} in the case $p=0$.
\ \\
\noindent
{\em Cases $p=1,2,3,4$.}
To prove \eqref{E:MAINVORTICITYESTIMATES} in these cases,
we note that it suffices to instead bound the related term
\begin{align*} 
\left|
	\int_{\mathfrak{M}_\tau}
		\mathring{F}^2 
		\langle 
			\partial_{\tau}^{2(p-1)} \Lagvort \AngMD^{4-p} \partial_{\tau} v,
			\partial_{\tau}^{2(p-1)}\Lagvort \AngMD^{4-p} \partial_{\tau} v
		\rangle_g 
	\, dy
\right|
 \end{align*} 
by the $ \mbox{right-hand side of~\eqref{E:MAINVORTICITYESTIMATES}}$;
the difference between the two terms 
is lower-order and is therefore easy to bound via the fundamental theorem of calculus,
as we described in Sect.~\ref{SS:FTCESTIMATES}.
The desired bounds are a simple consequence of the already proven estimate \eqref{hss1}.

\medskip

\noindent
{\bf Proof of \eqref{E:FLATMAINVORTICITYESTIMATES}.}
The starting point of the proof is again the identity~\eqref{E:ETAVORTICITY}.
We first address the case $p=0$, and we aim to prove the following preliminary
estimate for $\mu,\nu = 0,1,2,3$:
\begin{align} \label{E:FULLVORTH3BOUND}
\|(\Lagvort \upeta)_{\mu\nu}\|_{H^3(\mathfrak{M}_{\tau})}^2 \le \mathring{M} + \tau P(\norm(\tau)).
\end{align}
To prove \eqref{E:FULLVORTH3BOUND}, 
we start by noting that
it suffices to bound the square norm
$\| \cdot \|_{H^2(\mathfrak{M}_{\tau})}^2$
of $(\Lagvort \partial_A \upeta)_{\mu\nu}$
for $A = 1,2,3$; the remaining terms
that must be bounded to control
$\|(\Lagvort \upeta)_{\mu\nu} \|_{H^3(\mathfrak{M}_{\tau})}^2$
are lower-order and thus are
easy to bound via the fundamental theorem of calculus,
as we described in Sect.~\ref{SS:FTCESTIMATES}.
We now note that the square norm $\| \cdot \|_{H^2(\mathfrak{M}_{\tau})}^2$
of the terms on the first two lines of right-hand side of~\eqref{E:ETAVORTICITY} are easy to bound 
in the square norm $\| \cdot \|_{H^2(\mathfrak{M}_{\tau})}^2$
by $\leq \mathring{M} + \tau P(\norm(\tau))$ via the standard Sobolev calculus.
To treat the terms generated by the double time integral of
the term $\partial_A \left(\mathcal{B}_{\mu\nu}(\mathscr{A},\ITIMESMD v) \right)$ on right-hand side of~\eqref{E:ETAVORTICITY}, 
we again note that it suffices to bound the 
square norm $\| \cdot \|_{L^2(\mathfrak{M}_{\tau})}^2$
of the top-order part, which is of the form
\begin{align*}
\int_0^{\tau}\int_0^{\tau'} \MD^4 \partial_{\tau} \upeta \MD \partial_{\tau} \upeta \mathscr{A} \mathscr{A} d\tau'' d\tau'.
\end{align*}
We now integrate by parts in time to remove the 
time derivative off of the factor 
$\MD^4 \partial_{\tau} \upeta$.
Then the same arguments we used in treating the double time
integrals that we encountered in the proof of \eqref{E:MAINVORTICITYESTIMATES}
allow us to deduce that
\[
\left\|
	\int_0^{\tau}\int_0^{\tau'} \MD^4 \partial_{\tau} \upeta \MD \partial_{\tau} \upeta \mathscr{A} \mathscr{A} d\tau'' d\tau'
\right\|_{L^2(\mathfrak{M}_{\tau})}^2
\leq \tau P(\norm(\tau))
\]
as desired, where the smallness factor $\tau$ comes from the second time integral.
Using a similar argument,
we bound the terms generated by the double time integral of
the term $\MD^2 \partial_A (\Lagvort \g_\tau v)_{\mu\nu}$ 
generated by the last term on right-hand side of~\eqref{E:ETAVORTICITY}.
For this estimate, we use~\eqref{E:VORTICITYBASIC} to substitute for
$(\Lagvort \g_\tau v)_{\mu\nu}$.
Combining the above bounds, 
we conclude the desired estimate \eqref{E:FULLVORTH3BOUND}.


We will now use \eqref{E:FULLVORTH3BOUND} to obtain the desired bound \eqref{E:FLATMAINVORTICITYESTIMATES}
in the case $p=0$. We start with the following algebraic decomposition for $i,j\in\{1,2,3\}$, which 
follows easily from Def.~\ref{D:DIFFERENTIALOPERATORS}:
\begin{align} \label{E:VORTICITYRELATIONSHIPS}
(\Flatvort \upeta)_{ij} = (\Lagvort \upeta)_{ij} -(\mathscr{A}^K_i-\delta^K_i)\g_K\upeta_j 
+ \left(\mathscr{A}^K_j - \delta^K_j\right)\g_K\upeta_i. 
\end{align}
We bound the square norm $\| \cdot \|_{H^3(\mathfrak{M}_{\tau})}^2$
of the first term on right-hand side of~\eqref{E:VORTICITYRELATIONSHIPS} via \eqref{E:FULLVORTH3BOUND}.
To bound the square norm $\| \cdot \|_{H^3(\mathfrak{M}_{\tau})}^2$
of the last two terms on right-hand side of~\eqref{E:VORTICITYRELATIONSHIPS}
by $\leq \tau P(\norm(\tau))$,
we use the simple estimate
$
\left\| 
	\ITIMESMD \upeta
\right\|_{H^3(\mathfrak{M}_{\tau})}^2
\leq \norm(\tau)
$,
and, to gain the smallness factor $\tau$, 
the estimates
$
\left\| 
	\mathscr{A}^K_i-\delta^K_i 
\right\|_{L^{\infty}(\mathfrak{M}_{\tau})}
\leq \tau P(\norm(\tau))
$
(for $i= 1,2,3$ and $K=0,1,2,3$); 
these $L^{\infty}(\mathfrak{M}_{\tau})$ estimates
are a simple consequence of the fundamental theorem of calculus,
the standard Sobolev calculus, 
and the initial conditions 
\eqref{E:INITIALINVERSECHOVMATRIX}.
We have thus proved \eqref{E:FLATMAINVORTICITYESTIMATES}
in the case $p=0$.

{\em Cases $p=1,2,3,4$.}
To handle these cases,
we aim to first prove the following preliminary estimate
for $\mu,\nu = 0,1,2,3$,
in analogy with \eqref{E:FULLVORTH3BOUND}:
\begin{align} \label{E:FULLVORTBOUND}
\|\Lagvort (\partial_{\tau}^{2p} \upeta)_{\mu\nu} \|_{H^{3-p} (\mathfrak{M}_{\tau})}^2 
\leq \mathring{M} + \tau P(\norm(\tau)).
\end{align}
Once we prove \eqref{E:FULLVORTBOUND},
the desired bound \eqref{E:FLATMAINVORTICITYESTIMATES}
follows easily by using the same argument
given just above. To prove \eqref{E:FULLVORTBOUND},
we start by noting that it would follow from
\begin{align} \label{E:ALTERNATEFULLVORTBOUND}
\| \partial_{\tau}^{2(p-1)} (\Lagvort \partial_{\tau} v)_{\mu\nu} \|_{H^{3-p}(\mathfrak{M}_{\tau})}^2 
\leq \mathring{M} + \tau P(\norm(\tau));
\end{align}
the remaining terms that must be bounded to control
$
\|(\Lagvort \partial_{\tau}^{2p} \upeta)_{\mu\nu} \|_{H^{3-p} (\mathfrak{M}_{\tau})}^2 
$
are lower-order and thus are
easy to bound via the fundamental theorem of calculus,
as we described in Sect.~\ref{SS:FTCESTIMATES}.
We have already proved the desired bound \eqref{E:ALTERNATEFULLVORTBOUND}
as inequality \eqref{hss2}.
This completes our proof of \eqref{E:FLATMAINVORTICITYESTIMATES}
and therefore that of Prop.~\ref{P:VORTICITYESTIMATES}.
\end{proof}

\section{Energy estimates for horizontal derivatives and time derivatives}\label{S:ENERGY}
In this section, we derive energy estimates for the horizontal and time derivatives of the solution
up to top order. In particular, we do \emph{not} commute the evolution equations with the $\partial_3$
operator in this section. We provide the main energy estimates in Prop.~\ref{P:ENERGYLEMMA}.
The estimates rely on the bounds for the vorticity that we have independently obtained
in Prop.~\ref{P:VORTICITYESTIMATES}.

We start with the following preliminary lemma,
which shows that the $g$-norm and $h$-norm of the relevant energy integrands
are equal up to controlled error terms.

\begin{lemma}\label{L:EQUIVALENCE}
Let $\varepsilon$ be as in~\eqref{E:SMALLDENSITY}.
Under the bootstrap assumptions of Sect.~\ref{SS:BOOTSTRAPASSUMPTIONS},
we have the following estimates 
for $\tau \in [0,T]$.
\begin{enumerate}
\item For $p\in\{0,1,2,3,4\}$, we have
\begin{align}
\mathring{F} \langle \partial_{\tau}^{2p}\AngMD^{4-p} v,\partial_{\tau}^{2p}\AngMD^{4-p} v\rangle_g 
& = \mathring{F} \langle \partial_{\tau}^{2p}\AngMD^{4-p} v,\partial_{\tau}^{2p}\AngMD^{4-p} v\rangle_h 
+ O_{L^1(\mathfrak{M}_{\tau})}(\mathring{M} + \tau P(\norm(\tau))).
\label{E:MA1}
\end{align}
\item 
For $p\in\{0,1,2,3,4\}$, we have
\begin{align}
\mathring{F}^2 \langle\Lagdiff \partial_{\tau}^{2p}\AngMD^{4-p} \upeta,\Lagdiff\partial_{\tau}^{2p}\AngMD^{4-p} \upeta\rangle_g 
& = \mathring{F}^2 \langle\Lagdiff \partial_{\tau}^{2p}\AngMD^{4-p} \upeta,\Lagdiff\partial_{\tau}^{2p}\AngMD^{4-p} \upeta\rangle_h 
	\notag \\
& \ \
	+ O_{L^1(\mathfrak{M}_{\tau})}(\mathring{M} + \varepsilon \norm(\tau) + \tau P(\norm(\tau))). \label{E:NA1}
\end{align}
\end{enumerate}
\end{lemma}

\begin{proof}
We first prove \eqref{E:MA1}.
Recalling the definition~\eqref{E:HMETRIC} of the metric $h$,
we see that
\begin{align}\label{E:part1}
\langle \partial_\tau^{2p}\AngMD^{4-p} v, \, \partial_{\tau}^{2p}\AngMD^{4-p} v \rangle_h  
& = \langle \partial_\tau^{2p}\AngMD^{4-p} v, \, \partial_{\tau}^{2p}\AngMD^{4-p} v \rangle_g  + 2 (v_{\alpha} \partial_\tau^{2p}\AngMD^{4-p}v^{\alpha})^2. 
\end{align}
The desired bound \eqref{E:MA1} now follows as a simple consequence of
\eqref{E:part1} and \eqref{E:WEIGHTEDCONTRACTIONAGAINSTUFIRSTSOBOLEVESTIMATE}.

We now prove \eqref{E:NA1}. We first use \eqref{E:HMETRIC}-\eqref{E:HINVERSEMETRIC}, 
the identity \eqref{E:FOURVELOCITYAPPEARSINCHOVMATRIX},
and the second identity in \eqref{E:CHAINRULES}
to compute the following
identity, valid for any vectorfield $X$:
\begin{align}
\langle \Lagdiff X , \Lagdiff X \rangle_h
& := h_{\alpha \widetilde{\alpha}} 
				(h^{-1})^{\beta \widetilde{\beta}}  
				\Lagdiff X_{\beta}^{\alpha} 
				\Lagdiff X_{\widetilde{\beta}}^{\widetilde{\alpha}} 
		= \langle \Lagdiff X , \Lagdiff X \rangle_g + 2 \langle \g_\tau X, \g_\tau X\rangle_g \notag
			\\
	&				
	+ 		  2v_{\alpha}v_{\widetilde{\alpha}}g^{\beta\widetilde{\beta}}\mathscr{A}^K_{\beta}\mathscr{A}^L_{\widetilde{\beta}}\g_KX^{\alpha}\g_LX^{\widetilde{\alpha}} 
				+2\left(v_{\alpha}\g_{\tau}X^{\alpha}\right)^2. \label{E:expandingh2}
\end{align}
We now set $X := \partial_{\tau}^{2p}\AngMD^{4-p} \upeta$.
From the smallness assumption~\eqref{E:SMALLDENSITY},
we deduce the bound
$
2 \left|
	\mathring{F}^2
	\int_{\mathfrak{M}_{\tau}} 
		\langle \g_\tau X, \g_\tau X\rangle_g
	\, dy
	\right|
	\lesssim \varepsilon \norm(\tau)
$.
The desired bound \eqref{E:NA1} now follows as a simple consequence of the previous bound,
\eqref{E:expandingh2},
\eqref{E:TOPORDERWEIGHTEDCONTRACTIONAGAINSTUFIRSTSOBOLEVESTIMATE},
and the estimate 
$
\|
	\mathscr{A}^K_{\beta}
\|_{L^{\infty}(\mathfrak{M}_{\tau})}
\leq \mathring{M} + \tau P(\norm(\tau))
$,
which is an easy consequence of the fundamental theorem of calculus and Sobolev embedding.
\end{proof}


In the next proposition, we derive our main energy estimates.

\begin{proposition}\label{P:ENERGYLEMMA}
Let $\delta > 0$ be a constant and let $\varepsilon$ be as in \eqref{E:SMALLDENSITY}.
Under the bootstrap assumptions of Sect.~\ref{SS:BOOTSTRAPASSUMPTIONS},
we have the following estimates for $\tau \in [0,T]$
(where $\mathring{M}$ is allowed to depend on $\delta^{-1}$):
\begin{align} \label{E:ENERGYBOUND}
\mathcal{E}(\tau) 
& \leq C \norm(0)
+ (\delta + C \sqrt{\varepsilon}) \norm(\tau)
+ \tau P(\norm(\tau)).
\end{align}
\end{proposition}

\begin{proof}
For $p\in\{0,1,2,3,4\}$, we apply $\partial_{\tau}^{2p}\AngMD^{4-p}$ to the evolution equation~\eqref{E:LAGRNAGIANMAINEVOLUTIONEQUATION}. 
A straightforward application of the product rule leads to the following equations for $\mu=0,1,2,3$:
\begin{align}\label{E:DIFFERENTIATING}
\mathring{F} \Enth \partial_{\tau}^{2p+1}\AngMD^{4-p} v^{\mu} 
+ g^{\mu\alpha}\g_A\left(\mathring{F}^2\partial_{\tau}^{2p}\AngMD^{4-p} \mathscr{A}_{\alpha}^A \Enth \mathscr{J}^{-1}+\mathring{F}^2\mathscr{A}^A_{\alpha} \partial_{\tau}^{2p}\AngMD^{4-p}(\Enth \mathscr{J}^{-1})\right)
= \mathcal{R}_{p1}^{\mu},
\end{align}
where the error term $\mathcal{R}_{p1}^{\mu}$ is given explicitly by the following formula:
\begin{align}
\mathcal R^\mu_{p1} 
&= \sum_{p_1+p_2=2p \atop q_1+q_2 = 4-p, p_1+q_1>0} C_{p_1,q_1}\g_\tau^{p_1}\AngMD^{q_1}(\mathring{F} S) \g_\tau^{p_2+1}\AngMD^{q_2} v^\mu \notag \\
& + \sum_{p_1+p_2+p_3 =2p, q_1+q_2+q_3 = 4-p, \atop p_1+2q_1<8, p_3+2q_3<8} g^{\mu\nu}\g_K\left(\g_\tau^{p_1}\AngMD^{q_1}\mathscr{A}^K_\nu \g_\tau^{p_2}\AngMD^{q_2}(\mathring{F}^2) \g_\tau^{p_3}\AngMD^{q_3} (\mathscr{J}^{-1} \Enth)\right) \notag \\
& + 2\partial_{\tau}^{2p}\AngMD^{4-p}\left( v^{\mu}\mathring{F}\mathscr{J}^{-2}\Enth^{3/2}\partial_{\tau}\mathscr{J}\right). \label{E:RA1MU}
\end{align}
To simplify the notion, we set
$\dot{\upeta} := \partial_{\tau}^{2p}\AngMD^{4-p} \upeta$,
$\dot{v} := \partial_{\tau}^{2p}\AngMD^{4-p} v$,
and similarly for other quantities differentiated by
$\partial_{\tau}^{2p}\AngMD^{4-p}$.
To prove \eqref{E:ENERGYBOUND}, 
we will first use an integration by parts argument based on equation \eqref{E:DIFFERENTIATING} to show that
\begin{align}
& \frac{1}{2} \int_{\mathfrak{M}_\tau}  \left\{ \mathring{F} \Enth \langle \dot{v},\dot{v}\rangle_h
+\mathring{F}^2 \Enth \mathscr{J}^{-1}  \left(\langle\Lagdiff \dot{\upeta},\Lagdiff\dot{\upeta}\rangle_h 
+  \mathscr{J}^{-2}\frac{1+f}{1-f} \dot{\mathscr{J}}^2\right)
\right\}\, dy\notag  \\
& =
\frac{1}{2} \int_{\mathfrak{M}_0}  \left\{ \mathring{F} \Enth \langle \dot v,\dot v \rangle_h
+
\mathring{F}^2 \Enth \mathscr{J}^{-1}  \left(\langle\Lagdiff \dot{\upeta},\Lagdiff\dot{\upeta}\rangle_h
+  
\mathscr{J}^{-2}\frac{1+f}{1-f} \dot{\mathscr{J}}^2\right)
	\right\}
	\notag
	\\
& \ \ \ 
+ \frac{1}{2}\int_{\mathfrak{M}_{\tau'}}\mathring{F}^2 \Enth \mathscr{J}^{-1} \langle\Lagvort \,\dot\upeta,\Lagvort\,\dot\upeta\rangle_g 
\,dy \Big|^{\tau'=\tau}_{\tau'=0} \notag \\
& \ \ \ 
+  
\int_{\mathfrak{M}_{\tau'}}\mathring{F}^2  \Enth \mathscr{J}^{-1} \left( \mathscr{A}^0_{\beta}\mathscr{A}^L_{\alpha}  \g_L \dot{\upeta}^{\beta}\dot{v}^{\alpha}
+\mathscr{A}^0_{\alpha}  \mathscr{J}^{-1} \frac{1+f}{1-f}  \dot{\mathscr{J}}\dot{v}^{\alpha}\right) \,dy \Big|^{\tau'=\tau}_{\tau'=0} 
\notag \\
& \ \ \   + \int_{[0, \tau] \times \mathfrak{M}} \mathcal{R}_p\,dy d\tau'
+ O(\tau P(\norm(\tau))), \label{E:ENERGYIDENTITY1} 
\end{align}			
where 
\begin{align}
\mathcal{R}_p 
& := \mathcal R^\mu_{p3} \dot v_\mu + \frac{1}{2}\g_\tau(\mathring{F} \Enth)\langle \dot v,\dot v\rangle_g
+\frac{1}{2} \g_\tau\left(g_{\alpha \widetilde{\alpha}} g^{\beta \widetilde{\beta}}
				\mathscr{A}_{\beta}^K 
				\mathscr{A}_{\widetilde{\beta}}^L\right)\partial_K
				\dot{\upeta}^{\alpha} 
				\partial_L \dot{\upeta}^{\widetilde{\alpha}} 
				\notag \\
				& -\mathscr{A}_{\beta}^L
				\partial_L\dot{\upeta}^{\alpha}\left(-\g_{\tau}\mathscr{A}^K_{\alpha} \partial_K \dot{\upeta}^{\beta}
				+ g^{\beta \widetilde{\beta}}\g_{\tau}\mathscr{A}_{\widetilde{\beta}}^K\partial_K \dot{\upeta}_{\alpha}\right)
				+\frac{1}{2}\g_{\tau}\left(\mathring{F}^2    \Enth \mathscr{J}^{-3} \frac{1+f}{1-f}\right) \dot{\mathscr{J}}^2
				\notag \\
				&				-\mathring{F}^2 \Enth \mathscr{J}^{-3} \frac{1+f}{1-f}  \dot{\mathscr{J}} \mathcal{R}_{p4},
				\label{E:ENERGYESTERR}
\end{align}
with
\begin{align} 
\mathcal R_{p2}
&:= 2\mathscr{J}^{-1} \Enth^{3/2} \sum_{p_1+p_2=2p \atop q_1+q_2 = 4-p, p_1+q_1>0}C_{p_1,q_1}\g_\tau^{p_1}\AngMD^{ q_1}\mathring{F} \g_\tau^{p_2}\AngMD^{ q_2}\left(\mathscr{J}^{-1}\right) 
\notag \\
& + \left(2 \mathscr{J}^{-1}\Enth^{3/2}\mathring{F} + \Enth \right)\left(\partial_{\tau}^{2p}\AngMD^{4-p} \left(\mathscr{J}^{-1}\right) + \mathscr{J}^{-2}\partial_{\tau}^{2p}\AngMD^{4-p} \mathscr{J}\right)  
\notag \\
&+ \mathscr{J}^{-1}\partial_{\tau}^{2p}\AngMD^{4-p} \Enth - 2\mathscr{J}^{-1}\Enth^{3/2} \partial_{\tau}^{2p}\AngMD^{4-p}(\mathring{F} \mathscr{J}^{-1}) 
\notag \\
& +  \sum_{p_1+p_2=2p, \, q_1+q_2 = 4-p \atop 0<p_1+q_1<4-p}C_{p_1,q_1}\g_\tau^{p_1}\AngMD^{ q_1}\left(\mathscr{J}^{-1}\right)\g_\tau^{p_2}\AngMD^{ q_2} \Enth,
\label{E:ENERGYERRORP2}
	\\
\mathcal{R}_{p3}^{\mu} 
& := \mathcal{R}_{p1}^{\mu} - g^{\mu\alpha}\g_K\left(\mathring{F}^2 [\partial_{\tau}^{2p}\AngMD^{4-p},\mathscr{A}^K_{\alpha}]\right) \Enth \mathscr{J}^{-1} - g^{\mu\alpha}\g_K\left(\mathring{F}^2\A^K_\alpha \mathcal R_{p2}\right),
	\label{E:ENERGYESTERR3} \\
\label{E:ENERGYESTERR4}
\mathcal{R}_{p4} 
& := \sum_{p_1+p_2=2p, q_1+q_2 = 4-p\atop p_1+q_1>0} C_{p_1,q_1} \g_\tau^{p_1}\AngMD^{q_1} \Cof^K_\nu \g_\tau^{p_2}\AngMD^{q_2}\g_Kv^\nu.
\end{align}
We will then show that for $p = 0,1,2,3,4$, we have
\begin{align}\label{E:MAINERROR}
\left|
	\int_{[0,\tau] \times\mathfrak{M}} \mathcal{R}_p \,dy d\tau'
\right| 
& \leq \mathring{M} + \delta \norm(\tau) +  \tau P(\norm(\tau))
\end{align}
and also that
\begin{align} \label{E:ENERGYESTIMATEREMAININGERRORTERMS}
	\mbox{all remaining terms on the right-hand side of~\eqref{E:ENERGYIDENTITY1} 
	are $\leq$ \mbox{right-hand side of}~\eqref{E:ENERGYBOUND}},
\end{align}
from which \eqref{E:ENERGYBOUND} easily follows.

Note, that
the case  $p=0$ corresponds to $\AngMD^4$-differentiated problem, while the case $p=4$ corresponds to $\partial_\tau^8$-differentiated problem.  
The analysis for these end-point cases are very similar, but require some minor modifications.

\vspace{.1 in}
\noindent
{\bf Step 1. Proof of \eqref{E:ENERGYIDENTITY1}.} 
We start by decomposing the top-order terms inside the parentheses
on LHS~\eqref{E:DIFFERENTIATING}.
Using~\eqref{E:DETAINVERSEDIFFERENTIATED}, we deduce
\be\label{E:ONE}
\partial_{\tau}^{2p}\AngMD^{4-p} \mathscr{A}_{\alpha}^K 
= - \mathscr{A}^K_{\beta} \mathscr{A}^L_{\alpha} \g_L\partial_{\tau}^{2p}\AngMD^{4-p}\upeta^{\beta} 
+ [\partial_{\tau}^{2p}\AngMD^{4-p},\mathscr{A}^K_{\alpha}],
\ee
where $[\partial_{\tau}^{2p}\AngMD^{4-p},\mathscr{A}^K_{\alpha}]$ denotes the commutator, 
which is given explicitly by the formula
\begin{align*}
 [\partial_{\tau}^{2p}\AngMD^{4-p},\mathscr{A}^K_{\alpha}] =  \sum_{p_1+p_2=2p \atop q_1+q_2=4-p, p_1+q_1>0} C_{p_1,q_1}\g_\tau^{p_1}\AngMD^{q_1}(\mathscr{A}^K_{\alpha} \mathscr{A}^L_{\beta}) \g_\tau^{p_2}\AngMD^{q_2} \g_L\upeta^{\beta}.
\end{align*}
Similarly, we have 
\be\label{E:TWO}
\partial_{\tau}^{2p}\AngMD^{4-p} (\Enth \mathscr{J}^{-1}) = - \Enth \mathscr{J}^{-2}\frac{1+f}{1-f} \partial_{\tau}^{2p}\AngMD^{4-p} \mathscr{J} + \mathcal{R}_{p2},
\ee
where the error term $\mathcal{R}_{p2}$ is explicitly given by \eqref{E:ENERGYERRORP2}.
Substituting \eqref{E:ONE} and~\eqref{E:TWO} into~\eqref{E:DIFFERENTIATING}, we obtain
\begin{align}\label{E:HIGHORDEREQUATION}
& \mathring{F} S\g_{\tau}^{2p+1}\AngMD^{4-p} v^{\mu} 
- g^{\mu\alpha}\g_K\Big(\mathring{F}^2 \Enth \mathscr{J}^{-1} \mathscr{A}^K_{\beta} \mathscr{A}^L_{\alpha} 
\g_L \partial_{\tau}^{2p} \AngMD^{4-p}\upeta^{\beta} \\
& \quad  + \mathring{F}^2 \mathscr{A}^K_{\alpha}   \Enth \mathscr{J}^{-2} \frac{1+f}{1-f}  \partial_{\tau}^{2p}\AngMD^{4-p} \mathscr{J} \Big)  \notag
= \mathcal{R}_{p3}^{\mu}, 
\end{align}
where $\mathcal{R}_{p3}^{\mu}$ is given by \eqref{E:ENERGYESTERR3}.

We now use equation~\eqref{E:HIGHORDEREQUATION} to derive a divergence 
identity that, when integrated, will yield our desired energy identity.
We will use our previously mentioned shorthand notation
$\dot{v} = \partial_{\tau}^{2p}\AngMD^{4-p} v$, etc.
We start by contracting~\eqref{E:HIGHORDEREQUATION} with $g_{\mu\nu} \dot{v}^{\nu}$:
\begin{align}\label{E:CONTRACTED}
& \mathring{F} \Enth g_{\mu\nu}\g_{\tau}\dot{v}^{\mu}\dot{v}^{\nu}
-\g_K\Big(\mathring{F}^2 \Enth \mathscr{J}^{-1} \mathscr{A}^K_{\beta} \g_L \dot{\upeta}^{\beta}\mathscr{A}^L_{\nu} 
 + \mathring{F}^2 \mathscr{A}^K_{\nu}    \Enth \mathscr{J}^{-2} \frac{1+f}{1-f}  \dot{\mathscr{J}} \Big) \dot{v}^{\nu}  
 = g_{\mu\nu}\mathcal{R}^\mu_{p3}\dot{v}^{\nu}. 
\end{align}
The first product on LHS~\eqref{E:CONTRACTED} is a pure time derivative plus a harmless error term:
\begin{align}\label{E:VELOCITYENERGY}
\mathring{F} \Enth g_{\mu\nu}\g_{\tau}\dot{v}^{\mu}\dot{v}^{\nu}
=\frac{1}{2}\g_{\tau}\left(\mathring{F} \Enth g_{\mu\nu}\dot{v}^{\mu}\dot{v}^{\nu}\right) 
- \frac{1}{2}\g_{\tau}\left(\mathring{F} \Enth \right)g_{\mu\nu}\dot{v}^{\mu}\dot{v}^{\nu}.
\end{align}
Next, we rewrite the remaining expression on LHS~\eqref{E:CONTRACTED} as follows:
\begin{align}
& -\g_K\Big(\mathring{F}^2  \Enth \mathscr{J}^{-1} \mathscr{A}^K_{\beta} \g_L \dot{\upeta}^{\beta}\mathscr{A}^L_{\nu} 
+ \mathring{F}^2 \mathscr{A}^K_{\nu}    \Enth \mathscr{J}^{-2} \frac{1+f}{1-f}  \dot{\mathscr{J}} \Big) \dot{v}^{\nu} \nonumber \\
& = \mathring{F}^2  \Enth \mathscr{J}^{-1} \mathscr{A}^K_{\beta} \g_L \dot{\upeta}^{\beta}\mathscr{A}^L_{\nu} \g_K \dot{v}^{\nu}
+ \mathring{F}^2 a^K_{\nu}    \Enth \mathscr{J}^{-3} \frac{1+f}{1-f}  \dot{\mathscr{J}} \g_K \dot{v}^{\nu} \notag\\
& \ \ \ - \g_K\left(\mathring{F}^2  \Enth \mathscr{J}^{-1} \mathscr{A}^K_{\beta}\mathscr{A}^L_{\nu}  \g_L \dot{\upeta}^{\beta}\dot{v}^{\nu}
+ \mathring{F}^2 \mathscr{A}^K_{\nu}    \Enth \mathscr{J}^{-2} \frac{1+f}{1-f}  \dot{\mathscr{J}}\dot{v}^{\nu}\right).  \label{E:HIGHORDER2}
\end{align}
We now decompose the first term on the right-hand side of~\eqref{E:HIGHORDER2}.
We will use the following identity, 
valid for any vectorfield $X$ and obtained by
decomposing $\Lagdiff X$ into its symmetric and antisymmetric parts: 
\begin{align} \label{E:SPLITGRADIENTINTOSYMMETRICANDANTISYMMETRIC}
\Lagdiff_\beta X^\alpha \Lagdiff_\alpha X^\beta = - \langle\Lagvort X,\Lagvort X\rangle_g + \langle\Lagdiff X,\Lagdiff X\rangle_g.
\end{align}
Hence, from the second identity in \eqref{E:CHAINRULES},
\eqref{E:SPLITGRADIENTINTOSYMMETRICANDANTISYMMETRIC},
and the identity $\dot{v}^{\alpha} = \partial_{\tau} \dot{\upeta}^{\alpha}$,
we deduce
\begin{align}
& \mathscr{A}^K_{\beta} \mathscr{A}^L_{\alpha} \g_K \dot{v}^{\alpha} \g_L \dot{\upeta}^{\beta}  
= \frac{1}{2} \partial_{\tau} 
	(\Lagdiff_\beta \dot{\upeta}^\alpha \ \Lagdiff_\alpha \dot{\upeta}^\beta) 
	-  \frac{1}{2}
			\g_{\tau}\left(\mathscr{A}^K_{\beta} \mathscr{A}^L_{\alpha} \right)
			\g_K \dot{\upeta}^{\alpha} \g_L \dot{\upeta}^{\beta}
	\notag \\
				& = \frac{1}{2}\g_{\tau}\left(\langle\Lagdiff \dot{\upeta},\Lagdiff\dot{\upeta}\rangle_g 
				- \langle\Lagvort\dot\upeta,\Lagvort\dot\upeta\rangle_g \right)
			-  
			\frac{1}{2}
			\g_{\tau}\left(\mathscr{A}^K_{\beta} \mathscr{A}^L_{\alpha} \right)
			\g_K \dot{\upeta}^{\alpha} \g_L \dot{\upeta}^{\beta}.	\label{E:GRADIENTENERGY}
				\end{align}
We now decompose the second term on the right-hand side of~\eqref{E:HIGHORDER2}. 
Specifically, by differentiating the identity 
$\Cof^K_{\nu}\g_K v^{\nu} = \g_{\tau} \mathscr{J}$
(see \eqref{E:PARTIALTAUJISDETERMINEDINTERMSOFDIVU}), we obtain
\[
\Cof^K_{\nu} \g_K \dot{v}^{\nu} = \g_{\tau}\dot{ \mathscr{J}} +\mathcal{R}_{p4},
\]
where $\mathcal{R}_{p4}$ is given by \eqref{E:ENERGYESTERR4}.
Therefore, we have the following identity for the second product on the right-hand side of~\eqref{E:HIGHORDER2}
(see definition \eqref{E:COFMATDEF}):
\begin{align}\label{E:DIVERGENCEENERGY}
& \mathring{F}^2 \Cof^K_{\nu}  \Enth \mathscr{J}^{-3} \frac{1+f}{1-f}  \dot{\mathscr{J}} \g_K \dot{v}^{\nu} \\
& = \frac{1}{2}\g_{\tau}\left(\mathring{F}^2 \Enth \mathscr{J}^{-3} \frac{1+f}{1-f} \dot{\mathscr{J}}^2 \right)
-\frac{1}{2}\g_{\tau}\left(\mathring{F}^2 \Enth \mathscr{J}^{-3} \frac{1+f}{1-f}\right) \dot{\mathscr{J}}^2 \notag \\
& \ \ \ 
+\mathring{F}^2 \Enth \mathscr{J}^{-3} \frac{1+f}{1-f}  \dot{\mathscr{J}} \mathcal{R}_{p4}. \notag
\end{align}	

We now substitute the above identities into equation \eqref{E:CONTRACTED}
and integrate over $[0,\tau] \times \mathfrak{M}$.
We arrive at the desired identity \eqref{E:ENERGYIDENTITY1}
but with the first four inner products 
given by $\langle \cdot \rangle_h$ instead of $\langle \cdot \rangle_g$
and with the term
$
\int_{[0,\tau] \times \mathfrak{M}} \g_K\left(\mathring{F}^2  \Enth \mathscr{J}^{-1} \mathscr{A}^K_{\beta}\mathscr{A}^L_{\alpha}  \g_L \dot{\upeta}^{\beta}\dot{v}^{\alpha}
+\mathring{F}^2 \mathscr{A}^K_{\alpha} 
\Enth \mathscr{J}^{-2} \frac{1+f}{1-f}  \dot{\mathscr{J}}\dot{v}^{\alpha}\right) \,dy d\tau' 
$
in place of 
$
\int_{\mathfrak{M}_{\tau}}\mathring{F}^2  \Enth \mathscr{J}^{-1} \left( \mathscr{A}^0_{\beta}\mathscr{A}^L_{\alpha}  \g_L \dot{\upeta}^{\beta}\dot{v}^{\alpha}
+\mathscr{A}^0_{\alpha}  \mathscr{J}^{-1} \frac{1+f}{1-f}  \dot{\mathscr{J}}\dot{v}^{\alpha}\right) \,dy\Big|^{\tau}_0 
$.
The previous two integrals are in fact equal to each other, 
in view of our assumption that $\mathring{F}=0$ on $\partial\mathfrak{M}_{\tau}$. 
Moreover, by Lemma~\ref{L:EQUIVALENCE}, we may replace 
$\langle \cdot \rangle_g$
with $\langle \cdot \rangle_h$
up to controlled error terms.
We have therefore derived \eqref{E:ENERGYIDENTITY1}.

\vspace{.1 in}
\noindent
{\bf Step 2. Proof of the error estimates \eqref{E:MAINERROR} and \eqref{E:ENERGYESTIMATEREMAININGERRORTERMS}.}
We will show that for $p=0,1,2,3,4$, we have
$
\left|
	\int_{[0,\tau] \times\mathfrak{M}} \mathcal{R}^\mu_{p1} \partial_{\tau}^{2p}\AngMD^{4-p} v_{\mu} \,dy d\tau'
\right|
\leq \mathring{M} + \delta \norm(\tau) +  \tau P(\norm(\tau)),
$
where $\mathcal{R}^\mu_{p1}$ is defined by \eqref{E:RA1MU}
and featured on the right-hand side of~\eqref{E:ENERGYESTERR3}
(and therefore on the right-hand side of~\eqref{E:ENERGYESTERR} as well).
The integrals
$
\int_{[0,\tau] \times\mathfrak{M}} \cdots \,dy d\tau'
$
of the remaining terms on the right-hand side of~\eqref{E:ENERGYESTERR}
can be bounded using similar reasoning, and we omit those details. \ \\

\noindent
{\it The end-point case $p=0$.}
The top-order terms generated by the first sum on the right-hand side of~\eqref{E:RA1MU} are of the form
\[
R_1 = \int_{[0,\tau] \times\mathfrak{M}} g_{\mu\nu}\AngMD^4\mathring{F}  \Enth v^\mu \AngMD^4 v^\nu \,dyd s, \ \ 
R_2 = \int_{[0,\tau] \times\mathfrak{M}} g_{\mu\nu}\mathring{F}  \AngMD^4 \Enth v^\mu \AngMD^4 v^\nu \,dyd s,
\]
while the top-order terms generated by the second sum on the right-hand side of~\eqref{E:RA1MU} are of the form
\[
R_3 =  \int_{[0,\tau] \times\mathfrak{M}}\g_K\left(\mathscr{A}^K_\nu \AngMD^4(\mathring{F}^2) \mathscr{J}^{-1} \Enth\right) \AngMD^4 v^\nu\,dy d \tau', \ 
R_4 = \int_{[0,\tau] \times\mathfrak{M}}\g_K\left(\mathscr{A}^K_\nu \mathring{F}^2 \AngMD^3(\mathscr{J}^{-1} \Enth)\right) \AngMD^4 v^\nu\,dy d \tau'.
\]
We first bound $R_1$. Using the fundamental theorem of calculus, 
for $0 \leq \tau' \leq \tau$
we bound 
the $\| \cdot \|_{L^{\infty}(\mathfrak{M}_{\tau'})}$ norm of
the low-order integrand factor $\Enth$ by
$\leq \mathring{M} + \tau' P(\norm(\tau'))$.
Thus, using an $L^{\infty}-L^2-L^2$ Holder estimate, we deduce
\[
\left|
	R_1
\right| 
\lesssim_{\mathring{M} + \tau P(\norm(\tau))}
 \tau \|\AngMD^4 \mathring{F}\|_{L^2(\mathfrak{M})} \sup_{0 \leq \tau' \leq \tau} \| g_{\mu \nu} v^\mu \AngMD^4 v^\nu \|_{L^2(\mathfrak{M}_{\tau'})}. 
\]
Hence, using \eqref{E:NOTTIMEINTEGRATEDCONTRACTIONAGAINSTUFIRSTSOBOLEVESTIMATE}, we conclude that
\[
\left|
	R_1
\right| 
\leq 
\mathring{M} 
+
\tau P(\norm(\tau))
\]
as desired.
We bound $R_2$ using similar reasoning together with
the simple bound 
$\|\mathring{F} \AngMD^4 \mathscr{J}\|_{L^2(\mathfrak{M}_{\tau'})}^2 \leq P(\norm(\tau'))$.

In our analysis of $R_3$, we will use the bound
\be\label{E:DENSITYHARDY}
\left\|
	\frac{\AngMD^4(\mathring{F}^2)}{\mathring{F}}
\right\|_{L^2(\mathfrak{M})} \leq \mathring{M},
\ee 
which we now prove. To proceed, we compute
$
\displaystyle
\frac{\AngMD^4(\mathring{F}^2)}{\mathring{F}} = 2\AngMD^4\mathring{F} + 4\frac{\AngMD^3\mathring{F}\AngMD\mathring{F}}{\mathring{F}} + 6 \frac{\AngMD^2\mathring{F}\AngMD^2\mathring{F}}{\mathring{F}}
$.
Hence,
\begin{align*}
\left\|
	\frac{\AngMD^4(\mathring{F}^2)}{\mathring{F}}
\right\|_{L^2(\mathfrak{M})} &\lesssim \|\mathring{F}\|_{H^4(\mathfrak{M})} 
+ 
\left\|
	\frac{\AngMD^3\mathring{F}}{\mathring{F}}
\right\|_{L^2(\mathfrak{M})}\|\AngMD\mathring{F}\|_{L^\infty(\mathfrak{M})}
+  
\left\|
	\frac{\AngMD^2\mathring{F}}{\mathring{F}}
\right\|_{L^2(\mathfrak{M})}\|\AngMD^2 \mathring{F}\|_{L^\infty(\mathfrak{M})} \\
& \lesssim \|\mathring{F}\|_{H^4(\mathfrak{M})} \leq \mathring{M},
\end{align*}
where the next-to-last inequality follows from Lemma~\ref{L:HARDY}, 
an application of the $H^2\hookrightarrow L^\infty$ embedding, 
and the initial data assumption $\mathring{F}\in H^4(\mathfrak{M})$.
We have thus proved \eqref{E:DENSITYHARDY}.

To bound $R_3$,
we first integrate by parts
to express it in the simpler form
\[
R_3 = -  \underbrace{\int_{[0,\tau] \times\mathfrak{M}}\mathscr{A}^K_\nu\mathscr{J}^{-1} \AngMD^4(\mathring{F}^2)S \AngMD^4 \g_Kv^\nu
\,dy d \tau'}_{:=R_{3,1}} + \underbrace{\int_{\mathfrak{M}_{\tau'}}\mathscr{A}^0_\nu\mathscr{J}^{-1} \AngMD^4(\mathring{F}^2)S \AngMD^4 v^\nu\,dy\Big|^{\tau'=\tau}_{\tau'=0}}_{ =: R_{3,2}}.
\]
Since $\mathscr{A}^0_\nu v^\nu = \delta^0_0 = 1$,
the same reasoning that led to the proof of \eqref{E:WEIGHTEDCONTRACTIONAGAINSTUFIRSTSOBOLEVESTIMATE}
allows us to deduce that
$\| \mathring{F} \mathscr{A}^0_\nu\AngMD^4v^\nu \|_{L^2(\mathfrak{M}_{\tau'})}
\leq 
\mathring{M} 
+ \tau P(\norm(\tau))$
for $0 \leq \tau' \leq \tau$.
Using \eqref{E:DENSITYHARDY},
again bounding the $\| \cdot \|_{L^{\infty}(\mathfrak{M}_{\tau'})}$ norm 
low-order integrand factors by
$\leq \mathring{M} + \tau P(\norm(\tau))$,
and using Cauchy-Schwarz, we conclude that
\[
|R_{3,2}|\lesssim_{\mathring{M} + \tau P(\norm(\tau))} \tau P(\norm(\tau))
\]
as desired.
To bound $R_{3,1}$, we first integrate by parts in time to obtain the identity
\[
R_{3,1} = \int_{[0,\tau] \times\mathfrak{M}} \g_\tau\left(\Enth \mathscr{A}^K_\nu\mathscr{J}^{-1}\right) \AngMD^4(\mathring{F}^2) \AngMD^4 \g_K\upeta^\nu \,dy d \tau' + 
\int_{\mathfrak{M}_{\tau'}} \Enth \mathscr{A}^K_\nu\mathscr{J}^{-1} \AngMD^4(\mathring{F}^2) \AngMD^4 \g_K\upeta^\nu \,dy d \tau' \Big|^{\tau'=\tau}_{\tau'=0} \,.
\]
We now reason as in our analysis of $R_{3,2}$
and use the estimate \eqref{E:DENSITYHARDY}
to bound 
the above spacetime integral $\int_{[0,\tau] \times\mathfrak{M}} \cdots$
by
$$
\lesssim_{\mathring{M} + \tau P(\norm(\tau))}
\int_{\tau' =0}^{\tau}
	\left\|
		\frac{\AngMD^4(\mathring{F}^2)}{\mathring{F}}
	\right\|_{L^2(\mathfrak{M})}
	\left\|
		\mathring{F} \AngMD^4 \ITIMESMD \upeta
	\right\|_{L^2(\mathfrak{M}_{\tau'})}
\, d \tau'
\lesssim_{\mathring{M} + \tau P(\norm(\tau))}
\tau P(\norm(\tau))\,.
$$
We now reason as in our analysis of $R_{3,2}$
to bound the spatial integral
$\int_{\mathfrak{M}_\tau} \cdots$
on the right-hand side of of the equation $R_{3,1} = \cdots$ above
by
$$
\displaystyle
\lesssim_{\mathring{M} + \tau P(\norm(\tau))}	
	\left\|
		\frac{\AngMD^4(\mathring{F}^2)}{\mathring{F}}
	\right\|_{L^2(\mathfrak{M})}
	\left\|
		\mathring{F} \AngMD^4 \g_K \upeta
	\right\|_{L^2(\mathfrak{M}_{\tau})}
\lesssim_{\mathring{M} + \tau P(\norm(\tau))}
\mathring{M} \norm^{1/2}(\tau) \,.
$$
Using Young's inequality, 
we bound the right-hand side of of the previous
inequality by
$\leq \mathring{M} + \delta \norm(\tau)$
as desired (we have allowed $\mathring{M}$ to depend on $\delta$). Term $R_4$ is estimated analogously to the term $R_3.$

We now analyze the integral generated by the
first below-top-order term in the second sum on the right-hand side of~\eqref{E:RA1MU},
which corresponds to the case $p=0,$ $q_1=3,q_2=0,q_3=1$.
Specifically, we want to bound
\begin{align} \label{E:BELOWTOP1}
\left|
\int_{[0, \tau] \times \mathfrak{M}}
	\g_K\left(\AngMD^3\mathscr{A}^K_\mu\mathring{F}^2 \AngMD (\mathscr{J}^{-1} \Enth)\right) \AngMD^4 v^\mu 
	\, dy d \tau' 
\right|.
\end{align}
If $K=1,2,3$, we integrate by parts with respect to $K$ and then integrate the time derivative away from 
$\AngMD^4 \g_K v^\mu = \AngMD^4 \g_K \partial_{\tau} \upeta^\mu$
to obtain the identity
\begin{align*}
&
\sum_{K=1}^3
\int_{[0,\tau] \times\mathfrak{M}}\g_K \left(\AngMD^3\mathscr{A}^K_\mu\mathring{F}^2 \AngMD (\mathscr{J}^{-1} \Enth)\right) \AngMD^4 v^\mu\, dy d\tau'
	\\ 
& =
\sum_{K=1}^3
\int_{[0,\tau] \times\mathfrak{M}}\g_\tau \left(\AngMD^3\mathscr{A}^K_\mu\mathring{F}^2 \AngMD (\mathscr{J}^{-1} \Enth)\right) \AngMD^4 \partial_K \upeta^\mu\, dyd
\tau' \\
&\ \ 
- 
\sum_{K=1}^3
\int_{\mathfrak{M}_\tau} \AngMD^3 \mathscr{A}^K_\mu\mathring{F}^2 \AngMD (\mathscr{J}^{-1} \Enth) \AngMD^4 \partial_K \upeta^\mu\,dy\Big|_{\tau'=0}^{\tau'=\tau}.
\end{align*}
Using the standard Sobolev calculus,
it is easy to bound the
the magnitude of the spacetime integral on the right-hand side above 
by 
$\leq
 \mathring{M} + \tau P(\norm(\tau))
$.
To bound the the spatial integral $\int_{\mathfrak{M}_\tau} \cdots$ above, 
we first reason as in our analysis of $R_{3,2}$ to bound it by
$
\lesssim_{\mathring{M} + \tau P(\norm(\tau))}
\| 
	\mathring{F} \AngMD^3 \mathscr{A}^K_\mu
\|_{L^2(\mathfrak{M}_{\tau})}
\| 
	\mathring{F} \AngMD^4 \partial_K \upeta^\mu
\|_{L^2(\mathfrak{M}_{\tau})}
$.
By the fundamental theorem of calculus, we have
$
\| 
	\mathring{F} \AngMD^3 \mathscr{A}^K_\mu
\|_{L^2(\mathfrak{M}_{\tau})}
\leq
 \mathring{M} + \tau P(\norm(\tau))
$.
Moreover, we have
$\| 
	\mathring{F} \AngMD^4 \partial_K \upeta^\mu
\|_{L^2(\mathfrak{M}_{\tau})}
\lesssim \norm^{1/2}(\tau)
$.
Hence, by Young's inequality,
the integral over $\mathfrak{M}_\tau$ under consideration is
$\leq \mathring M + \delta \norm(\tau) + \tau P(\norm(\tau))$ 
as desired. If $K=0$, there is no need to integrate by parts as the top-order term 
in~\eqref{E:BELOWTOP1} scales like 
\[
\int_{[0, \tau] \times \mathfrak{M}} \mathring{F}^2 \AngMD^3 \MD v \AngMD (\mathscr{J}^{-1} \Enth) \AngMD^4 v\, dy d \tau',
\]
which is easy to bound in magnitude by
$\leq \tau P(\norm(\tau))$. 
In the case $p=0,q_1=3,q_2=1,q_3=0$ 
in the second sum on the right-hand side of~\eqref{E:RA1MU},
we can obtain the desired bound
using the same reasoning together with inequality~\eqref{E:DENSITYHARDY}.
The remaining cases $q_2=3$ or $q_3=3$ are straightforward to handle via standard Sobolev embeddings; 
we refer the reader to~\cite{CoutandShkoller2012} for an analogous proof in the non-relativistic case.

Finally, we bound the integral generated by the last term on the right-hand side of~\eqref{E:RA1MU}. 
The top-order term occurs when all derivatives fall on $\partial_{\tau} \mathscr{J}$, in which case
we integrate by parts in $\AngMD$ to obtain the top-order integral
\[
\int_{[0,\tau] 
	\times \mathfrak{M}}g_{\mu\nu}v^{\mu}\mathring{F}\mathscr{J}^{-2}\Enth^{3/2}\AngMD^3 \partial_{\tau}\mathscr{J} \AngMD^5 v^\nu
\, dy d \tau'. 
\]
Next, we note that
$\AngMD^3 \partial_{\tau} \mathscr{J} 
= - \frac{1}{2} \mathscr{J}^3 \AngMD^3 \partial_{\tau} (\mathscr{J}^{-2})
$
plus products of terms involving $\leq 4$ derivatives of $\upeta$.
Hence, reasoning as in our analysis of $R_{3,2}$,
and using Young's inequality, 
we bound the magnitude of the above integral by 
$\lesssim_{\mathring{M} + \tau P(\norm(\tau))}
	\tau P(\norm(\tau))
	+
	\int_{[0,\tau]}
			\| 
				\sqrt{\mathring{F}} v_{\nu} \AngMD^5 v^\nu
			\|_{L^2(\mathfrak{M}_{\tau'})}^2
			+
			\| 
				\sqrt{\mathring{F}} \AngMD^3 \partial_{\tau} (\mathscr{J}^{-2})
			\|_{L^2(\mathfrak{M}_{\tau'})}
	\, d \tau'	
$.
To bound the time integral of 
$\| 
	\sqrt{\mathring{F}} v_{\nu} \AngMD^5 v^\nu
\|_{L^2(\mathfrak{M}_{\tau'})}^2$ by 
$\leq \tau P(\norm(\tau))$ as desired,
we use \eqref{E:NOTTIMEINTEGRATEDCONTRACTIONAGAINSTUSECONDSOBOLEVESTIMATE}.
To bound the time integral of the remaining term, we will use the following interpolation estimate:
\begin{align*}
\int_{\tau'=0}^{\tau}
		\| 
			\sqrt{\mathring{F}} \AngMD^3 \partial_{\tau} (\mathscr{J}^{-2})
		\|_{L^2(\mathfrak{M}_{\tau'})}^2
	\, d \tau'
& \lesssim
\int_{\tau'=0}^{\tau}
	\| 
		\AngMD^3 (\mathscr{J}^{-2})
	\|_{L^2(\mathfrak{M}_{\tau'})}
	\| 
		\mathring{F} \AngMD^3 \partial_{\tau}^2 (\mathscr{J}^{-2})
	\|_{L^2(\mathfrak{M}_{\tau'})}
\, d \tau'	
	\\
& \ \
	+ 
	\mathring{M}
	+
 \| 
		\AngMD^3 (\mathscr{J}^{-2})
	\|_{L^2(\mathfrak{M}_{\tau})}
	\| 
		\mathring{F} \AngMD^3 \partial_{\tau} (\mathscr{J}^{-2})
	\|_{L^2(\mathfrak{M}_{\tau})}.
\end{align*}
In view of the definition~\eqref{E:NORM} of $\norm$,
we see that the time integral on the right-hand side of of the above inequality 
is $\leq \tau P(\norm(\tau))$ as desired. 
Moreover, the below-top-order term
$
\| 
	\mathring{F} \AngMD^3 \partial_{\tau} (\mathscr{J}^{-2})
\|_{L^2(\mathfrak{M}_{\tau})}
$
is controllable via the fundamental theorem of calculus and thus by Young's inequality,
we have
$
\| 
	\AngMD^3 (\mathscr{J}^{-2})
\|_{L^2(\mathfrak{M}_{\tau})}
\| 
	\mathring{F} \AngMD^3 \partial_{\tau} (\mathscr{J}^{-2})
\|_{L^2(\mathfrak{M}_{\tau})}
\leq
\mathring{M} 
+ \delta \norm(\tau)	
+ \tau P(\norm(\tau))
$
as desired. 
We have thus bounded the integral generated by the last term on the right-hand side of~\eqref{E:RA1MU}
All the remaining terms arising due to the application of the product rule in the last term on the right-hand side of~\eqref{E:RA1MU} are lower-order. Thus, using standard energy estimates and the bound 
$\|\frac{\AngMD^q \mathring{F}}{\mathring{F}}\|_{L^2(\mathfrak{M})}
\lesssim \| \mathring{F} \|_{H^4(\mathfrak{M})}$, valid for $q=1,2,3$, 
yield that the corresponding spacetime integral is bounded in magnitude by
$\leq \tau P(\norm(\tau))$ as desired.\\

We have thus proved the desired bound
$
\left|
	\int_{[0,\tau] \times\mathfrak{M}} \mathcal{R}^\mu_{01} \partial_{\tau}^{2p}\AngMD^{4-p} v_{\mu} \,dy d\tau'
\right|
\leq \mathring{M} + \delta \norm(\tau) +  \tau P(\norm(\tau)).
$

\vspace{.1 in}
\noindent
{\it The end-point case $p=4$.} When $p=4,$ the corresponding error term $\mathcal R^\mu_{41}$ 
from \eqref{E:RA1MU} takes the form
\begin{align}
\mathcal R^\mu_{41} = & \sum_{p=1}^8 C_p\g_\tau^p(\mathring{F} S) \g_\tau^{9-p} v^\mu 
 + \sum_{p=1}^7 g^{\mu\nu}\g_K\left(\g_\tau^p\mathscr{A}^K_\nu \mathring{F}^2 \g_\tau^{8-p} (\mathscr{J}^{-1} \Enth)\right) \notag \\
& + 2\partial_{\tau}^8 \left( v^{\mu}\mathring{F}\mathscr{J}^{-2}\Enth^{3/2}\partial_{\tau}\mathscr{J}\right). \label{E:RA=41MU}
\end{align}
Since $\mathring{F}$ is $\tau$-independent, there are no top-order derivatives falling on $\mathring{F}$, 
which is different than the end-point case $p=0$.
Therefore, the spacetime integrals generated by the first and the third term on the right-hand side of~\eqref{E:RA=41MU} 
can easily be bounded by using the Cauchy-Schwarz inequality and the definition~\eqref{E:NORM} of $\norm$.
To bound the spacetime integral generated by the second term on the right-hand side of~\eqref{E:RA=41MU}, 
we isolate the most challenging top-order terms:
\begin{align*}
I_1 
& = \int_{[0, \tau] \times \mathfrak{M}} \g_K\left(\g_\tau^{7}\mathscr{A}^K_\nu \mathring{F}^2 \g_\tau(\mathscr{J}^{-1} \Enth)\right)\g_\tau^8 v^\nu \, dy d \tau', 
	\\
I_2 
& =  \int_{[0, \tau] \times \mathfrak{M}}\g_K\left(\g_\tau\mathscr{A}^K_\nu \mathring{F}^2 \g_\tau^{7} (\mathscr{J}^{-1} \Enth)\right) \g_\tau^8 v^\nu \, dy d \tau'.
\end{align*}
To bound $I_1$ and $I_2$, we separately analyze the cases $K=0$ and $K > 0$. If $K=0,$ then by the Cauchy-Schwarz inequality
and the definition~\eqref{E:NORM} of $\norm$, we obtain
\[
\big|I_1\big|+\big|I_2\big| \lesssim \tau P(\norm(\tau)).
\]
If $K=1,2,3$,
we first integrate by parts in $K$ and then again in time to move one $\g_\tau$ derivative away from $\g_K\g_\tau^8 v$,
which yields the identity
\begin{align*}
I_1 
&= 
\int_{[0,\tau] \times\mathfrak{M}} \g_\tau\left(\g_\tau^{7}\mathscr A^K_\nu \mathring{F}^2 \g_\tau(\mathscr{J}^{-1} \Enth)\right)\g_\tau^8\g_K \upeta^\nu \, dy d \tau'
- \int_{\mathfrak{M}_\tau'} 
		\g_\tau^7 \mathscr{A}^K_\nu \mathring{F}^2 \g_\tau(\mathscr{J}^{-1} \Enth) \g_\tau^8 \g_K\upeta^\nu
	\, dy
	\Big|_{\tau'=0}^{\tau'=\tau}
\end{align*}
We bound the spacetime integral using an $L^2-L^2-L^\infty$ Holder estimate and the definition of the norm $\norm(\tau)$.
To bound the spatial integral $\int_{\mathfrak{M}_\tau}$, 
we first use the fundamental theorem of calculus to obtain 
$\|\mathring F \g_\tau^7 \mathscr A^K_\nu\|\lesssim \mathring M + \tau P(\norm(\tau))$.
Also using Young's inequality and the bound $\|\mathring F \g_\tau^8\g_k\upeta^\nu\|_{L^2(\mathfrak{M}_\tau)}^2\leq C\norm(\tau)$, 
we obtain the desired estimate as follows:
\[
\Big| \int_{\mathfrak{M}_\tau} \g_\tau^7 \mathscr{A}^K_\nu \mathring{F}^2 \g_\tau(\mathscr{J}^{-1} \Enth)\g_\tau^8\g_K \upeta^\nu \, dy \leq \mathring{M} + \delta\norm(\tau)+ \tau P(\norm(\tau)).
\]
The integral $I_2$ can be treated in the same way. 
For bounds for similar error terms 
bounded in the context of non-relativistic fluids, 
we refer the reader to page 596 of~\cite{CoutandShkoller2012}.
The remaining cases $p=2,\dots,6$ on the right-hand side of~\eqref{E:RA=41MU}
are lower-order and can easily be treated with Holder's inequality and Sobolev embedding.

\vspace{.1 in}
\noindent
{\it The remaining cases $p=1,2,3$.} 
In the remaining cases $p=1,2,3$, the spacetime integrals
$
\left|
	\int_{[0,\tau]\times\mathfrak{M}}\mathcal R^\mu_{p1}\g_\tau^{2p}\AngMD^{4-p}v_\mu \,dy d \tau'
\right|
$ 
can be bounded by using arguments similar to those that we used in the 
endpoint cases $p=0,4$; we omit the details.
We have thus shown the desired bound
$
\left|
	\int_{[0,\tau] \times\mathfrak{M}} \mathcal{R}^\mu_{p1} \,dy d\tau'
\right|
\leq \mathring{M} + \delta \norm(\tau) + \tau P(\norm(\tau))
$.

It remains for us to prove \eqref{E:ENERGYESTIMATEREMAININGERRORTERMS}.
All integrals $\int_{\mathfrak{M}_0} \cdots \, dy$ on the right-hand side of~\eqref{E:ENERGYIDENTITY1}
are trivially bounded by $\leq \mathring{M}$.
Next, we use Prop.~\ref{P:VORTICITYESTIMATES} to deduce 
the following desired bound for the second integral on the right-hand side of~\eqref{E:ENERGYIDENTITY1}:
\begin{align*}
\left|
	\int_{\mathfrak{M}_{\tau}} \mathring{F}^2  \Enth \mathscr{J}^{-1} \langle\Lagvort \dot{\upeta},\Lagvort \dot{\upeta}\rangle_g \, dy
\right| 
\leq \mathring{M} + (\delta + C \varepsilon)\norm(\tau)+ \tau P(\norm(\tau)).
\end{align*}
To bound the third integral on the right-hand side of~\eqref{E:ENERGYIDENTITY1},
we use Young's inequality, an
$L^\infty-L^2-L^2$  Holder's  inequality,
and the smallness assumption~\eqref{E:SMALLDENSITY} to conclude the desired bound
(where the final constant $C$ is allowed to depend on $\delta^{-1}$)
\begin{align*}
&
\left|
	\int_{\mathfrak{M}_{\tau}} \mathring{F}^2  \Enth \mathscr{J}^{-1} \mathscr{A}^0_{\beta}\mathscr{A}^L_{\alpha}  \g_L \g_\tau^{2p}\AngMD^{4-p}{\upeta}^{\beta}\g_\tau^{2p}\AngMD^{4-p}{v}^{\alpha}
	\, dy
	\right| 
	\\
& \qquad \qquad \qquad \leq \delta \|\mathring{F} \g_\tau^{2p}\AngMD^{4-p} \ITIMESMD \upeta\|_{L^2(\mathfrak{M}_{\tau})}^2 
+ \frac C\delta \|\mathring{F}\|_{L^{\infty}(\mathfrak{M}_{\tau})} 
\|\sqrt{\mathring{F}}\g_\tau^{2p}\AngMD^{4-p} v\|_{L^2(\mathfrak{M}_{\tau})}^2 \\
&\qquad \qquad \qquad  \leq \mathring{M} + \left(\delta + C \sqrt{\varepsilon} \right) \norm(\tau) + \tau P(\norm(\tau)).
\end{align*}
Similarly, we have
\[
\left|
	\int_{\mathfrak{M}_{\tau}} \mathring{F}^2 \Enth \mathscr{J}^{-1} \mathscr{A}^0_{\alpha}  \mathscr{J}^{-1} \frac{1+f}{1-f}  \AngMD^4 {\mathscr{J}}\AngMD^4{v}^{\alpha}
	\, dy
\right| 
\leq \mathring{M} + (\delta + C \sqrt{\varepsilon}) \norm(\tau) + \tau P(\norm(\tau)).
\]
We have therefore proved \eqref{E:ENERGYESTIMATEREMAININGERRORTERMS}
and completed the proof of Prop.~\ref{P:ENERGYLEMMA}.
\end{proof}

%

\section{Bounds for the $\partial_3$ Derivatives via Elliptic Estimates}\label{S:ELLIPTIC}
We now use the previous estimates to establish our main a priori estimates
involving the higher $\partial_3$ derivatives of the solution. 
The main result is Prop.~\ref{P:HIGHERPARTIAL3DERIVATIVESVIAELLIPTIC}.
Here are the main ideas behind the analysis.
\begin{itemize}
	\item From the Hodge-type estimate \eqref{E:HODGEEST},
		the tangential trace inequality \eqref{E:SIMPLETRACE},
		and the fact that we have already established bounds for
		the horizontal and time derivatives of the solution 
		(via energy estimates),
		we see that the desired bounds follow from suitable bounds on
		the derivatives of the vorticity and the divergence.
	\item We have also already established bounds for
	the higher $\partial_3$ derivatives of the vorticity.
	Thus, the primary remaining task is to estimate the 
	higher $\partial_3$ derivatives of the
	divergence $\Flatdiv \underline \upeta$.
	As an intermediate step, we
	show that the desired bounds for $\Flatdiv \underline \upeta$
	would follow from related bounds for the Jacobian determinant $\mathscr{J}$;
	see Lemma~\ref{L:ESTIMATESFORDIV}.
	\item To control $\mathscr{J}$, we use the fact that by virtue of the evolution equations,
	the $\partial_3$ derivatives of
	$\mathscr{J}^{-2}$ can be expressed in terms of quantities that have already
	been bounded; see \eqref{E:SOLVINGFORPARTIAL3DERIVATIVE}.
	At a heuristic level, this is the last estimate needed to close
	the whole process.
	We note that the desired $L^2$ bounds for 
	$\partial_3 (\mathscr{J}^{-2})$
	involve a somewhat subtle integration by parts that relies on the
	physical vacuum condition.
	\item In reality, to close the estimates up to top order, we must use induction
	in the number of $\partial_3$ derivatives, which ensures that
	each new estimate depends only on quantities that have already been
	bounded in the induction.
\end{itemize}

In the next lemma, we provide higher-order bounds for  
$\Flatdiv \underline{\upeta}$ in terms of the higher-order 
Sobolev norms of the Jacobian determinant $\mathscr{J}$. 
We recall that $\underline{\upeta} = (\upeta^1, \upeta^2, \upeta^3)$ 
is the spatial part of the flow map.

\begin{lemma} [\textbf{Bounds for} 
$\| 
			\Flatdiv \g_{\tau}^{2p} \AngMD^q \MD^r \underline{\upeta}
		\|_{L^2(\mathfrak{M}_{\tau})}$]
 \label{L:ESTIMATESFORDIV}
	Let $p,q,r$ be non-negative integers with $p=0,1,2,3$ and $q + r = 3 - p$
	and let $\delta > 0$ be a constant.
	Under the bootstrap assumptions of Sect.~\ref{SS:BOOTSTRAPASSUMPTIONS},
	we have the following estimates for $\tau \in [0,T]$
	(where $\mathring{M}$ is allowed to depend on $\delta^{-1}$):
	\begin{align}\notag
		\| 
			\Flatdiv \g_{\tau}^{2p} \AngMD^q \MD^r \underline{\upeta}
		\|_{L^2(\mathfrak{M}_{\tau})}
		& \leq
		\mathring{M} 
			+ \delta \norm(\tau)
			+ \tau P(\norm(\tau))
				\\
		& \ \ 
		+ \|
				\g_{\tau}^{2p} \AngMD^q \MD^r \mathscr{J}
			\|_{L^2(\mathfrak{M}_{\tau})}
		+ \|
				\Flatvort \g_{\tau}^{2p} \AngMD^q \MD^r \underline{\upeta}
			\|_{L^2(\mathfrak{M}_{\tau})}
			\notag \\
		& \ \ 
		+ 
		\| 
			\g_{\tau}^{2p} \AngMD^{q+1} \MD^r \upeta
		\|_{L^2(\mathfrak{M}_{\tau})}.
			\label{E:ESTIMATESFORDIV}
	\end{align}
\end{lemma}

\begin{proof}
We first claim that for $p=1,2,3$ and $q + r = 3 - p$, we have the following identity:
\begin{align}
	\left(
		1 
		+ 
		\mathscr{A}_0^3 \frac{v^3}{v^0}
	\right)
	\Flatdiv \g_{\tau}^{2p} \AngMD^q \MD^r \underline{\upeta}
	& =  
		\mathscr{J}^{-1} \g_{\tau}^{2p} \AngMD^q \MD^r \mathscr{J}
		- \mathscr{A}_0^3 \sum_{a=1}^3 \frac{v^a}{v^0} (\Flatvort \g_{\tau}^{2p} \AngMD^q \MD^r \underline{\upeta})_{3a}
			\notag \\
		& \ \ 
		+ \mathscr{A}_0^3 \frac{v^3}{v^0} \sum_{a=1}^3\sum_{A=1}^2 \delta_a^{A} \partial_{A} \g_{\tau}^{2p} \AngMD^q \MD^r \upeta^a	
		- \sum_{a=1}^3\sum_{A=1}^2\mathscr{A}_0^{A} \frac{v^a}{v^0} \partial_{A} \g_{\tau}^{2p} \AngMD^q \MD^r \upeta_a
		\notag \\
	& \ \
		- \sum_{a=1}^2\frac{v^a}{v^0} \partial_a \g_{\tau}^{2p} \AngMD^q \MD^r \upeta_3 \notag \\
&	\ \ 	+ O_{L^2(\mathfrak{M}_{\tau})}(\mathring{M} + \tau P(\norm(\tau))).
		 \label{E:FLATDIVIDENTITY}
\end{align}

Similarly when $p=0$ (that is, for $q + r = 3$) we have, 
\begin{align} 
	\left(
		1 
		+ 
		\mathscr{A}_0^3 \frac{v^3}{v^0}
	\right)
	\Flatdiv \AngMD^q \MD^r \underline{\upeta}
	& =  
		\mathscr{J}^{-1} \AngMD^q \MD^r \mathscr{J}
		- \mathscr{A}_0^3 \sum_{a=1}^3\frac{v^a}{v^0} (\Flatvort \AngMD^q \MD^r \underline{\upeta})_{3a}
		\notag	\\
	& \ \
		+
		\frac{1}{v^0} \mathscr{A}_0^3 v_{\alpha} \partial_3 \AngMD^q \MD^r \upeta^{\alpha}
		- \frac{1}{v^0} \mathscr{A}_0^3 \sum_{a=1}^2 v^a \partial_a \AngMD^q \MD^r  {\upeta}_3
		\notag \\
	& \ \ 
		+ \mathscr{A}_0^3 \frac{v^3}{v^0} \sum_{a=1}^3\sum_{A=1}^2\delta_a^{A} \partial_{A} \AngMD^q \MD^r \upeta^a	
		- \sum_{a=1}^3\sum_{A=1}^2\mathscr{A}_0^{A} \frac{v^a}{v^0} \partial_{A} \AngMD^q \MD^r \upeta_a
		\notag \\
	& \ \
		+ O_{L^2(\mathfrak{M}_{\tau})}(\mathring{M} + \tau P(\norm(\tau))).
		\label{E:NOIMEDERIVATIVESFLATDIVIDENTITY}
\end{align}
We first prove \eqref{E:FLATDIVIDENTITY}.
Differentiating \eqref{E:PARTIALTAUJISDETERMINEDINTERMSOFDIVU} 
with $\g_{\tau}^{2p-1} \AngMD^q \MD^r$
using \eqref{E:DETAINVERSEDIFFERENTIATED},
and referring to Definition~\ref{D:DIFFERENTIALOPERATORS},
we obtain that
\begin{align}
\mathscr{J}^{-1} \g_{\tau}^{2p} \AngMD^q \MD^r \mathscr{J} 
& = \mathscr{A}_{\alpha}^K \partial_K \g_{\tau}^{2p} \AngMD^q \MD^r \upeta^{\alpha}
	+ O_{L^2(\mathfrak{M}_{\tau})}(\mathring{M} + \tau P(\norm(\tau)))
	\notag \\
& = 
	\Flatdiv \g_{\tau}^{2p} \AngMD^q \MD^r \underline{\upeta}
	+
	\sum_{A=1}^3\mathscr{A}_0^{A} \partial_{A} \g_{\tau}^{2p} \AngMD^q \MD^r \upeta^0
	+\sum_{A=1}^3
	(\mathscr{A}_a^{A} 
	- \delta_a^{A}
	)
	\partial_{A}
	\g_{\tau}^{2p} \AngMD^q \MD^r \upeta^a
	\notag \\
& \ \ \
	+ \mathscr{A}_{\alpha}^0 \g_{\tau}^{2p+1} \AngMD^q \MD^r \upeta^{\alpha}
	+ O_{L^2(\mathfrak{M}_{\tau})}(\mathring{M} + \tau P(\norm(\tau)))
	\notag
		\\
& = 
	\Flatdiv \g_{\tau}^{2p} \AngMD^q \MD^r \underline{\upeta}
	+
	\sum_{A=1}^3\mathscr{A}_0^{A} \partial_{A} \g_{\tau}^{2p} \AngMD^q \MD^r \upeta^0
	\notag \\
	& \ \ \
	+ O_{L^2(\mathfrak{M}_{\tau})}(\mathring{M} + \tau P(\norm(\tau))).
	 \label{E:HIGHERJACDETDERIVATIVES}
\end{align}
Differentiating \eqref{E:NORMALIZATIONLAGRANGIAN} with $\g_{\tau}^{2p-1} \AngMD^q \MD^r$,
we rewrite the term $\sum_{A=1}^3\mathscr{A}_0^{A} \partial_{A} \g_{\tau}^{2p} \AngMD^q \MD^r \upeta^0$ from
the first line of the right-hand side of \eqref{E:HIGHERJACDETDERIVATIVES}
as follows:
\begin{align}
	\sum_{A=1}^3\mathscr{A}_0^{A} \partial_{A} \g_{\tau}^{2p} \AngMD^q \MD^r \upeta^0
	& =
	\sum_{A,a=1}^3\mathscr{A}_0^{A} \frac{v^a}{v^0} \partial_{A} \g_{\tau}^{2p} \AngMD^q \MD^r \upeta_a
		+  O_{L^2(\mathfrak{M}_{\tau})}(\mathring{M} + \tau P(\norm(\tau))).
\notag 		\\
	& 
	= 
	\sum_{a=1}^3\mathscr{A}_0^3 \frac{v^a}{v^0} \partial_3 \g_{\tau}^{2p} \AngMD^q \MD^r \upeta_a
	+
	\sum_{A=1}^2\sum_{a=1}^3\mathscr{A}_0^{A} \frac{v^a}{v^0} \partial_{A} \g_{\tau}^{2p} \AngMD^q \MD^r \upeta_a \notag \\
	& \ \ \ 
	+ O_{L^2(\mathfrak{M}_{\tau})}(\mathring{M} + \tau P(\norm(\tau))) \,.
	 \label{E:PARTIAL3TERMRFIRSTREWRITING}
\end{align}
Again referring to Definition~\ref{D:DIFFERENTIALOPERATORS}, 
we rewrite the factor 
$\frac{v^a}{v^0} \partial_3 \g_{\tau}^{2p} \AngMD^q \MD^r \upeta_a$ from the first product on the
right-hand side of \eqref{E:PARTIAL3TERMRFIRSTREWRITING} as follows:
\begin{align} 
\sum_{a=1}^3\frac{v^a}{v^0} \partial_3 \g_{\tau}^{2p} \AngMD^q \MD^r \upeta_a
& = \sum_{a=1}^3\frac{v^a}{v^0} (\Flatvort \g_{\tau}^{2p} \AngMD^q \MD^r \underline{\upeta})_{3a}
+\sum_{a=1}^2\frac{v^a}{v^0} \partial_a \g_{\tau}^{2p} \AngMD^q \MD^r \upeta_3+ \frac{v^3}{v^0} \partial_3 \g_{\tau}^{2p} \AngMD^q \MD^r \upeta_3
\notag	\\
& =\sum_{a=1}^3 \frac{v^a}{v^0} (\Flatvort \g_{\tau}^{2p} \AngMD^q \MD^r \underline{\upeta})_{3a}
+ \frac{v^3}{v^0} \Flatdiv \g_{\tau}^{2p} \AngMD^q \MD^r \underline{\upeta} \notag \\
& \ \ 
- \sum_{a=1}^3\sum_{A=1}^2\frac{v^3}{v^0} \delta_a^{A} \partial_{A} \g_{\tau}^{2p} \AngMD^q \MD^r \upeta^a+\sum_{a=1}^2\frac{v^a}{v^0} \partial_a \g_{\tau}^{2p} \AngMD^q \MD^r \upeta_3.
\label{E:PARTIAL3FACTORREWRITING}
\end{align}

Combining 
\eqref{E:HIGHERJACDETDERIVATIVES},
\eqref{E:PARTIAL3TERMRFIRSTREWRITING},
and \eqref{E:PARTIAL3FACTORREWRITING}
and carrying out simple algebraic computations, we
conclude \eqref{E:FLATDIVIDENTITY} in the case that $p=1,2,3$.

The proof of \eqref{E:NOIMEDERIVATIVESFLATDIVIDENTITY} is similar. 
It is based on the general Jacobian differentiation identity \eqref{E:JACOBIANDETDIFFERNTIATED}
as opposed to the special case of \eqref{E:PARTIALTAUJISDETERMINEDINTERMSOFDIVU}.
Moreover, in the analog of the step \eqref{E:PARTIAL3TERMRFIRSTREWRITING},
we no longer differentiate \eqref{E:NORMALIZATIONLAGRANGIAN}.
Instead, we use the algebraic decomposition
\begin{align} \label{E:ALGEBRAICDECOMPOFETA0THREEDERIVATIVE}
\mathscr{A}_0^3 \partial_3 \AngMD^q \MD^r \upeta^0
& = - \frac{1}{v^0}\mathscr{A}_0^3 v_{\alpha} \partial_3 \AngMD^q \MD^r \upeta^{\alpha}	
		+ \frac{1}{v^0} \mathscr{A}_0^3 v^a \partial_3 \AngMD^q \MD^r \underline{\upeta}_a.
\end{align}
We then algebraically express the last product on the right-hand side of \eqref{E:ALGEBRAICDECOMPOFETA0THREEDERIVATIVE}
in terms of $\Flatvort \AngMD^q \MD^r \underline{\upeta}$, 
$\Flatdiv \AngMD^q \MD^r \underline{\upeta}$, etc.,
like in~\eqref{E:PARTIAL3TERMRFIRSTREWRITING}.
The remaining calculations are similar to the remaining ones in the previous case;
we omit those details.

To prove~\eqref{E:ESTIMATESFORDIV}, we first observe that 
by \eqref{E:INITIALINVERSECHOVMATRIX},
we have that
$$\left[1 + \mathscr{A}_0^3 \frac{v^3}{v^0}\right]|_{\tau=0} 
= \frac{1 + (\mathring{v}^1)^2 + (\mathring{v}^2)^2}{1 + (\mathring{v}^1)^1 + (\mathring{v}^2)^2 + (\mathring{v}^3)^2}\,,$$
which is bounded from above and uniformly from below away from $0$ (depending on the data).

We first study the cases $p=1,2,3$.  We compute the norm $\| \cdot \|_{L^2(\mathfrak{M}_{\tau})}$
	of equation \eqref{E:FLATDIVIDENTITY}. 
	Up to $O_{L^2(\mathfrak{M}_{\tau})}(\mathring{M} + \tau P(\norm(\tau)))$ errors, 
	we may use the fundamental theorem of calculus
	to replace all undifferentiated quantities
	in equation \eqref{E:FLATDIVIDENTITY},
	such as the factor
	$
		1 
		+ 
		\mathscr{A}_0^3 \frac{v^3}{v^0}
	$
	on the left-hand side of \eqref{E:FLATDIVIDENTITY}, 
	with their initial values (at $\tau =0$).
	The desired estimate \eqref{E:ESTIMATESFORDIV} follows easily from these observations.
	
	In the remaining case $p=0$, we
compute the norm $\| \cdot \|_{L^2(\mathfrak{M}_{\tau})}$
	of equation \eqref{E:NOIMEDERIVATIVESFLATDIVIDENTITY}
	and argue similarly, this time
	using the estimate \eqref{E:CONTRACTIONAGAINSTUFIRSTSOBOLEVESTIMATE}
	to bound the term
	$\frac{1}{v^0} \mathscr{A}_0^3 v_{\alpha} \partial_3 \AngMD^q \MD^r \upeta^{\alpha}$
	on the second line of the right-hand side of \eqref{E:NOIMEDERIVATIVESFLATDIVIDENTITY}.
	\end{proof}

In the next lemma, we bound the terms corresponding to the tangential boundary
term in the elliptic estimates, that is, the corresponding to the last term on the right-hand side of
\eqref{E:HODGEEST}.

\begin{lemma}[\textbf{Estimates for the tangential trace term}]
\label{L:SIMPLETRACE}
Let $p,q,r$ be non-negative integers with
$p = 0,1,2,3$ and $q + r = 3-p$.
Under the bootstrap assumptions of Sect.~\ref{SS:BOOTSTRAPASSUMPTIONS},
we have the following estimates for $\tau \in [0,T]$:
\begin{align} \label{E:SIMPLETRACE}
	\sum_{a=1}^2
	\| 
		\AngMD \partial_{\tau}^{2p} \AngMD^q \MD^r \upeta^a
	\|_{H^{-0.5}(\partial \mathfrak{M}_{\tau})}
	& \leq
		\mathring{M}
		+
		\tau P(\norm(\tau))
			\\
		& \ \ 
		+
		\| 
			\mathring{F} \partial_{\tau}^{2p} \AngMD^{q+1} \MD^r \upeta 
		\|_{L^2(\mathfrak{M}_{\tau})}	
		+
		\| 
			\mathring{F} \partial_{\tau}^{2p} \AngMD^{q+1} \MD^{r+1} \upeta 
		\|_{L^2(\mathfrak{M}_{\tau})}	.
		\notag
\end{align}

\end{lemma}

\begin{proof}
	By Lemma~\ref{L:TANGENTIALTRACE}, we have that 
	\begin{align} \label{E:TANGENTIALTRACEFIRSTAPPLICATIONELLIPTIC}
		\mbox{LHS \eqref{E:SIMPLETRACE}}
		& \lesssim_{\mathring{M}}
			\| 
				\Flatvort \partial_{\tau}^{2p} \AngMD^{q+1} \MD^r \upeta 
			\|_{H^1(\mathfrak{M}_\tau)'}
			+
			\| 
				\partial_{\tau}^{2p} \AngMD^{q+1} \MD^r \upeta 
			\|_{L^2(\mathfrak{M}_\tau)}.
	\end{align}
	Because we may integrate by parts in $\AngMD$ against test functions without incurring boundary terms, 
	it follows that
	the first term on the right-hand side of \eqref{E:TANGENTIALTRACEFIRSTAPPLICATIONELLIPTIC}
	is $\lesssim \| 
				\Flatvort \partial_{\tau}^{2p} \AngMD^q \MD^r \upeta 
			\|_{L^2(\mathfrak{M}_\tau)}
			\leq
			$
			right-hand side of \eqref{E:SIMPLETRACE} as desired,
			where in the last step, we have used \eqref{E:FLATMAINVORTICITYESTIMATES}.

	To bound the last term on the right-hand side of \eqref{E:TANGENTIALTRACEFIRSTAPPLICATIONELLIPTIC}
	by $\lesssim$ right-hand side of \eqref{E:SIMPLETRACE},
	we use the weighted embedding result \eqref{E:WEIGHTEDEMBEDDING} with $k=1$:
	\begin{align} \label{E:WEIGHTEDEMBEDDINGAPPLIEDINTRACECONTEXT}
		\| 
			\partial_{\tau}^{2p} \AngMD^{q+1} \MD^r \upeta 
		\|_{L^2(\mathfrak{M}_\tau)}
		& \lesssim_{\mathring{M}}
		\| 
			\mathring{F} \partial_{\tau}^{2p} \AngMD^{q+1} \MD^r \upeta 
		\|_{L^2(\mathfrak{M}_\tau)}
		+
		\| 
			\mathring{F} \partial_{\tau}^{2p} \AngMD^{q+1} \MD^{r+1} \upeta 
		\|_{L^2(\mathfrak{M}_\tau)}.
	\end{align}
	We have thus proved \eqref{E:SIMPLETRACE}.
	\end{proof}

\subsection{The main estimates}
\label{SS:ELLITPICMAINEST}

We now use the previous estimates to establish our main a priori estimates
involving the higher $\partial_3$ derivatives of the solution.

\begin{proposition}\label{P:ELLIPTIC}
	Let $\delta > 0$ be a constant and let $\varepsilon$ be as in \eqref{E:SMALLDENSITY}.
	Under the bootstrap assumptions of Sect.~\ref{SS:BOOTSTRAPASSUMPTIONS},
	we have the following estimates for $\tau \in [0,T]$
	(where $\mathring{M}$ is allowed to depend on $\delta^{-1}$):
\label{P:HIGHERPARTIAL3DERIVATIVESVIAELLIPTIC}
\begin{align} \label{E:MAINTRANSVERSALDERIVATIVESESTIMATE}
& \sum_{p=0}^3 
		\| 
			\g_{\tau}^{2p} \upeta 
		\|_{H^{4-p}(\mathfrak{M}_{\tau})}^2 
+ \sum_{p=0}^3 
		\|
			\g_{\tau}^{2p} (\mathscr{J}^{-2})
		\|_{H^{3-p}(\mathfrak{M}_{\tau})}^2 
+ \sum_{p=0}^3 
		\|
			\mathring{F} \g_{\tau}^{2p} (\mathscr{J}^{-2})
		\|_{H^{4-p}(\mathfrak{M}_{\tau})}^2 
\\
&  \leq
	\mathring{M} 
	+ (\delta + C \varepsilon) \norm(\tau)
	+ \tau P(\norm(\tau)). 
\notag
\end{align}
\end{proposition}

\begin{proof}
{\bf Step 1. A high-order identity involving $\mathscr{J}^{-2}$.}
Let $p,q,r$ be non-negative integers with $p \in \lbrace 0,1,2,3 \rbrace$ and $q + r = 3-p$.  
Applying the operator $\g_\tau^{2p}\AngMD^q\MD^r$ to~\eqref{E:LAGRANGIANENTHALPYEVOLUTIONNEEDEDFORELLIPTIC} we arrive at the following higher-order version of it (for $\mu = 0,1,2,3$):	
\begin{align} \label{E:SOLVINGFORPARTIAL3DERIVATIVE}
		\mathring{F} \Cof^3_{\mu} \g_3 \partial_{\tau}^{2p} \AngMD^q \MD^r (\mathscr{J}^{-2})
		+
		2 \Cof^3_{\mu} (\g_3 \mathring{F}) \partial_{\tau}^{2p} \AngMD^q \MD^r (\mathscr{J}^{-2})
		 & = \mathfrak{I}_{\mu}^{p;q;r},
	\end{align}
where the error term $\mathfrak{I}_{\mu}^{p;q;r}$ is given by 
\begin{align}
		 \mathfrak{I}_{\mu}^{p;q;r} 
		 & = 
			- 
			\Enth^{-1/2} \partial_{\tau}^{2p+2} \AngMD^q \MD^r \upeta_{\mu}
			+ 2 \mathscr{J}^{-2}(\partial_{\tau} \upeta_{\mu}) (\mathring{F}  \partial_{\tau}^{2p+1} \AngMD^q \MD^r \mathscr{J}) 
				\label{E:INHOMSOLVINGFORPARTIAL3DERIVATIVE} \\
				& \ \
			-
			\left\lbrace
				\mathring{F} \g_3(\mathscr{J}^{-2}) 
 				+
 				2 (\g_3 \mathring{F}) \mathscr{J}^{-2} 
 			\right\rbrace
			(\partial_{\tau}^{2p} \AngMD^q \MD^r a^3_{\mu})
			\notag \\
		& 
			\ \
			- \left[\g_{\tau} \left(\mathscr{J}^{-2}\right)\right] (\mathring{F} \partial_{\tau}^{2p} \AngMD^q \MD^r a^0_{\mu}) 
			- \sum_{A=1}^2 
				\left[\g_{A} \left(\mathscr{J}^{-2}\right)\right] (\mathring{F} \partial_{\tau}^{2p} \AngMD^q \MD^r a^{A}_{\mu}) 
				\notag \\
		& \ \
			- \sum_{A=1}^2 
				\Cof^{A}_{\mu} \left[\mathring{F} \g_{A} \partial_{\tau}^{2p} \AngMD^q \MD^r \left(\mathscr{J}^{-2}\right)\right]
				\notag \\
 		& \ \
 			- 2 \sum_{A=1}^2 \frac{\partial_{A} \mathring{F}}{\mathring{F}} \mathscr{J}^{-2}  
 				(\mathring{F} \partial_{\tau}^{2p} \AngMD^q \MD^r \Cof_{\mu}^{A})
 		- 2 \sum_{A=1}^2 
 				\frac{\partial_{A} \mathring{F}}{\mathring{F}} 
 				\Cof_{\mu}^{A}
 				(\mathring{F} \partial_{\tau}^{2p} \AngMD^q \MD^r \mathscr{J}^{-2}).
 					\notag \\
		& \ \ +  O_{L^2(\mathfrak{M}_{\tau})}(\mathring{M} + \tau P(\norm(\tau))). \notag
 	\end{align}
The term $O_{L^2(\mathfrak{M}_{\tau})}(\mathring{M} + \tau P(\norm(\tau)))$ on the right-hand side of~\eqref{E:INHOMSOLVINGFORPARTIAL3DERIVATIVE} accounts for
all of the lower-order terms arising from the application of the Leibniz rule,
which can be bounded via the fundamental theorem of calculus,
as we explained in Sect.~\ref{SS:FTCESTIMATES}.
We emphasize that terms containing high-order spatial derivatives of $\mathring{F}$ can be estimated using Lemma~\ref{L:HARDY},
much like in our proof of~\eqref{E:DENSITYHARDY}.
Note that due to the smallness assumption~\eqref{E:SMALLDENSITY}, 
the norm $\| \cdot \|_{L^2(\mathfrak{M}_{\tau})}$ 
of all the terms appearing in the third, fourth, and the fifth line 
of~\eqref{E:INHOMSOLVINGFORPARTIAL3DERIVATIVE} can be bounded by $C\varepsilon \norm(\tau),$ where we keep in mind that the cofactor matrix $\Cof^K_\mu$ scales like $\ITIMESMD \upeta$ from
the point of view of derivative count. 
Furthermore, we note that the cofactor matrix entries $\Cof^3_\mu$ are polynomials in the components of
$\ITIMESMD \upeta$ with a special structure: 
only {\em purely tangential} derivatives $\g_\tau,\AngMD$ of $\upeta$ are present in the products.
Therefore, we have
\begin{align}
\|\mathfrak{I}_{\mu}^{p;q;r} \|_{L^2(\mathfrak{M}_{\tau})} \lesssim &\mathring{M} + (\delta+\varepsilon)\norm(\tau)+ \tau P(\norm(\tau))  \notag \\
& + \|\g_\tau^{2p+2}\AngMD^q\MD^r \upeta_\mu\|_{L^2(\mathfrak{M}_{\tau})} + \|\mathring{F}\g_\tau^{2p+1}\AngMD^q\MD^r \mathscr{J}\|_{L^2(\mathfrak{M}_{\tau})} 
+ \|\g_\tau^{2p}\AngMD^{q+1}\MD^r \upeta\|_{L^2(\mathfrak{M}_{\tau})} . \label{E:IPQRESTIMATE}
\end{align}
 
\medskip 
 
\noindent
{\bf Step 2. Integration-by-parts estimate.}

Using \eqref{E:SOLVINGFORPARTIAL3DERIVATIVE}, we compute 
\begin{align} \label{E:UNUSUALIBPID}
 \frac{1}{g^{\mu \nu} \Cof^3_{\mu} a^3_{\nu}}
 \mathfrak{I}^{p;q;r}_{\alpha} \mathfrak{I}^{p;q;r \ \alpha}
 & = 2 (\g_3\mathring{F})^2 \left(\partial_{\tau}^{2p} \AngMD^q \MD^r \left(\mathscr{J}^{-2}\right)\right)^2
 + \mathring{F}^2\left(\partial_{\tau}^{2p} \AngMD^q \MD^r \g_3\left(\mathscr{J}^{-2}\right)\right)^2 
 	\\
& \ \
	- 2 (\g_3^2 \mathring{F}) \mathring{F} \left(\partial_{\tau}^{2p} \AngMD^q \MD^r (\mathscr{J}^{-2})\right)^2
	+ 2 \g_3 
			\left\lbrace
				\mathring{F} (\g_3 \mathring{F}) \left(\partial_{\tau}^{2p} \AngMD^q \MD^r (\mathscr{J}^{-2})\right)^2
			\right\rbrace.
	\notag
\end{align}
We now integrate \eqref{E:UNUSUALIBPID} over $\mathfrak{M}_{\tau}$.
The integrals of the first two terms on the right-hand side of \eqref{E:UNUSUALIBPID}
sum to
$\geq  
2 (\min_{\mathfrak{M}} |\g_3 \mathring{F}|^2)
\| \partial_{\tau}^{2p} \AngMD^q \MD^r \left(\mathscr{J}^{-2}\right)\|_{L^2(\mathfrak{M}_{\tau})}^2
+ \|\mathring{F} \partial_{\tau}^{2p} \AngMD^q \MD^r \g_3 \left(\mathscr{J}^{-2}\right)\|_{L^2(\mathfrak{M}_{\tau})}^2
$,
where $\min_{\mathfrak{M}} |\g_3 \mathring{F}| > 0$ in view of the physical boundary condition.
Next, using 
$\| \g_3^2 \mathring{F} \|_{L^{\infty}(\mathfrak{M})} 
\lesssim
\| \mathring{F} \|_{H^4(\mathfrak{M})} 
\leq \mathring{M}$,
Cauchy-Schwarz,
and Young's inequality, we deduce that for any number $\theta > 0$, we have
\begin{align*}
\int_{\mathfrak{M}_\tau}(\g_3^2 \mathring{F}) \mathring{F} \left(\partial_{\tau}^{2p} \AngMD^q \MD^r (\mathscr{J}^{-2})\right)^2
&\le \theta \|\partial_{\tau}^{2p} \AngMD^q \MD^r (\mathscr{J}^{-2})\|_{\LM}^2 + C_\theta\|\mathring{F} \partial_{\tau}^{2p} \AngMD^q \MD^r (\mathscr{J}^{-2})\|_{\LM}^2 \\
& \le \theta \|\partial_{\tau}^{2p} \AngMD^q \MD^r (\mathscr{J}^{-2})\|_{\LM}^2 + \mathring{M} + \tau P(\norm(\tau)),
\end{align*}
where $\mathring{M}$ above is allowed to depend on $\theta^{-1}$ and
we used the fundamental theorem of calculus to bound $\|\mathring{F} \partial_{\tau}^{2p} \AngMD^q \MD^r (\mathscr{J}^{-2})\|_{\LM}^2$.
Finally, we note that by \eqref{E:BOOTSTRAPCOF3ONEFORMPOSITIVITY}, the factor
$
\displaystyle 
\frac{1}{g^{\mu \nu} a^3_{\mu} a^3_{\nu}}
$
on LHS \eqref{E:UNUSUALIBPID} is $\leq \mathring{M}$ in $\| \cdot \|_{L^{\infty}(\mathfrak{M}_{\tau})}$.
Combining the above estimates and choosing 
$\theta = \min_{\mathfrak{M}} |\g_3 \mathring{F}|^2 < 2 \min_{\mathfrak{M}} |\g_3 \mathring{F}|^2$,
we conclude that 
\begin{align} 
\|\mathring{F} \partial_{\tau}^{2p} \AngMD^q \MD^{r+1} \left(\mathscr{J}^{-2}\right)\|_{L^2(\mathfrak{M}_{\tau})}^2
& +
\| \partial_{\tau}^{2p} \AngMD^q \MD^r \left(\mathscr{J}^{-2}\right)\|_{L^2(\mathfrak{M}_{\tau})}^2 
 	\notag \\
& \lesssim \mathring{M} + \tau P(\norm(\tau)) +\int_{\mathfrak{M}_{\tau}} \left|\mathfrak{I}^{p;q;r}_{\alpha} \mathfrak{I}^{p;q;r\,\alpha}\right|\, dy.
	\label{E:PRELIMINARYESTIMATESOLVINGFORPARTIAL3DERIVATIVE}
\end{align}
Using~\eqref{E:IPQRESTIMATE} we infer that 
\begin{align} 
& \|\mathring{F} \partial_{\tau}^{2p} \AngMD^q \MD^{r+1} \left(\mathscr{J}^{-2}\right)\|_{L^2(\mathfrak{M}_{\tau})}^2
+ \| \partial_{\tau}^{2p} \AngMD^q \MD^r \left(\mathscr{J}^{-2}\right)\|_{L^2(\mathfrak{M}_{\tau})}^2  
 \lesssim \mathring{M} + (\delta+\varepsilon)\norm(\tau) + \tau P(\norm(\tau)) \notag \\
& \ \ \  \ + \|\g_\tau^{2p+2}\AngMD^q\MD^r \upeta_\mu\|_{L^2(\mathfrak{M}_{\tau})} + \|\mathring{F}\g_\tau^{2p+1}\AngMD^q\MD^r \mathscr{J}\|_{L^2(\mathfrak{M}_{\tau})} 
+ \|\g_\tau^{2p}\AngMD^{q+1}\MD^r \upeta\|_{L^2(\mathfrak{M}_{\tau})} . 
\label{E:MAINELLIPTIC}
\end{align}

\medskip

\noindent
{\bf Step 3. Case $(p,q,r)=(3,0,0)$.}
Let $p=3$ and $(q,r)=0$. We want to estimate the terms appearing on the right-hand side of~\eqref{E:MAINELLIPTIC}. 
First, using Lemma~\ref{L:WEIGHTEDEMBEDDING} and Definition~\ref{D:ENERGY}, 
we obtain the following estimate:
\[
\|\g_\tau^8\upeta_\mu\|_{L^2(\mathfrak{M}_{\tau})} 
\lesssim \|\mathring{F}\g_\tau^8 \upeta_\mu\|_{L^2(\mathfrak{M}_{\tau})} 
+  
\|\mathring{F} \MD  \g_\tau^8 \upeta_\mu\|_{L^2(\mathfrak{M}_{\tau})} 
\lesssim \mathcal{E}.
\]
From similar reasoning, we deduce
$\|\g_\tau^{6}\AngMD \upeta\|_{L^2(\mathfrak{M}_{\tau})} \lesssim \mathcal{E}$.
Finally, the term $\|\g_\tau^{7} \mathscr{J}\|_{L^2(\mathfrak{M}_{\tau})}$ can be bounded by the fundamental theorem of calculus.
Therefore, from~\eqref{E:MAINELLIPTIC}, Prop.~\ref{P:ENERGYLEMMA}, 
and the above estimates, we deduce
\begin{align}\label{E:PIS3APPLICATIONKEYIBPPVCESTIMATE}
\|\mathring{F} \partial_{\tau}^{6} \left(\mathscr{J}^{-2}\right)\|_{L^2(\mathfrak{M}_{\tau})}^2
+ \| \partial_{\tau}^{6} \left(\mathscr{J}^{-2}\right)\|_{L^2(\mathfrak{M}_{\tau})}^2 \leq 
\mathring{M} + (\delta + C \varepsilon)\norm(\tau) + \tau P(\norm(\tau)).
\end{align}
To bound $\|\g_\tau^6\upeta\|_{H^1(\mathfrak{M}_{\tau})}^2$ 
we first split it into its spatial and $0$ component:
\begin{align} \label{E:PARTIAL6TAUUPETASPATIALTIMESPLITTING}
	\| \g_{\tau}^6 \upeta \|_{H^1(\mathfrak{M}_{\tau})}^2
	& \lesssim
	\| \g_{\tau}^6 \underline{\upeta} \|_{H^1(\mathfrak{M}_{\tau})}^2
	+
	\| \g_{\tau}^6 \upeta^0 \|_{H^1(\mathfrak{M}_{\tau})}^2.
\end{align}
From \eqref{E:0COMPONENTINTEMRSOFSPATIALCOMPONENT}, we see that
$\| \g_{\tau}^6 \upeta^0 \|_{H^1(\mathfrak{M}_{\tau})}^2
\lesssim 
\| \g_{\tau}^6 \underline{\upeta} \|_{H^1(\mathfrak{M}_{\tau})}^2
+
\mathring{M} 
	+ \tau P(\norm(\tau))
$.
It therefore suffices to bound
$\| \g_{\tau}^6 \underline{\upeta} \|_{H^1(\mathfrak{M}_{\tau})}^2$.
To this end, we first use Lemma~\ref{L:HODGEEST} with $s=1$ to obtain
\begin{align} 
	\| \partial_{\tau}^6 \underline{\upeta} \|_{H^1(\mathfrak{M}_\tau)} 
	& 
	\lesssim_{\mathring{M}} 
	\| \partial_{\tau}^6 \underline{\upeta} \|_{L^2(\mathfrak{M}_\tau)} 
	+ 
	\|\Flatdiv \partial_{\tau}^6 \underline{\upeta} \|_{L^2(\mathfrak{M}_\tau)} 
	+ 
	\| \Flatvort \partial_{\tau}^6 \underline{\upeta} \|_{L^2(\mathfrak{M}_\tau)} 
		\notag \\
	& \ \ 
	+ 
	\sum_{a=1}^2 \|\AngMD \partial_{\tau}^6 \upeta^a \|_{H^{-1/2}(\partial \mathfrak{M}_{\tau})}.
	\label{E:HODGEESTRECALLED}
\end{align}
We now claim that
\begin{align}	
 \mbox{right-hand side of~\eqref{E:HODGEESTRECALLED}} 
 \leq  \mathring{M} + (\delta + C\varepsilon)\norm(\tau)+ \tau P(\norm(\tau)). \label{E:SIXTIMEDERIVATIVESONUPETAHODGEEST}
\end{align}
To prove \eqref{E:SIXTIMEDERIVATIVESONUPETAHODGEEST}, we
first use Prop.~\ref{P:VORTICITYESTIMATES} to deduce
$\| \Flatvort \partial_{\tau}^6 \underline{\upeta} \|_{L^2(\mathfrak{M}_\tau)} \leq \mathring{M} + \tau P(\norm(\tau))$.
We then use Lemma~\ref{L:ESTIMATESFORDIV} with $(p,q,r)=(3,0,0)$,
Prop.~\ref{P:VORTICITYESTIMATES},
the bound $\|\g_\tau^{6}\AngMD \upeta\|_{L^2(\mathfrak{M}_{\tau})} \lesssim \mathcal{E}$ proved above,
and the energy estimates \eqref{E:ENERGYBOUND}
to deduce
$\|\Flatdiv \partial_{\tau}^6 \underline{\upeta} \|_{L^2(\mathfrak{M}_\tau)}
\leq \mathring{M} + \tau P(\norm(\tau))
$.
Finally, we use Lemma~\ref{L:SIMPLETRACE} 
and the energy estimates \eqref{E:ENERGYBOUND}
to deduce
$\sum_{a=1}^2 \|\AngMD \partial_{\tau}^6 \upeta^a \|_{H^{-1/2}(\partial \mathfrak{M}_{\tau})}
\lesssim \mathring{M} + (\delta+\varepsilon)\norm(\tau)+ \tau P(\norm(\tau))
$.
We have thus shown \eqref{E:SIXTIMEDERIVATIVESONUPETAHODGEEST}
and completed the proof of 
\eqref{E:MAINTRANSVERSALDERIVATIVESESTIMATE} in the case $(p,q,r)=(3,0,0)$.

\medskip 

\noindent
{\bf Step 4. Case $(p,q,r)=(2,1,0)$ and $(p,q,r)=(2,0,1)$.}
These cases rely heavily on the already established bound~\eqref{E:SIXTIMEDERIVATIVESONUPETAHODGEEST}.
Most steps in the proof are the same, so we only explain the handful of
important changes. The first important point is
that we must treat the case $(p,q,r)=(2,1,0)$ before the case $(2,0,1)$.
Using arguments identical to the ones described in {\bf Step 3} and using~\eqref{E:SIXTIMEDERIVATIVESONUPETAHODGEEST}, 
we can obtain the following desired estimates corresponding to the case $(2,1,0)$: 
\begin{align}\label{E:210}
\| \partial_{\tau}^4\AngMD \underline{\upeta} \|_{H^1(\mathfrak{M}_\tau)}  + 
			\|\g_{\tau}^{4}\AngMD \mathscr{J}\|_{L^2(\mathfrak{M}_{\tau})}^2 
		+\|\mathring{F} \g_{\tau}^{4}\AngMD \mathscr{J}\|_{H^{1}(\mathfrak{M}_{\tau})}^2 
		\leq \mathring{M} + (\delta + C\varepsilon)\norm(\tau)+ \tau P(\norm(\tau)).
\end{align}

We now treat the case $(2,0,1)$. Just like in {\bf Step 3}, 
we want to bound the terms on the right-hand side of~\eqref{E:MAINELLIPTIC}.
Specifically, the terms of interest are
\[
\|\g_\tau^{6}\MD \upeta_\mu\|_{L^2(\mathfrak{M}_{\tau})}, \ \   \|\mathring{F}\g_\tau^{5}\MD \mathscr{J}\|_{L^2(\mathfrak{M}_{\tau})}, \ \  
\|\g_\tau^{4}\AngMD\MD \upeta\|_{L^2(\mathfrak{M}_{\tau})}.
\] 
Note that 
$
\|\g_\tau^{6}\MD \upeta_\mu\|_{L^2(\mathfrak{M}_{\tau})}
$
is
bounded via~\eqref{E:HODGEESTRECALLED}-\eqref{E:SIXTIMEDERIVATIVESONUPETAHODGEEST}, 
that $\|\mathring{F}\g_\tau^{5}\MD \mathscr{J}\|_{L^2(\mathfrak{M}_{\tau})}$
can be bounded via the fundamental theorem of calculus,
while 
$
\|\g_\tau^{4}\AngMD\MD \upeta\|_{L^2(\mathfrak{M}_{\tau})}
$ 
can be bounded via~\eqref{E:210}. 
We now run the same argument as in {\bf Step 3} but this time also relying on the already established bounds~\eqref{E:SIXTIMEDERIVATIVESONUPETAHODGEEST}
and~\eqref{E:210}, thereby obtaining the following analog of \eqref{E:210}:
\begin{align}\label{E:201}
\| \partial_{\tau}^4\MD\underline{\upeta} \|_{H^1(\mathfrak{M}_\tau)}  + 
			\|\g_{\tau}^{4}\MD\mathscr{J}\|_{L^2(\mathfrak{M}_{\tau})}^2 
		+\|\mathring{F} \g_{\tau}^{4}\MD \mathscr{J}\|_{H^{1}(\mathfrak{M}_{\tau})}^2 
		\leq \mathring{M} + (\delta + C \varepsilon)\norm(\tau)+ \tau P(\norm(\tau)).
\end{align}
We have therefore proved~\eqref{E:MAINTRANSVERSALDERIVATIVESESTIMATE} in the case $(p,q,r)=(2,0,1)$.

\medskip

{\bf Step 5. The remaining cases.}

Using essentially the same arguments,
we may continue the induction scheme and establish the desired estimates
by considering the remaining cases in the following order:
$(1,2,0)$, $(1,1,1)$, $(1,0,2)$, $(0,3,0)$, $(0,2,1)$, $(0,1,2)$, 
and $(0,0,3)$. This completes our proof of the desired estimate 
\eqref{E:MAINTRANSVERSALDERIVATIVESESTIMATE}.
\end{proof}

In the next corollary,
we show that
$\Lagvort v$
enjoys a bit of extra regularity compared
to the regularity we have already derive.

\begin{corollary}\label{C:LAST}
Let $\delta > 0$ be a constant and let $\varepsilon$ be as in \eqref{E:SMALLDENSITY}.
Under the bootstrap assumptions of Sect.~\ref{SS:BOOTSTRAPASSUMPTIONS}, 
we have the following estimates for $\tau \in [0,T]$:
	\begin{align}
	 \|\Lagvort v(\tau)\|_{H^3(\mathfrak M_{\tau})}^2   + \|\mathring{F}\AngMD^4\Lagvort v(\tau)\|_{L^2(\mathfrak M_{\tau})}^2
	  \leq
	\mathring{M} 
	+ (\delta+ C \varepsilon) \norm(\tau)
	+ \tau P(\norm(\tau)). 
	\end{align}
\end{corollary}

\begin{proof}
Using formula~\eqref{E:ETAVORTICITY} (with the index $A=0$) and considering only the derivative count, 
rather than the precise structure of the identity, we find that
\begin{align}
&D^3\left(\Lagvort v(\tau)\right) \notag \\
& \qquad = \int_0^\tau F \left(D^3\ITIMESMD v\ITIMESMD v + D^3\ITIMESMD S \ITIMESMD v + \ITIMESMD S D^3\ITIMESMD v + D^3\partial_\tau^2 v v\right)\,d\tau' + O_{L^2(\mathfrak M_\tau)}(\mathring M + \tau P(\mathscr S(\tau))),  \label{E:D3}
\end{align}
where $F$ is some function bounded in $L^\infty(\mathfrak M_\tau)$ by the virtue of the fundamental theorem of calculus.
The first term on the right-hand side of \eqref{E:D3} can be bounded by integrating by parts in time: 
\[
\int_0^\tau F D^3\ITIMESMD v\ITIMESMD  v \,d\tau' = F D^3\ITIMESMD\upeta\ITIMESMD v\big|^{\tau}_0 - \int_0^\tau gD^3\ITIMESMD \upeta\ITIMESMD\partial_\tau v \,d\tau' + O_{L^2(\mathfrak M_\tau)}(\mathring M + \tau P(\mathscr S(\tau))).
\]
We now use the Young inequality, the already established bound~\eqref{E:MAINTRANSVERSALDERIVATIVESESTIMATE}, and Sobolev embedding to conclude that 
\[
\Big|\int_0^\tau f D^3\ITIMESMD v \ITIMESMD v \,d\tau'\Big| \leq \mathring{M} 
	+ (\delta+ C \varepsilon) \norm(\tau)
	+ \tau P(\norm(\tau)). 
\]
The remaining error terms are bounded in a similar fashion.

The estimate for $\|\mathring{F}\AngMD^4\Lagvort \, v(\tau) \|_{L^2(\mathfrak M_{\tau})}^2$  
follows from the same argument.
\end{proof}

\section{Proof of the Main Theorem}
\label{S:PROOFOFMAINTHEOREM}

We now use the previous estimates to establish Theorem~\ref{T:MAIN}.
Under the bootstrap assumptions of Sect.~\ref{SS:BOOTSTRAPASSUMPTIONS}, we combine Props.~\ref{P:VORTICITYESTIMATES},~\ref{P:ENERGYLEMMA},~\ref{P:ELLIPTIC}, and Cor.~\ref{C:LAST},
and choose $\varepsilon,\delta$
to be sufficiently small, thereby arriving at the following inequality
(where $\mathring{M} \lesssim_{\delta^{-1}} \norm(0)$):
\[
\norm(\tau) \leq \mathring{M} + \tau P(\norm(\tau)).
\]
By a standard continuity argument, 
it follows that there exists a time $T>0$ such that 
\[
\norm(\tau) \leq 2 \mathring{M}, \ \ \tau \in [0,T].
\]
Using Lemma~\ref{L:IMPROVMENTOFBOOTSTRAP}
and shrinking the size of $T$ if necessary, 
we obtain strict improvements of the bootstrap assumptions.
We have thus proved the theorem.

\section{The equation of state $p(\rho) = \rho^\gamma,$ $\gamma>1$}
\label{S:OTHEREQUATIONSOFSTATE}

In this section, we briefly explain how to adapt our functional framework to the 
equation of state $p(\rho) = \rho^\gamma$ with $\gamma > 1$.
In analogy with~\cite{CoutandShkoller2012}, we define $p_0$ to be the smallest integer satisfying 
\[
1+\frac 1{\gamma-1} - p_0 \leq 2.
\]
We then define the square norm
\begin{align} 
\norm_\gamma(\tau)  
&:= 
\sup_{\tau' \in[0,\tau]}
	\sum_{p=0}^4\| \g_{\tau}^{2p} \upeta \|_{H^{4-p}(\mathfrak{M}_{\tau'})}^2 
		 + 
\sup_{\tau' \in[0,\tau]}
	\sum_{p=0}^3 \|\mathring{F}\g_{\tau}^{ 2p}(\mathscr{J}^{-2}) \|_{H^{4-p}(\mathfrak{M}_{\tau'})}^2 
	\notag  \\
& \ \
	+ 
	\sup_{\tau' \in[0,\tau]}
	\sum_{p=0}^4
	\left[
		\|\mathring{F} \g_{\tau}^{2p} \AngMD^{4-p} D \upeta \|_{L^2(\mathfrak{M}_{\tau'})}^2
		+ 
		\left\|
			\sqrt{\mathring{F}} \g_{\tau}^{2p+1} \AngMD^{4-p} \upeta 
		\right\|_{L^2(\mathfrak{M}_{\tau'})}^2
	\right]
\notag \\
& \ \ + \sum_{p=0}^{p_0}\sup_{\tau' \in[0,\tau]}\|\sqrt{\mathring F}^{1+\frac1{\gamma-1}-p}\g_\tau^{8+p_0-p}D\upeta\|_{L^2(\mathfrak{M}_{\tau'})}^2.
\label{E:NORMGAMMA}
\end{align}
Notice that the number of $\tau$ derivative gets larger as $\gamma>1$ approaches $1$.

In Lagrangian coordinates, problem~\eqref{E:EULER}--\eqref{E:normalizationeulerian} takes the following form:
\begin{align} \label{E:INITIALFUNCTTIONGAMMA}
	f \mathscr{J}
	& = \mathring{F},
		\\
	\mathring{F} \Enth \partial_{\tau} v_{\mu}
	+
	\partial_K
	\left\lbrace
		\mathring{F}^\gamma a^K_{\mu} \mathscr{J}^{-\gamma} \Enth
	\right\rbrace
	+ \mathring{F}^\gamma v_{\mu} \partial_{\tau}\left( \mathscr{J}^{-\gamma}\Enth\right) 
	& = 0,
	 \label{E:LAGRNAGIANMAINEVOLUTIONEQUATIONGAMMA} \\
g_{\alpha\beta}v^{\alpha}v^{\beta} 
& = -1,
\label{E:NORMALIZATIONLAGRANGIANGAMMA}
\end{align}
where 
\begin{align}
	\mathring{F}
	& :=  \mathring{n}  \mathring{v}^0,
		\label{E:MATHRINGFDEFINITIONGAMMA}
\end{align}
and $S=s\circ\upeta,$ with $s$ given by~\eqref{E:ENTHALPYSPECIAL0}.

With respect to the norm~\eqref{E:NORMGAMMA}, 
one can obtain the following analog of Theorem~\ref{T:MAIN}.
\begin{theorem}[A priori estimates in Sobolev spaces when $\gamma>1$]\label{T:MAINGAMMA}
Let the initial particle number density $\mathring n\in H^4(\mathfrak{M})$ satisfy the physical vacuum boundary condition~\eqref{E:PHYSICALVACUUM}
and the condition~\eqref{E:SMALLDENSITY}. 
If $\upeta^\mu$ is a  smooth solution to~\eqref{E:INITIALFUNCTTIONGAMMA}-~\eqref{E:NORMALIZATIONLAGRANGIANGAMMA}, 
then there exists a time $T=T(\norm_\gamma(0))$ and a constant $C$ such that 
\begin{align}
\norm_\gamma(\tau) \le C \norm_\gamma(0), \ \ \tau \in [0,T].
\end{align}
\end{theorem}
The proof of Theorem~\ref{T:MAINGAMMA} follows the same methodology as the proof of Theorem~\ref{T:MAIN}. 
Notice that in this case, 
$\mathring{n}^{\gamma-1}(x)$ behaves like the distance function $d(\cdot,\partial\mathfrak{M})$ in the vicinity 
of the boundary $\partial \mathfrak{M}$. 
The need for a higher number of $\tau$-derivatives when $\gamma<2$ therefore arises from the weighted 
Sobolev embedding~\eqref{E:WEIGHTEDEMBEDDING} when $k<1,$ as it leads to Sobolev spaces with less regularity.

\section*{Acknowledgements} 
S.S. was supported by the National Science Foundation under grant no. DMS-1301380 and by a Royal
Society Wolfson Merit Award.
J.S. was supported by NSF grant no. DMS-1162211,
NSF CAREER grant no. DMS-1454419,
a Sloan Research Fellowship provided by the Alfred P. Sloan foundation,
and from a Solomon Buchsbaum grant administered by the Massachusetts Institute of Technology.


\begin{thebibliography}{99}


\bibitem{dCsK1993}
    	\textsc{Christodoulou, D. and Klainerman, S.:}
      {\it The Global Nonlinear Stability of the {M}inkowski Space},
      Princeton Mathematical Series,
      Vol. 41
      (Princeton University Press, Princeton, 1993).




\bibitem{dC2007}
 \textsc{Christodoulou, D.:}
 {\it The Formation of Shocks in 3-Dimensional Fluids},
EMS Monographs in Mathematics, 2007

 \bibitem{dCsM2012}
      \textsc{Christodoulou, D. and Miao, S.:}
       {\it Compressible Flow and Euler's Equations},
       Surveys of Modern Mathematics,
       Vol. 9
    	 (International Press, Somerville, 2014).     

\bibitem{CoutandHoleShkoller2013}
Well-posedness of the free-boundary compressible 3-D Euler equations with surface tension and the zero surface tension limit.
{\em SIAM J. Math. Anal.}
{\bf 45}, 3690--3767 (2013)

\bibitem{CoLiSh2010}
\textsc{Coutand, D., Lindblad, H., Shkoller, S.:}
A priori estimates for the free-boundary  3D compressible Euler equations in physical vacuum.
{\em Comm. Math. Phys.}
{\bf 296}, no. 2, 559--587 (2010)

\bibitem{CoutandShkoller2011}
\textsc{Coutand, D., Shkoller, S.:}
Well-posedness in smooth function spaces for the moving boundary 1-D compressible Euler equations in physical vacuum.
{\em Comm. Pure Appl. Math.}
{\bf 64},  328--366 (2011)



\bibitem{CoutandShkoller2012}
\textsc{Coutand, D., Shkoller, S.:}
Well-posedness in smooth function spaces for the moving boundary three-dimensional compressible Euler equations in physical vacuum.
{\em Arch. Ration. Mech. Anal.}
{\bf 206}, no. 2, 515--616 (2012)

\bibitem{HaShSp2}
\textsc{Had\v zi\'c, M., Shkoller, S., Speck, J.:}
Well-posedness for the relativistic Euler equations with a free fluid-vacuum boundary.
{\em In preparation}

\bibitem{HsLiMa2004}
\textsc{Hsu, C-H., Lin, S-S., Makino, T}.:
On spherically symmetric solutions of the relativistic Euler equation.
{\em Journal of Differential Equations}
{\bf 201}, 1--24 (2004) 


\bibitem{JangMasmoudi2009}
\textsc{Jang, J., Masmoudi, N.:}
Well-posedness for compressible Euler with physical vacuum singularity. 
{\em Commun. Pure Appl. Math.} 
{\bf 62}, 1327--1385 (2009)


\bibitem{JangMasmoudi2015}
\textsc{Jang, J., Masmoudi, N.}
Well-posedness of compressible Euler equations in a physical vacuum.
{\em Communications on Pure and Applied Mathematics}
{\bf 68}, n0. 1, 61--111 (2015)

\bibitem{JangLeflochMasmoudi2015}
\textsc{Jang, J., LeFloch, P. G., Masmoudi, N.:}
Lagrangian formulation and a priori estimates for relativistic fluid flows with vacuum.
{\em Preprint, available online at http://arxiv.org/abs/1511.02366}

\bibitem{sKiR2003}
      \textsc{Klainerman, S. and Rodnianski, I.:}
      Improved local well-posedness for quasilinear wave equations in dimension three.
      {\em Duke Math. J.}  
      {\bf 117}
      (2003)
      1--124.      

\bibitem{sKiRjS2012}
\textsc{Klainerman, S.,  Rodnianski, I., and Szeftel, J.:}
 The bounded $L^2$ curvature conjecture,
 To appear in {\it Invent. Math.}
 

\bibitem{Kufner}
\textsc{Kufner, A.:}
{\em Weighted Sobolev Spaces.}
Wiley-Interscience, New York (1985)


\bibitem{Lin1987}
\textsc{Lin, L.W.:} 
On the vacuum state for the equations of isentropic gas dynamics. 
{\em J. Math. Anal. Appl.} 
{\bf 121}, 406--425 (1987)

\bibitem{Lindblad2005}
\textsc{Lindblad, H.:}
Well-posedness for the motion of a compressible liquid with free surface boundary. 
{\em Commun. Math. Phys.}
{\bf 260}, 319--392 (2005)

\bibitem{LiuSmoller1980}
\textsc{Liu, T.-P., Smoller, J.:} 
On the vacuum state for isentropic gas dynamics equations.{\em Adv. Math.} 
{\bf 1}, 345--359 (1980)


\bibitem{Makino1986}
\textsc{Makino, T.:} 
On a local existence theorem for the evolution equation of gaseous stars.{\em Patterns and Waves. Stud. Math. Appl.}, {\bf 18} 
North-Holland, Amsterdam, 459Ð479,1986


\bibitem{sMpY2014}
\textsc{Miao, S. and Yu, P.:}
On the formation of shocks for quasilinear wave equations.
arXiv preprint (2014).


\bibitem{Oliynyk2012}
\textsc{Oliynyk, T.:}
On the existence of solutions to the relativistic Euler equations in 2 dimensions with a vacuum boundary.
{\em Class. Quantum Grav.}
{\bf 29}, 155013 (2012)

\bibitem{Oliynyk2015} 
\textsc{Oliynyk, T.:}
A priori estimates for relativistic liquid bodies. 
{\em Preprint, available online at arXiv:1501.00045}.


\bibitem{Rendall1992}
\textsc{Rendall, A. D.:} 
The initial value problem for a class of general relativistic fluid bodies. 
{\em J. Math.Phys.} 
{\bf 33}, no. 3, 1047--1053 (1992)


\bibitem{jS2014}
\textsc{Speck, J.:}
{\em Shock Formation in Small-Data Solutions to $3D$ Quasilinear Wave Equations.}
arXiv preprint (2014)


\bibitem{Sy1937}
\textsc{Synge, J. L.:} Relativistic hydrodynamics.
{\em Proceedings of the London Mathematical Society sec 2}
{\bf 43}, 376--416 (1937)

\bibitem{Sy1957}
\textsc{J. L. Synge:}
{\em The Relativistic Gas.}
Series in Physics. Monographs. 
North-Holland Pub. Co., 1957


\bibitem{Trakhinin2009}
\textsc{Trakhinin, Y.:} 
Local existence for the free boundary problem for the non-relativistic and relativistic compressible Euler equations with a vacuum boundary condition.{\em Commun. Pure Appl. Math.} 
{\bf 62}, 1551--1594 (2009)


\bibitem{Weinberg2009}
\textsc{S. Weinberg:} 
{\em Cosmology}, 
Oxford University Press, Oxford, 2008. MR 2410479


\bibitem{ZeNo}
\textsc{Ya. B. Zel'dovich, I. D. Novikov:}
Relativistic Astrophysics.


 

 



\end{thebibliography}
\end{document}